\documentclass[11pt,letterpaper]{amsart}
\usepackage{graphicx, amssymb, amsmath, amsthm}
\usepackage[utf8]{inputenc}
\usepackage{epsfig}
\usepackage{hyperref}
\usepackage{comment}
\usepackage{mathrsfs}
\numberwithin{equation}{section}
\usepackage{subfigure}
\usepackage{tikz}
\usepackage{here}
\usepackage{float}
\usepackage{pifont}
\usepackage{enumitem}
\usetikzlibrary{matrix,arrows}
\usetikzlibrary{shapes}
\usetikzlibrary{calc}
\usetikzlibrary{arrows}
\usetikzlibrary{decorations.pathreplacing,decorations.markings}
\usepackage[all]{xy}

\usepackage{tikz-cd}

\usetikzlibrary{patterns}

\newtheorem{theorem}{Theorem}[section]

\newtheorem{lemma}[theorem]{Lemma}

\newtheorem{corollary}[theorem]{Corollary}
\newtheorem{conjecture}[theorem]{Conjecture}
\newtheorem{proposition}[theorem]{Proposition}
\newtheorem{question}[theorem]{Question}

\theoremstyle{definition}
\newtheorem{definition}[theorem]{Definition}
\newtheorem{remark}[theorem]{Remark}

\makeatletter
\newcommand{\Extend}[5]{\ext@arrow0099{\arrowfill@#1#2#3}{#4}{#5}}
\makeatother

\DeclareMathOperator{\Lip}{Lip}
\DeclareMathOperator{\dist}{dist}
\DeclareMathOperator{\diam}{diam}

\newcommand{\ve}{\varepsilon}
\newcommand{\ti}[1]{\tilde{#1}}

\begin{document}

\title{Positive scalar curvature metrics and aspherical summands}

\author[S. Chen]{Shuli Chen}
\address[Shuli Chen]{Department of Mathematics, Stanford University, 450 Jane Stanford Way, Bldg 380, Stanford, CA 94305, United States 
\newline
\indent
{Current Address: Department of Mathematics, University of Chicago, 5734 S University Ave, Chicago IL, 60637, United States}}
\email{shulichen@uchicago.edu}

\author[J. Chu]{Jianchun Chu}
\address[Jianchun Chu]{School of Mathematical Sciences, Peking University, Yiheyuan Road 5, Beijing 100871, People's Republic of China}
\email{jianchunchu@math.pku.edu.cn}

\author[J. Zhu]{Jintian Zhu}
\address[Jintian Zhu]{Institute for Theoretical Sciences, Westlake University, 600 Dunyu Road, Hangzhou, Zhejiang 310030, People's Republic of China}
\email{zhujintian@westlake.edu.cn}

\renewcommand{\subjclassname}{\textup{2020} Mathematics Subject Classification}
\subjclass[2020]{Primary 53C21, 53A10}

\begin{abstract}
We prove for $n\in\{3,4,5\}$ that the connected sum of a closed aspherical $n$-manifold with an arbitrary \emph{non-compact} manifold does not admit a complete metric with nonnegative scalar curvature. In particular, a special case of our result answers a question of Gromov.

More generally, we generalize the partial classification result of Chodosh, Li, and Liokumovich to the non-compact domination case with our newly-developed technique.

Our result unifies all previous results of this type, and confirms the validity of Gromov's non-compact domination conjecture for closed aspherical manifolds of dimensions 3, 4, and 5.

\end{abstract}

\maketitle

\section{Introduction}\label{Introduction}

In differential geometry, the scalar curvature, as
a certain average of the sectional curvatures along all tangential planes, is the weakest curvature invariant of a Riemannian metric.
The study of the relationship between the scalar curvature and topology on closed manifolds has a very long history (see \cite{Lic63, SY79a, GL80a, GL80b, SY87,Sch89,LM89,Gro96,Gro19,SY22} and references therein).

For closed surfaces the Gauss--Bonnet formula relates curvature and topology, which in particular yields that there is no smooth metric with positive Gaussian curvature on a torus.
Towards a generalization in higher dimensions,
the well-known Geroch conjecture asserts that the $n$-torus $\mathbb{T}^{n}$  cannot admit any smooth metric of positive scalar curvature, which was verified by Schoen and Yau \cite{SY79a} for $3\leq n\leq 7$ using minimal hypersurfaces, and by Gromov and Lawson \cite{GL80a} for all dimensions using spinors. Recently, Stern \cite{Ste22} also gave a new proof in dimension three using harmonic maps. It is well-known that the Geroch conjecture as well as its generalizations has had several important consequences in geometry and mathematical physics, including Schoen--Yau's proof of the Riemannian positive mass theorem in general relativity \cite{SY79b,Sch89,SY22} and Schoen's resolution of the Yamabe problem \cite{Sch84}.

A natural extension of the Geroch conjecture arises from the consideration for aspherical manifolds, which appear to be an important object in many branches of mathematics (e.g. known as $K(\pi,1)$-space or Eilenberg--MacLane space in topology). Recall that a manifold is called \emph{aspherical} if it has contractible universal cover (or equivalently, if the homotopy group $\pi_i$ vanishes for all $i\geq 2$). Concerning the geometry it was conjectured that all closed aspherical $n$-manifolds do not admit any metric of positive scalar curvature (see \cite{SY87,Gro86}). Up to now, a few progresses have been made towards this conjecture. The three dimensional case was verified by Gromov and Lawson \cite{GL83}. We also note that Schoen and Yau \cite{SY79c} previously obtained related classification results for 3-manifolds with positive scalar curvature.
In dimension four, the special case with non-zero first Betti number was confirmed by Wang \cite{Wang19}, and the general case was proven by Chodosh and Li \cite{CL20} (also see a previous outline from Schoen and Yau \cite{SY87}). In dimension five,  Chodosh and Li \cite{CL20} and Gromov \cite{Gro20} independently verified this conjecture based on recent development of Gromov's $\mu$-bubble method \cite{Gro19}. However, the aspherical conjecture is still widely open in dimensions greater than five.

Motivated by the Liouville theorem from conformal geometry, there have been many efforts to generalize known topological obstructions for positive scalar curvature on closed manifolds to those on non-compact complete manifolds. For example, Lesourd, Unger and Yau \cite{LUY20} reduced the Liouville theorem of locally conformally flat manifolds to an affirmative answer to the generalized Geroch conjecture: $\mathbb{T}^{n}\# X^{n}$ does not admit a complete metric with positive scalar curvature for any non-compact $n$-manifold $X^{n}$. (Actually they just need the generalized Geroch conjecture to hold when $n\leq 6$.) Later, Chodosh and Li \cite{CL20} verified this conjecture in dimensions no greater than seven and thereby completed the program for the Liouville theorem (also see the contribution of Lesourd--Unger--Yau \cite{LUY20} in dimension three). It is also worth mentioning that Wang and Zhang \cite{WZ22} verified the generalized Geroch conjecture in all dimensions with an additional spin assumption on the non-compact manifold $X$.

The obstruction for positive scalar curvature on non-compact connected sums was later generalized to more general cases. For example, Schoen--Yau--Schick (SYS) manifolds (for explicit definition, see e.g. \cite[Section 5]{Gro18}) provide {important generalizations of the} $n$-torus. They were first considered by Schoen and Yau  \cite{SY79a} to introduce the dimension-descent argument and later by Schick \cite{Schi98} to construct a counterexample to the unstable Gromov--Lawson--Rosenberg conjecture. As a corresponding generalization for the generalized Geroch conjecture, Lesourd, Unger and Yau \cite{LUY20} proved that the connected sum $M^{n}\# X$ cannot admit any complete metric with positive scalar curvature when $M$ is an $n$-dimensional SYS manifold with $n=3$ or $4\leq n\leq7$ along with some additional technical assumptions.
Later, Chodosh and Li \cite{CL20} dealt with a special class of SYS manifolds and obtained the same conclusion when $3\leq n\leq 7$. As a unification of previous results,
the first-named author \cite{Chen22} finally established the obstruction for positive scalar curvature on non-compact connected sums with any closed SYS manifold when $3\leq n\leq 7$.

In order to obtain the above-mentioned results, one relies crucially on the abundance of nonzero homology classes for carrying out Schoen--Yau's dimension-descent argument, and Gromov's $\mu$-bubble method can then be applied in a direct way to overcome the issue of non-compactness. However, this is not the case when the underlying manifold is only aspherical. For example, in dimension four, there exist infinitely many closed aspherical 4-manifolds that are homology 4-spheres \cite{RT05}.  Even in cases simpler than non-compact connected sum, not many results are known for non-compact manifolds constructed from closed aspherical ones and in particular
Gromov \cite[p. 151]{Gro19} proposed the following start-up question:
\begin{question}
Are there complete metrics with positive scalar curvature on closed aspherical manifolds with punctures of dimension $4$ and $5$?
\end{question}
Concerning this question we consider the more general problem whether there is any complete metric with positive scalar curvature on connected sums of non-compact manifolds with closed aspherical manifolds. It turns out that we are able to prove the following

\begin{theorem}\label{Thm: main}
Let $N^n$, $n\in\{3,4,5\}$, be a closed aspherical manifold and let $X^n$ be an arbitrary $n$-manifold. Then the connected sum $Y= N \# X$ admits no complete metric with positive scalar curvature.
\end{theorem}

\begin{remark}\label{Rem: orientable}
Recall that the closed case is already known in \cite{CLL23}, so we just need to handle the case when $X$ is non-compact. By lifting, $N\#X$ has a finite cover which is a (possibly multiple-necked) connected sum of some orientable covers of $N$ and $X$. Then in the proof of Theorem \ref{Thm: main}, we assume without loss of generality that $N$ and $X$ are both orientable.
\end{remark}

As a corollary,
the answer to Gromov's question is \emph{no} since any closed aspherical manifold with punctures is just the connected sum of a closed aspherical manifold and an $n$-sphere with punctures.

Beyond the obstruction from non-compact connected sums Gromov \cite[p. 252]{Gro19} also made a non-compact domination conjecture characterizing the obstruction for positive scalar curvature on non-compact manifolds having certain domination property. To be more precise he proposed the following definition and conjecture:

\begin{definition}[Quasi-proper map {\cite[p. 18]{Gro19}}]\label{Defn: quasi-proper}
Let $M^{n}$ and $N^{n}$ be orientable $n$-manifolds. A continuous map $f:M \to N$ is said to be quasi-proper if for all proper maps $\phi:\mathbb{R}_{+}\to M$, the composed map $f\circ\phi:\mathbb{R}_{+}\to N$ is either proper or converges to a point in {$N$} for $t\to\infty$.
\end{definition}

\begin{conjecture}[Non-compact domination conjecture]\label{Conj: domination}
    If a closed orientable $n$-manifold $N$ cannot be dominated by closed manifolds with positive scalar curvature, then it also cannot be dominated by complete manifolds with positive scalar curvature. Here $N$ is dominated by $M$ means that there is a quasi-proper map from $M$ to $N$ with non-zero degree.
\end{conjecture}
As a further attempt, we extend the underlying compact-to-complete principle behind Gromov's non-compact domination conjecture to the partial classification results for closed manifolds with positive scalar curvature established in \cite{GL83,CLL23}. Recall
\begin{theorem}
[\cite{GL83,CLL23}]\label{Thm: compact partial classification}
    Let $N^n$, $n\in\{3,4,5\}$, be a closed orientable $n$-manifold, whose universal cover is $(n-2)$-connected. Suppose there is a closed orientable Riemannian manifold $(M^n,g)$ with positive scalar curvature admitting a non-zero degree map $f:M\to N$. Then a finite cover of $N$ is homotopy equivalent to $\mathbb{S}^n$ or connected sums
of $\mathbb{S}^{n-1} \times \mathbb{S}^1$.
\end{theorem}

Since we have to deal with non-compact manifolds in our extension, let us define what it means for a quasi-proper map between non-compact manifolds to have non-zero degree throughout this paper.
\begin{definition}\label{Defn: degree}
	Let $M^{n}$ and $N^{n}$ be orientable $n$-manifolds (possibly non-compact). A quasi-proper map $f:M\to N$ is said to have non-zero degree if
	\begin{itemize}\setlength{\itemsep}{1mm}
	
	\item $S_\infty$ consists of discrete points, where
 $$S_\infty=\bigcap_{K\subset M \text{ compact}} \overline{f(M-K)};$$
	\item the composed map
	$$H_n^{\mathrm{lf}}(M)\xrightarrow{i_*} H_n^{\mathrm{lf}}(f^{-1}(N-S_\infty))\xrightarrow{f_*} H_n^{\mathrm{lf}}(N-S_\infty)$$
	is non-zero, where $H_{\ast}^{\mathrm{lf}}(\cdot)$ are the locally-finite singular homology groups {with $\mathbb{Z}$ coefficients} and $i_{*}$ is the restriction map (see e.g. \cite[Lecture 6, Section 1.3]{Mor18}).
\end{itemize}
\end{definition}
After removing the compactness assumption on $M$ in Theorem \ref{Thm: compact partial classification} we can show the following
\begin{theorem}\label{Thm: partial classification}
    Let $N^n$, $n\in\{3,4,5\}$, be a closed orientable $n$-manifold, whose universal cover is $(n-2)$-connected. Suppose there is a complete orientable Riemannian manifold $(M^n,g)$ with positive scalar curvature admitting a non-zero degree map $f:M\to N$. Then a finite cover of $N$ is homotopy equivalent to $\mathbb{S}^n$ or connected sums
of $\mathbb{S}^{n-1} \times \mathbb{S}^1$.
\end{theorem}
In particular, we partially solve Gromov's non-compact domination conjecture.
\begin{corollary}
Conjecture \ref{Conj: domination} holds for closed aspherical $n$-manifolds $N$ for $n=3,4,5$.
\end{corollary}

\begin{remark}
When $n=3$, we can further conclude that the manifold $N$ in Theorem \ref{Thm: partial classification} has no aspherical factor in its prime decomposition. This is claimed by Gromov \cite[p. 140]{Gro19} without details.
\end{remark}

Moreover, if $(M,g)$ is only assumed to have nonnegative scalar curvature, we can establish the following rigidity result.
\begin{theorem}\label{Thm: mapping rigidity}
Let $N^n$, $n\in\{3,4,5\}$, be a closed orientable $n$-manifold, whose universal cover is $(n-2)$-connected. Suppose there is a complete orientable Riemannian manifold $(M^n,g)$ with nonnegative scalar curvature admitting a non-zero degree map $f:M\to N$. Then
\begin{itemize}\setlength{\itemsep}{1mm}
    \item either a finite cover of $N$ is homotopy equivalent to $\mathbb{S}^n$ or connected sums of $\mathbb{S}^{n-1} \times \mathbb{S}^1$;
    \item or $M$ and $N$ are both closed aspherical, and the metric on $M$ is flat.
\end{itemize}
\end{theorem}

\subsection{Outline of the proof}
\subsubsection{On the proof of Theorem \ref{Thm: main}}
Our proof of Theorem \ref{Thm: main} follows exactly the same strategy as when dealing with the aspherical conjecture, which we shall recall first.

As we mentioned before, the main difficulty in dealing with aspherical manifolds is the lack of homological information (recall that the definition is purely homotopic). To create certain non-trivial homology class the strategy is a chain-closing program first raised in \cite{SY87} and further developed in \cite{CL20}. Let $(N,g)$ be a closed aspherical Riemannian manifold and $(\tilde N,\tilde g)$ be its universal cover. In particular, $\tilde N$ is non-compact and one can therefore find a geodesic line $\tilde\sigma:(-\infty,+\infty)\to (\tilde N,\tilde g)$. Notice that the boundary of a tubular neighborhood of the ray $\tilde\sigma|_{[0,+\infty)}$ just provides a locally-finite chain having non-zero intersection number with $\tilde\sigma$. If one can close this chain without creating further intersections, then the newly obtained cycle represents a non-trivial homology class and this leads to a contradiction to the contractibility of $\tilde N$.

In practice, the main idea for chain-closing is by cutting and pasting. Intuitively the boundary of a tubular neighborhood of the ray $\tilde\sigma|_{[0,+\infty)}$ has to be an infinitely long cylinder at infinity. Then one can just cut the cylinder at a finite length and then obtain a chain with the boundary from cutting. By taking both the tubular neighborhood and the cutting length large enough, the chain boundary can be guaranteed to be far away from the line $\tilde\sigma$. Now one wants to glue another relatively small chain disjoint from the line $\tilde\sigma$ along the chain boundary so that the desired homologically non-trivial cycle is obtained. The search of the relatively small chain can be interpreted as a quantitative filling problem with prescribed chain boundary, where the \emph{uniformly} positive scalar curvature comes into play.
After applying Gromov's $\mu$-bubble method to the minimizing hypersurface obtained from solving the Plateau problem with prescribed chain boundary, one can reduce the original filling problem into a new one where the prescribed boundary becomes a codimension-two closed submanifold with its stabilized positive scalar curvature (see Definition \ref{Defn: stabilized curvature}) bounded from below by a positive constant depending on the positive infimum $\inf R(g)$ of the scalar curvature.

As a cover of a closed manifold, the universal cover $(\tilde N,\tilde g)$ satisfies the following uniform filling property: there is a positive function $F=F(r)$ depending on $(N,g)$ such that if a boundary chain $C$ is supported in some geodesic ball $B_r(p)$ then we can find a chain $\Gamma$ supported in a larger geodesic ball $B_{F(r)}(p)$ with the same center such that $\partial\Gamma=C$. From this fact the key to solve the new quantitative filling problem is to break the codimension-two submanifold into small boundary chains with uniformly bounded diameters, which was finally carried out by Chodosh and Li through their slice-and-dice argument (when the underlying manifold $N$ has dimension five).

In our attempt to prove Theorem \ref{Thm: main} with the same strategy, the main difficulty is the lack of uniformly positive scalar curvature on non-compact manifolds, which was the key to solve the quantitative filling problem through the slice-and-dice argument. The trouble appears around infinity, where the lower bound of scalar curvature decays to zero and consequently the blocks coming from the slice-and-dice argument have no uniform control on their diameters.
To handle this issue a natural idea is to use relative homology theory, where the blocks around infinity without control can be eliminated in a topological way. Let us give some details on this point. Recall that we are dealing with the non-compact connected sum $(Y=N\#X, g)$, where $N$ is closed and aspherical. Since $N$ has contractible universal cover $\tilde N$, we consider the corresponding cover $\tilde Y=\tilde N\#_{\pi_1(N)}X$ of $Y$ in our case. As stated above, the scalar curvature could decay to zero on these $X$-copies inside $\tilde Y$. So we make the decomposition $$
\tilde Y=\tilde N_{\ve}\cup\left(\bigcup_i \tilde X_{\ve,i}\right),$$
where
$$
\tilde N_\ve=\tilde N-\left(\bigcup_i B_i\right)\mbox{ and }\tilde X_{\ve,i}\approx X_\ve:=X-B.
$$
{Here the subscript $\ve$ is a fixed small positive constant (we refer the reader to Subsection \ref{subsec: set-up and notations} for the precise definitions of the above notations).} Then we consider the pair $(\tilde Y,\bigcup_i\tilde X_{\ve,i})$, which can be verified to satisfy very similar homological properties as a contractible manifold (see Lemma \ref{Lem: excision}).

As planned we set a similar chain-closing program as before except that relative cycles are considered now. Similar as before, we construct a proper line $\tilde\sigma:(-\infty,+\infty) \to \tilde N_\ve$ by lifting (see Lemma \ref{geodesic}) and take the boundary of a tubular neighborhood of $\tilde \sigma|_{[0,+\infty)}$. Through cutting at a finite length we can obtain a hypersurface $\tilde M_{n-1}$ with boundary. Let $M_{n-1}$ be the minimizing hypersurface obtained from solving the Plateau problem with prescribed boundary $\partial\tilde M_{n-1}$. Then using $\mu$-bubble method we can construct a codimension-two closed submanifold $M_{n-2} \subset M_{n-1}$. By adjusting the size of the tubular neighborhood as well as the cutting length we can guarantee
\begin{itemize}\setlength{\itemsep}{1mm}
\item $\dist(M_{n-2},\tilde \sigma)>L$ for arbitrarily large $L>0$;
\item $M_{n-2}$ has stabilized scalar curvature no less than $R(g)-\mu_{\mathrm{loss}}$ for arbitrarily small $\mu_{\mathrm{loss}}>0$.
\end{itemize}
The \emph{essential} difference here is that no matter how small the constant $\mu_{\mathrm{loss}}$ is, the stabilized scalar curvature of $M_{n-2}$ can be negative somewhere since the scalar curvature $R(g)$ does not have a positive lower bound on $Y$. Nevertheless, we still hope to solve the following relative quantitative filling problem for $M_{n-2}$: find a chain $\Gamma$ disjoint from $\tilde\sigma$ such that $\partial\Gamma=M_{n-2}$ modulo $\bigcup_i\tilde X_{\ve,i}$.

Notice that the core region $\tilde N_\ve$, as the cover of a compact region of $Y$, satisfies similar quantitative filling property (see Lemma \ref{Lem: quantitative filling}). In order to solve the quantitative filling problem we can take the same idea to break $M_{n-2}$ into small relative boundary chains based on the slice-and-dice argument, but the situation is different now. Although the use of relative homology theory appears to be very effective to handle the issue caused by the lack of uniform diameter control for the blocks around infinity of $\tilde X_{\ve,i}$, we still encounter great difficulty in obtaining diameter bounds for the remaining blocks that stay away from the infinity of $\tilde X_{\ve,i}$. The challenge comes from the fact that only inradius estimate of a compact surface with boundary could be derived from its stabilized positive scalar curvature (see Lemma \ref{Lem: inradius}). In the case of solving the aspherical conjecture, it suffices to deal with closed surfaces, where we can obtain diameter bounds by removing small disks and estimate the inradii. While in our case we have to deal with compact surfaces with boundary (especially because we need to handle the submanifold $M_{n-2}$ with negative-scalar-curvature regions removed), whose diameters unfortunately cannot be controlled by their inradii. For example, a cylinder with disks periodically removed  has finite inradius but its diameter could be arbitrarily large. This difficulty requires us to introduce some new ideas and make a necessary refinement of Chodosh--Li's slice-and-dice argument.

Before a detailed description of our refinement, let us first describe Chodosh--Li's slice-and-dice argument from \cite{CL20}. Recall that the chain-closing program when $N$ has dimension five was eventually reduced to a quantitative filling problem for a closed $3$-manifold $M_3$ with stabilized uniformly positive scalar curvature. Two steps, slicing and dicing, are designed in \cite{CL20} to break $M_3$ into small blocks. Since the diameter estimate from stabilized positive scalar curvature can only be established for surfaces (but not for the $3$-dimensional blocks), one needs the slicing step to reduce the topological complexity so as to obtain diameter estimates of the blocks using diameter estimates of their boundary surfaces.
In practice, one constructs finitely many closed minimizing surfaces $S_k$ (called slicing surface) from non-trivial classes in $H_2(M_3)$ and slices them off such that the resulting compact $3$-manifold, denoted by $\hat M_3$, satisfies the property that the map $H_2(\partial\hat M_3) \to H_2(\hat M_3)$ is surjective. As in the dicing step, $\hat M_3$ is diced into finitely many blocks with bounded thickness along finitely many $\mu$-bubbles $D_l$ (called dicing surface). The previous slicing step guarantees that for each block the thickness is measured from exactly one dicing surface, for which the diameter of the blocks can be derived from that of the slicing surfaces (obtained from stabilized uniformly positive scalar curvature) as well as the bounded thickness.

Back to our case, recall from previous paragraphs that when $n=5$ we have to deal with a closed $3$-manifold $M_3$ with \emph{islands} (short for those regions where the stabilized scalar curvature is negative). After repeating the slice-and-dice argument we can also obtain slicing surfaces $S_k$ and dicing surfaces $D_l$ with islands. Let us just restate the difficulty here: only inradius estimate is available for these surfaces with islands removed, which gives no control on diameter if the intrinsic distances between islands are very large. Our key observation, roughly speaking, is that: if all the islands of a slicing or dicing surface lie in exactly one end $\tilde X_{\ve,i_0}$, then we can still obtain an \emph{extrinsic-diameter} bound of this surface (modulo $\bigcup_i\tilde X_{\ve,i}$). Therefore, we need to establish avoidance criteria to prevent slicing and dicing surfaces from touching two different ends.

The avoidance criteria are closely related to the fact that if a complete surface has two ends, then it cannot admit any metric with stabilized positive scalar curvature. As a quantitative version we can show that if a surface has a core region with stabilized scalar curvature bounded from below by a positive constant $\sigma$, then this core region cannot extend along two different directions with stabilized nonnegative scalar curvature beyond distance $L=L(\sigma)$. Then our trick is to push the end $X_\ve$ much closer to infinity such that the ends $\tilde X_{\ve,i}$ are at a far enough distance away from each other. In practice, we set an appropriate hypersurface $\Lambda$ separating $\partial X_\ve$ and the infinity, and take the new end  $X_\Lambda$ to be the region outside $\Lambda$ (correspondingly the end $\tilde X_{\ve,i}$ is adjusted to $\tilde X_{\Lambda,i}$ with boundary $\tilde \Lambda_{,i}$).
As illustrated in Figure \ref{Fig: avoidance-criterion}, we can set the (avoidance) hypersurface $\Lambda_2$ sufficiently deep (i.e. $d_0$ large enough) in $X_\ve$ such that the slicing surfaces $S_k$ and dicing surfaces $D_l$ cannot touch two ends at the same time (see also Proposition \ref{Prop: closed}). Besides a single slicing surface or dicing surface we also need to derive an extrinsic-diameter bound for connected components of the slice-and-dice trace (i.e. the union of all $S_k$ and $D_l$) and then for the $3$-dimensional blocks (modulo $\bigcup_i\tilde X_{\ve,i}$), where more involved avoidance criteria have to be established for the combinations of several slicing and dicing surfaces
 (see Proposition \ref{Prop: compact} and \ref{Prop: slice and dice}).
Based on our new avoidance criteria we can finally obtain the desired diameter estimates (modulo $\bigcup_i\tilde X_{\ve,i}$) and solve the relative quantitative filling problem for $M_3$.

\begin{figure}[htbp]
\centering
\includegraphics[width=10cm]{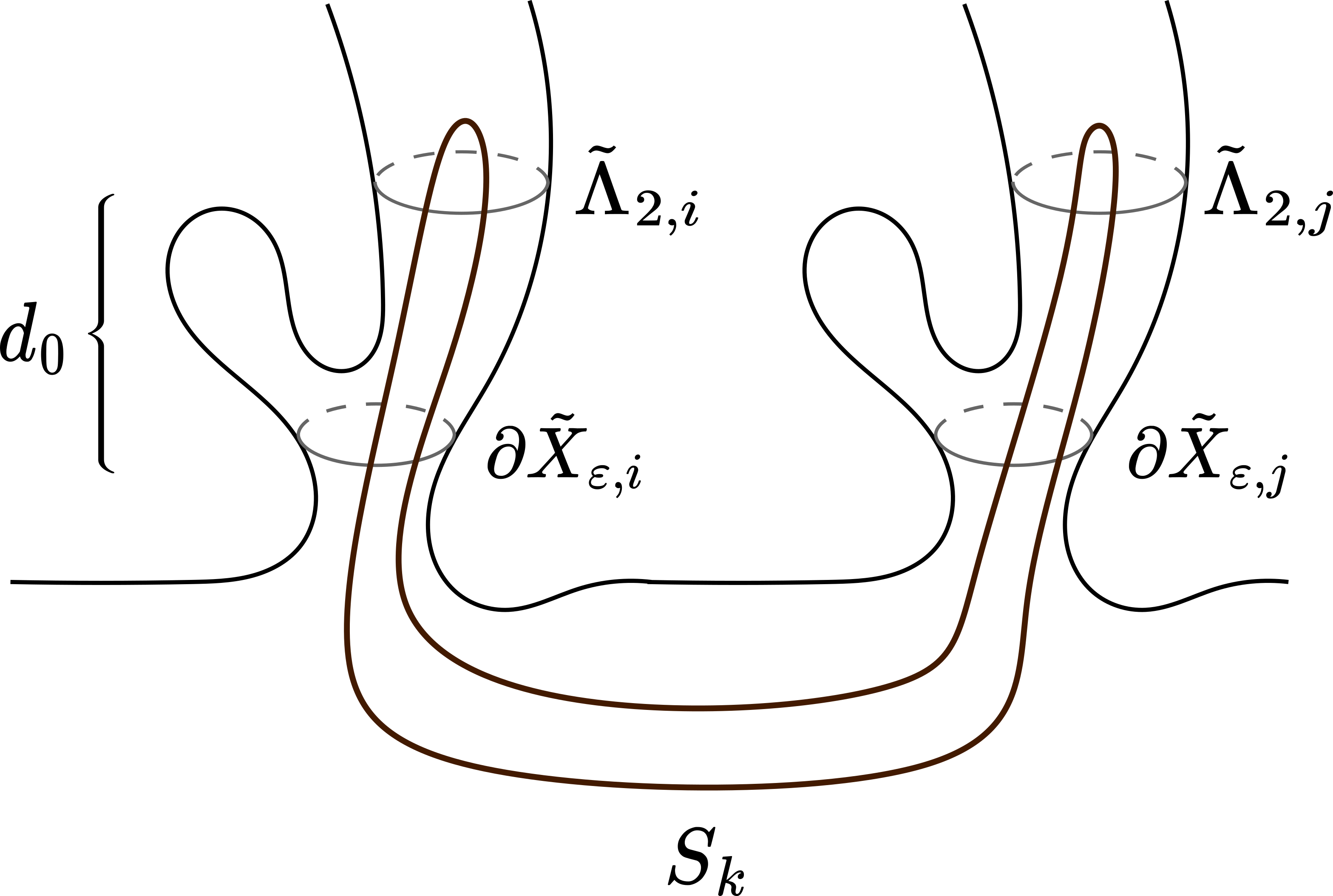}
\caption{The slicing surface $S_k$ cannot touch two ends $\tilde X_{\Lambda_2,i}$ and $\tilde X_{\Lambda_2,j}$ when $d_0$ is large enough, so the situation in the figure will not happen.}
\label{Fig: avoidance-criterion}
\end{figure}

\subsubsection{On the proof of Theorem \ref{Thm: partial classification} and \ref{Thm: mapping rigidity}}
Next let us sketch the proof for Theorem \ref{Thm: partial classification}. First we recall the quantitative-topology description from \cite{CLL23}. Suppose a closed Riemannian $n$-manifold $N$ satisfies the following quantitative filling property: in the universal cover any closed embedded codimension-two submanifold can be filled in its $L$-neighborhood (up to multiplicity) for some universal constant $L$ independent of the choice of the submanifold. Then a finite cover of this closed manifold $N$ is homotopy equivalent to either $\mathbb{S}^n$ or connected sums of $\mathbb{S}^{n-1} \times \mathbb{S}^1$.

Given this fact our goal is to verify the quantitative filling property for a closed manifold $N$ when it is dominated by a complete Riemannian manifold $M$ with positive scalar curvature. Let us still focus on the most complicated case when $n=5$. The key point is that our modified slice-and-dice argument actually yields the desired filling property. To see this we start with an arbitrary $3$-manifold $\Sigma$ in the universal cover $\tilde N$. To make use of positive scalar curvature we just pull $\Sigma$ back to $\hat\Sigma$ on some cover $\hat M$ of $M$. Based on our modified slice-and-dice argument we can decompose the pull-back $\hat\Sigma$ into blocks with uniformly bounded size and blocks around infinity. By Definition \ref{Defn: degree} blocks around infinity do not make any contribution to the filling since they are all mapped to small balls. On the other hand, the images of the blocks with uniformly bounded size can be filled quantitatively using the $(n-2)$-connectedness of $\tilde N$, which gives the desired quantitative filling property.

The proof of Theorem \ref{Thm: mapping rigidity} is based on a delicate analysis of the topology of closed Ricci-flat manifolds (see Proposition \ref{Prop: mapping rigidity general n}).

\bigskip

The organization of this paper is as follows. In Section \ref{Preliminaries}, we introduce set-up and notations, and collect some preliminary results that will be used throughout this work. In Section \ref{Proof of main theorem}, we prove our main result, Theorem \ref{Thm: main}. In Section \ref{Sec: compact-to-complete}, we prove
Theorem \ref{Thm: partial classification} and \ref{Thm: mapping rigidity}. In Appendix \ref{construction of rho}, we give the explicit construction of some function which will be used in the proof of Proposition \ref{Prop: compact}.

\bigskip

{\bf Acknowledgments. }We thank Professor Otis Chodosh for bringing this question to our attention and for his many useful communications, and Professor Lu Wang for her interest and helpful suggestions. {We also thank the referees for their helpful comments. 
The first-named author was partially supported by the National Science Foundation under Grant No. DMS-1928930, while in residence at the Simons Laufer Mathematical Sciences Institute (formerly MSRI) during the Fall 2024 Semester. The second-named author was partially supported by National Key R\&D Program of China 2024YFA1014800 and 2023YFA1009900, NSFC grants 12471052 and 12271008, and the Fundamental Research Funds for the Central Universities, Peking University. The third-named author was partially supported by National Key R\&D Program of China 2023YFA1009900, NSFC grant 12401072, and the start-up fund from Westlake University.}

\section{Preliminaries}\label{Preliminaries}
\subsection{Set-up and notations}\label{subsec: set-up and notations}
As mentioned in Remark \ref{Rem: orientable}, we let $N$ be a smooth closed orientable aspherical $n$-dimensional manifold, and we let $X$ be a smooth non-compact orientable manifold.
Let $\ti N$ denote the universal cover of $N$ with $\pi:\ti N\to N$ the covering map.
Let $p \in N$ be the point where we perform the connected sum operation.
For convenience, we label $\pi^{-1}(p)$ as follows:
$$
\pi^{-1}(p)= \{p_{i}\mid i\in\pi_1(N)\}.
$$
To better describe the neighborhood of $p$ in $N$, we fix a smooth metric $h$ on $N$. Take $\ve$ to be a small positive constant such that
\begin{itemize}\setlength{\itemsep}{1mm}
\item the exponential map $\exp_p^h:B_{\ve}\to B^h_{\ve}(p)$ is diffeomorphism;
\item $B^h_{\ve}(p)$ gives a fundamental neighborhood of $p$, namely we have
$$\pi^{-1}(B_{\ve}^h(p))=\bigsqcup_{i}B^{\ti h}_{\ve}(p_i),$$
where $\ti h$ is the lifted metric of $h$ on $\ti N$.
\end{itemize}
For convenience, we denote
$$
N_{\ve}:=N-B^h_{\ve}(p) \, \mbox{ and } \, \ti N_\ve:=\ti N-\bigcup_iB^{\ti h}_\ve(p_i).
$$
To perform the connected sum we take $X_\ve = X-B$, where $B$ is a small $n$-ball in $X$.
Then without loss of generality  we can take our manifold $Y$ to be $$Y = N \# X= N_\ve \cup X_\ve,$$ where $N_\ve$ and $X_\ve$ are glued on the boundary sphere.
Let
$$\ti Y = \ti N \#_{\pi_1(N)} X$$ and we also denote the covering map $\ti Y \to Y$ by $\pi$.

In the following, we work with a complete metric $g$ on $Y$ with positive scalar curvature and denote by $\ti g$ its lift on $\ti Y$. Our goal is to deduce a contradiction.

Since $g$ is complete, for any smooth closed hypersurface $\Lambda \subset X_\ve$
that comes from smoothing out a level set of the distance function $\dist_g(\cdot,S^h_{\ve}(p))$ with $S^h_{\ve}(p):=\partial B^h_{\ve}(p)$,
it divides $X_\ve$ into an unbounded part and a bounded, ``annulus-like" part. For convenience, we use $X_\Lambda$ to denote the unbounded part of $X$ with boundary $\Lambda$
and let $Y_\Lambda: = Y - X_\Lambda$ (see Figure \ref{manifold_Y}).

In the covering space $\tilde Y$ we write
$$\tilde X_\ve:=\pi^{-1}(X_{\ve}),\,\tilde X_\Lambda:=\pi^{-1}(X_{\Lambda}),\,\tilde Y_\Lambda:=\pi^{-1}(Y_{\Lambda}).$$
Let $\tilde X_{\ve,i}$ denote the component of $\ti X_\ve$ whose boundary is $\partial\tilde X_{\ve,i}=S^{\tilde h}_{\ve}(p_i)$. Denote
$$
\ti X_{\Lambda,i}=\ti X_{\Lambda}\cap \ti X_{\ve,i}\mbox{ and }\tilde \Lambda_{,i}=\pi^{-1}(\Lambda)\cap \ti X_{\ve,i}.
$$
For the reader's convenience, we collect all these notations in Figure \ref{covering_of_Y}.

\begin{figure}[htbp]
\centering
\includegraphics[width=\linewidth]{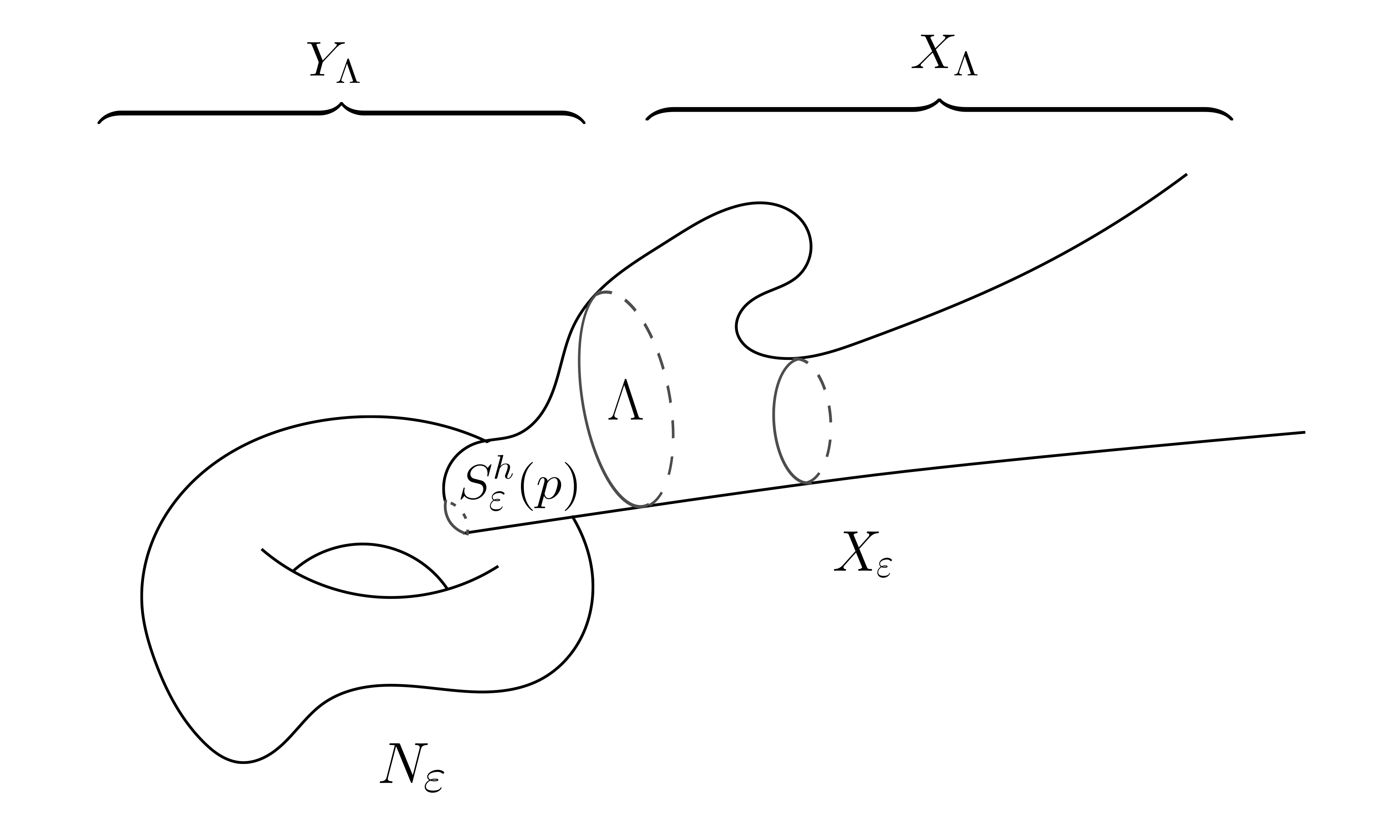}
\caption{The connected sum $Y=N\# X$}
\label{manifold_Y}
\end{figure}
\begin{figure}[htbp]
\centering
\includegraphics[width=\linewidth]{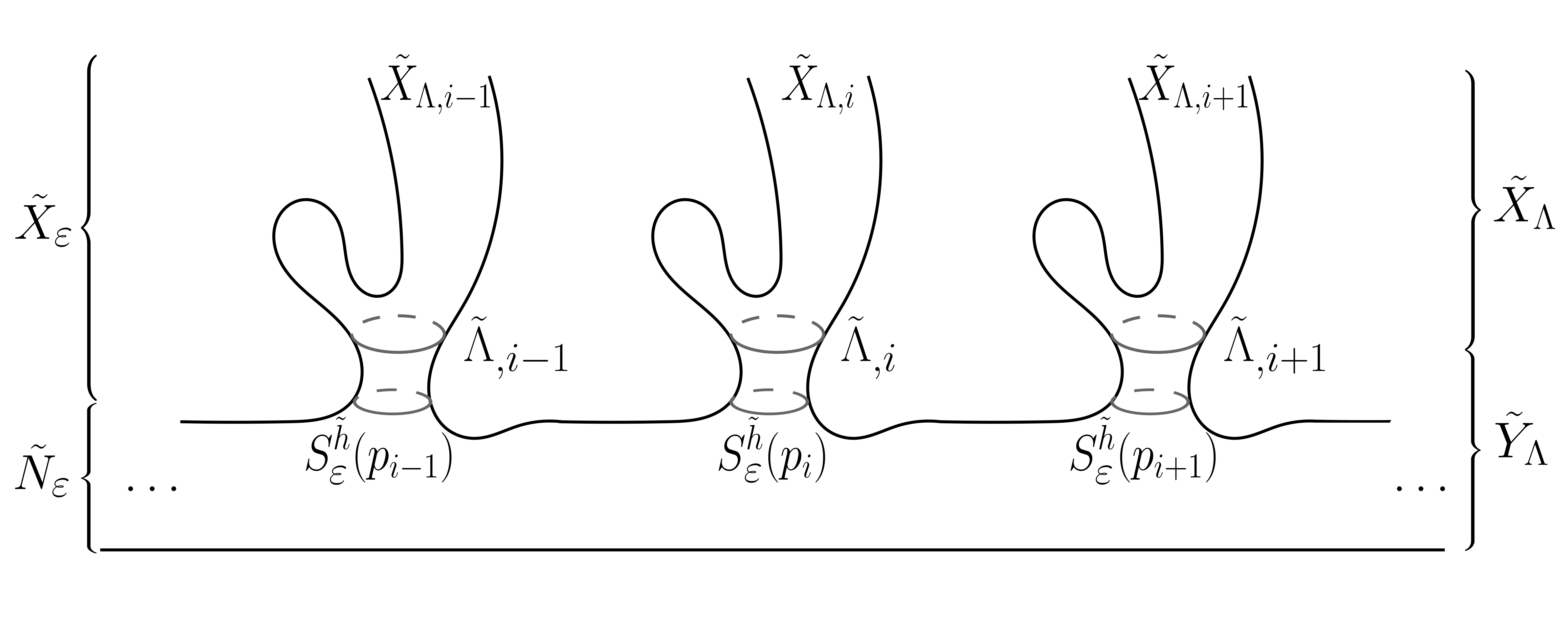}
\caption{The covering space $\ti{Y}=\ti{N}\#_{\pi_1(N)}X$}
\label{covering_of_Y}
\end{figure}

By smoothing out a level-set of the distance function $\dist_g(\cdot,S^h_{\ve}(p))$, we get a smooth closed hypersurface $\Lambda_1 \subset X_\ve$ homologous to $\partial X_\ve = S^h_{\ve}(p)$, and we pick a positive constant $d_{\Lambda_1}$ such that
\begin{equation}\label{Eq: 1}
\dist_g(\Lambda_1,S^h_{\ve}(p))\geq 6d_{\Lambda_1}.
\end{equation}

\subsection{Line construction}\label{subsec: line}
In this subsection we show that we can find a proper line contained in $\tilde{N}_\ve$. Recall
\begin{lemma}[Lemma 6 of {\cite[first version]{CL20}}] \label{Lem: pi_1}
Suppose $N$ is a smooth closed aspherical $n$-dimensional manifold. Then $\pi_1(N)$ is an infinite group, and every nontrivial element of it has infinite order.
\end{lemma}
\begin{proof}
We include the proof for completeness. Since $\tilde{N}$ is contractible we have $H_n(\tilde{N}) = 0$. Any closed orientable connected $n$-manifold $M$ satisfies $H_n(M) = \mathbb Z$, so $\tilde{N}$ is non-compact. Thus, $\pi_1(N)$ is an infinite group.

Take any element $[\sigma] \in \pi_1(N)$. Suppose to the contrary that $[\sigma]$ has finite order $k$. Let $C$ be the cyclic group generated by $[\sigma]$. Then the manifold $\ti{N}/C$ is an Eilenberg--MacLane space $K(\mathbb{Z}_k, 1)$, which is homotopy equivalent to $S^\infty/\mathbb{Z}_k$. This is impossible, as {$H_l(\ti{N}/C,\mathbb{Z}_{k}) = 0$} for all $l > n$.
\end{proof}
Fix a simple closed non-contractible loop $\sigma:[0,1]\to N$ such that $ 0 \neq [\sigma] \in \pi_1(N)$ and $\sigma \cap B^h_{\ve}(p) = \emptyset$.
By Lemma \ref{Lem: pi_1}, $\sigma$ lifts to disjoint lines in the universal cover $\ti N$ and let $\ti \sigma: \mathbb R \to \ti N$ be one lift of $\sigma$ in $\ti N$. By our construction, $\sigma([0,1])\subset N_\ve$, so this induces a loop in $Y$, which we still denote by $\sigma$. Similarly, we also regard $\ti\sigma$ as a line in $\ti Y$.

\begin{lemma}\label{geodesic}
There exists a constant $A$ such that
\begin{itemize}\setlength{\itemsep}{1mm}
\item[(i)] $\ti{\sigma}(\mathbb{R})\subset\ti N_{\ve}$;
\item[(ii)]

$\dist_{\ti{g}}(\ti{\sigma}(t_1),\ti{\sigma}(t_2)) \leq A|t_{1}-t_{2}|$ for all $t_{1},t_{2}\in\mathbb{R}$;
\item[(iii)]
$\lim_{|t_1-t_2|\to +\infty}\dist_{\ti{g}}(\ti{\sigma}(t_1),\ti{\sigma}(t_2))=+\infty$,
\end{itemize}
where $\ti g$ is the lifted metric of $g$ on $\ti Y$.
\end{lemma}
\begin{proof}
(i) follows from the construction of $\ti{\sigma}$. Set $A=\sup_{t\in[0,1]}|\sigma'(t)|_{g}$. Since $\ti{\sigma}$ is the lift of $\sigma$,  $\sup_{t\in\mathbb{R}}|\ti{\sigma}'(t)|_{\ti{g}}=A$. In particular, for any $t_{1},t_{2}\in\mathbb{R}$, we see
\[
\dist_{\ti{g}}(\ti{\sigma}(t_1),\ti{\sigma}(t_2))
\leq \mathrm{length}_{\ti{g}}(\ti{\sigma}|_{[t_{1},t_{2}]}) \leq A|t_{1}-t_{2}|,
\]
which is (ii). For (iii), by applying deck transformations, it suffices to show that
$$
\lim_{t\to +\infty}\dist_{\ti{g}}(\ti{\sigma}(0),\ti{\sigma}(t))=+\infty.
$$
We argue by contradiction. Suppose that $\{t_{i}\}$ is a sequence such that
\[
\lim_{i\rightarrow\infty}t_{i} = +\infty ,\quad \dist_{\ti{g}}(\ti{\sigma}(0),\ti{\sigma}(t_i)) \leq C
\]
for some constant $C$ independent of $i$. Let $\lfloor\cdot\rfloor$ be the greatest integer function. Then by (ii), we see that
\[
\begin{split}
\dist_{\ti{g}}(\ti{\sigma}(0),\ti{\sigma}(\lfloor t_i\rfloor))
\leq {} & \dist_{\ti{g}}(\ti{\sigma}(0),\ti{\sigma}(t_i))+\dist_{\ti{g}}(\ti{\sigma}(t_{i}),\ti{\sigma}(\lfloor t_i\rfloor)) \\
\leq {} & C+A|t_{i}-\lfloor t_{i}\rfloor| \leq C+A.
\end{split}
\]
This implies that $\{\ti{\sigma}(\lfloor t_i\rfloor)\}$ has a convergent subsequence. However, since $\ti \sigma$ is the lift of $\sigma$ and $\ti{\sigma}(\lfloor t_i\rfloor)\in\pi^{-1}(\sigma(0))$, $\{\ti{\sigma}(\lfloor t_i\rfloor)\}$ does not admit any convergent subsequence, and we get a contradiction.
\end{proof}

\subsection{Auxiliary short function}
In this subsection we construct some auxiliary functions for later use. Since their Lipschitz {constants are} strictly less than one, we refer to them as short functions.
\begin{lemma}\label{Lem: rho lambda}
Let $\Lambda$ be a smooth closed {separating} hypersurface $\Lambda \subset X_\ve$. Then there is a proper smooth function
$$\rho_{\Lambda}:X_{\Lambda} \to [0,+\infty)$$
such that
\begin{itemize}\setlength{\itemsep}{1mm}
\item $\rho_{\Lambda}\equiv 0$ in an open neighborhood of $\Lambda$;
\item $|\mathrm d\rho_{\Lambda}|_g < 1$.
\end{itemize}
\end{lemma}
\begin{proof}
Fix a small positive constant $\tau$ such that the distance function
$$\mathfrak d(\cdot):=\dist_g(\cdot,\Lambda)$$ is smooth in the $\tau$-neighborhood of $\Lambda$ in $Y$. We start with the function
$$
\rho_1=\max\{0,\mathfrak d-\tau\}.
$$
Through a standard mollification as in {\cite{GW79}}, we are able to construct a proper smooth function $\rho_2: X_{\Lambda} \to [0,+\infty)$ having bounded Lipschitz constant and satisfying $\rho_2\equiv 0$ in the $(\tau/2)$-neighborhood of $\Lambda$. After taking $$\rho_{\Lambda}=\frac{\rho_2}{\Lip\rho_2+1}$$ we can guarantee $|\mathrm d\rho_{\Lambda}|_g< 1$
and this completes the proof.
\end{proof}

From \eqref{Eq: 1} we see that
$$
\dist_{\tilde g} (\ti X_{\Lambda_1,i},\ti X_{\Lambda_1,j})\geq 12 d_{\Lambda_1}\mbox{ when } i\neq j.
$$
We construct suitable ``lifts" of $\rho_{\Lambda_1}$ for later use.

\begin{lemma}\label{Lem: rho_i}
For each $i$ we can construct a smooth function
$$
\tilde \rho_i:\tilde Y\to \mathbb R
$$
such that
\begin{itemize}\setlength{\itemsep}{1mm}
\item the set $\tilde \rho_i^{-1}([-d_{\Lambda_1},d_{\Lambda_1}])$ is contained in $\ti Y_{\Lambda_1}$;
\item Let $\rho_{\Lambda_1}$ be the function from Lemma \ref{Lem: rho lambda}. Then there is a constant $d_{\Lambda_1}'>d_{\Lambda_1}$ such that
$$\tilde \rho_i=-\rho_{\Lambda_1}\circ \pi-d_{\Lambda_1}'\mbox{ in } \ti X_{\Lambda_1,i}$$
and
$$\tilde \rho_i=\rho_{\Lambda_1}\circ \pi+d_{\Lambda_1}'\mbox{ in }  \ti X_{\Lambda_1,j} \mbox{ when }j\neq i;$$
\item $|\mathrm d\tilde\rho_i|_{\tilde g}< 1$.
\end{itemize}
\end{lemma}

\begin{proof}
Let $\mathfrak d$ be the signed distance function to $\ti\Lambda_{1,i}$ taking negative values in $\ti X_{\Lambda_1,i}$. Put
$
\hat \rho_1=\mathfrak d-3d_{\Lambda_1}
$.
Through a standard mollification as in {\cite{GW79}} for any $\tau>0$ we can find a smooth proper function $\hat\rho$ satisfying
$|\hat \rho-\hat \rho_1|\leq \tau \mbox{ and }|\mathrm d\hat \rho|_{\tilde g}\leq 1+\tau. $
Define
$$
\hat\rho_2=\frac{\hat \rho}{1+2\tau}.
$$
By taking $\tau$ small enough we can guarantee $$\hat \rho_2\leq-2d_{\Lambda_1}\mbox{ in } \ti X_{\Lambda_1,i}$$ and $$\hat \rho_2\geq 2d_{\Lambda_1}\mbox{ in }  \ti X_{\Lambda_1,j}\mbox{ when } j\neq i.$$
Take a smooth and even cut-off function $\eta:(-\infty,+\infty)\to [0,1]$ such that $\eta\equiv 1$ in $[-d_{\Lambda_1},d_{\Lambda_1}]$ and $\eta\equiv 0$ in $(-\infty,-\frac{3}{2}d_{\Lambda_1}]\cup[\frac{3}{2}d_{\Lambda_1},+\infty)$. Denote
$$
\hat \rho_3=\int_0^{\hat \rho_2}\eta(t)\,\mathrm dt
$$
and
$$d_{\Lambda_1}'=\int_0^{2d_{\Lambda_1}}\eta(t)\,\mathrm dt.$$
Define
\[
\tilde \rho_i =
\begin{cases}
\quad \quad \quad \hat\rho_3 & \mbox{ in } \ti Y_{\Lambda_1};\\[1mm]
\ -\rho_{\Lambda_1}\circ\pi-d_{\Lambda_1}' & \mbox{ in } \ti X_{\Lambda_1,i};\\[1mm]
\ \ \ \rho_{\Lambda_1}\circ \pi+d_{\Lambda_1}' & \mbox{ in } \cup_{j\neq i} \ti X_{\Lambda_1,j}.
\end{cases}
\]
Then it is easy to verify that $\tilde \rho_i$ satisfies all the desired properties.
\end{proof}

\subsection{Weight functions}

\begin{lemma}\label{Lem: weight function 1}
Given any positive constants $\nu_0$ and $d_0$, there is a positive constant $\delta_0=\delta_0(\nu_0,d_0)$ such that for any positive constant $\delta\in (0,\delta_0)$ we can find $T_\delta>d_0$ and a smooth function
$$h_\delta:(-T_\delta,T_\delta)\to \mathbb R$$ such that
\begin{itemize}\setlength{\itemsep}{1mm}
\item $\lim_{t\to\pm T_\delta}h_\delta(t)=\mp\infty$ and $h_\delta'<0$;
\item $h_\delta^2-2|h_\delta'|+\nu_0\chi_{[-d_0,d_0]}\geq \delta$.
\end{itemize}
\end{lemma}
\begin{proof}
Take a smooth cut-off function $\eta:(-\infty,+\infty)\to [0,1]$ such that $\eta\equiv 1$ in $(-\infty,-d_0]$, $\eta\equiv 0$ in $[d_0,+\infty)$ and $\eta'\leq 0$. Take $\delta_0$ to be a positive constant no greater than $\nu_0/2$ such that
\begin{equation}\label{Eq: delta0}
\delta_0\|2\eta-1\|_{C^0}^2+2\sqrt{\delta_0}\|2\eta-1\|_{C^1}<\nu_0/2.
\end{equation}
For any $\delta\in (0,\delta_0)$ and integer $k>d_0$ we denote
$$
f_k^+:(-k,+\infty)\to (\sqrt\delta,+\infty),\,t\mapsto\sqrt \delta\coth\left(\frac{\sqrt\delta }{2}(k+t)\right),
$$
and
\begin{equation}\label{Eq: fk-}
f_k^-:(-\infty,k)\to (-\infty,-\sqrt\delta),\,t\mapsto-\sqrt \delta\coth\left(\frac{\sqrt\delta }{2}(k-t)\right).
\end{equation}
It is easy to verify $(f_k^{\pm})^2+2(f_k^\pm)'=\delta$, $(f_k^{\pm})'<0$, $f_k^+>0$, $f_k^-<0$.
Define
$$
h_\delta=\eta f_k^++(1-\eta)f_k^-:(-k,k)\to\mathbb R
$$
with $k$ to be determined later.
Then it is easy to verify
$$
\lim_{t\to\pm k} h_\delta=\lim_{t\to \pm k}f_{k}^{\mp}=\mp\infty
$$
and
$$
h_\delta'=\eta (f^+_k)'+(1-\eta)(f_k^-)'+\eta'(f^+_k-f_k^-)<0.
$$
Notice that $h_\delta$ converges smoothly to the function $(2\eta-1)\sqrt\delta$ in the closed interval $[-d_0,d_0]$ as $k\to+\infty$. From estimate \eqref{Eq: delta0} we can take $k$ large enough such that
$$h_\delta^2-2|h_\delta'|+\nu_0>\nu_0/2\geq \delta\mbox{ in }[-d_0,d_0].$$
Denote such $k$ by $T_{\delta}$. Outside $[-d_0,d_0]$ we have $h_\delta^2-2|h_\delta'|=\delta$ and so $h_\delta$ satisfies all the desired properties.
\end{proof}

\begin{lemma}\label{Lem: weight function 2}
Given any positive constants $\nu_0$ and $d_0$, there is a positive constant $\delta_0=\delta_0(\nu_0,d_0)$ such that for any positive constant $\delta\in (0,\delta_0)$, we can find $T_\delta>d_0$ and a smooth function
$$h_\delta:[0,T_\delta)\to (-\infty,0]$$ such that
\begin{itemize}\setlength{\itemsep}{1mm}
\item $h_\delta\equiv 0$ around $t=0$, $\lim_{t\to T_\delta}h_\delta(t)=-\infty$ and $h_\delta'\leq 0$;
\item $h_\delta^2-2|h_\delta'|+\nu_0\chi_{[0,d_0]}\geq \delta$.
\end{itemize}
\end{lemma}
\begin{proof}
The proof is similar to that of Lemma \ref{Lem: weight function 1}. Take a smooth cut-off function $\eta:[0,+\infty)\to [0,1]$ such that $\eta\equiv 1$ around $t=0$, $\eta\equiv 0$ in $[d_0,+\infty)$ and $\eta'\leq 0$. Take $\delta_0$ to be a positive constant no greater than $\nu_0/2$ such that
\begin{equation*}
\delta_0\|\eta-1\|_{C^0}^2+2\sqrt{\delta_0}\|\eta-1\|_{C^1}<\nu_0/2.
\end{equation*}
For any $\delta\in (0,\delta_0)$ and integer $k>d_0$ we take
$
h_{\delta}=(1-\eta)f_k^-
$,
where $f_k^-$ is the function from \eqref{Eq: fk-}. As $k\to +\infty$ the function $h_\delta$ converges smoothly to $-\sqrt\delta (1-\eta)$ {on $[0,d_{0}]$} and so by taking $k$ large enough we can guarantee
$$
h_\delta^2-2|h_\delta'|+\nu_0\geq \nu_0/2\geq\delta\mbox{ in }[0,d_0].
$$
Denote such $k$ by $T_{\delta}$. Outside $[0,d_0]$ we have $h_\delta^2-2|h_\delta'|=\delta$ from the definition of $f_k^-$.
The remaining properties can be verified easily and we complete the proof.
\end{proof}

\subsection{Topology}
Here we collect some topological facts about  $\tilde{Y}$. To deal with the arbitrary ends $\tilde{X}_\ve$ over which we have no control, we use relative homology theory so that the ends $\tilde{X}_\ve$ have no impact on our topological analysis. All the homology groups in this paper are assumed to have $\mathbb Z$ coefficients, unless otherwise noted.
\begin{lemma}\label{Lem: excision}
We have
$$
H_k(\ti Y, \ti X_{\ve})=0
$$
for all integers $k\geq 2$. Moreover, if $\gamma$ is a closed curve in $\tilde Y$, then $[\gamma]=0\in H_1(\tilde Y,\tilde X_\ve)$.
\end{lemma}
\begin{proof}
From excision we have
$$
H_k(\tilde Y,\tilde X_{\ve})=H_k\left(\tilde N_\ve,\bigcup_i \partial B^{\tilde h}_{\ve}(p_i)\right)=H_k\left(\tilde N,\bigcup_i B_\ve^{\tilde h}(p_i)\right).
$$
From the exact sequence of pair we have
$$
H_{k}\left(\bigcup_i B_\ve^{\tilde h}(p_i)\right)\to H_k(\tilde N)\to H_k\left(\tilde N,\bigcup_i B_\ve^{\tilde h}(p_i)\right)
\to H_{k-1}\left(\bigcup_i B_\ve^{\tilde h}(p_i)\right).
$$
Notice that the first and fourth terms are zero when $k\geq 2$ and so it follows from the contractibility of $\tilde N$ that
$$
H_k\left(\tilde N,\bigcup_iB^{\tilde h}_{\ve}(p_i)\right)=H_k(\tilde N)=0.
$$
Since $\tilde Y=\tilde N\#_{\pi_1(N)}X$, we can fix a continuous map
$$
F:\tilde Y\to \tilde N\mbox{ with }F|_{\tilde N_\ve}=\mathrm {id}:\tilde N_\ve\to \tilde N_\ve.
$$
{From $F(\tilde X_\ve)\subset\bigcup_{i} B_\ve^{\tilde h}(p_i)$ and the excision property,} we see that $F_*:H_1(\tilde Y,\tilde X_\ve)\to H_1(\tilde N,\bigcup_{i} B_\ve^{\tilde h}(p_i))$ is an isomorphism.
To see $[\gamma]=0\in H_1(\tilde Y,\tilde X_\ve)$ for any closed curve $\gamma$, we just need to show $[F\circ \gamma]=0\in H_1(\tilde N,\bigcup_{i} B_\ve^{\tilde h}(p_i))$. It is well-known that the curve $F\circ \gamma$ can be deformed continuously to a new closed curve which avoids touching all balls $B_\ve^{\tilde h}(p_i)$. This yields $\partial[F\circ \gamma]=0$ under the map
$$
\partial:H_1\left(\tilde N,\bigcup_i B_\ve^{\tilde h}(p_i)\right)\to H_{0}\left(\bigcup_i B_\ve^{\tilde h}(p_i)\right).
$$
The exactness then implies that $[F\circ\gamma]$ must lie in the image of
$$
H_1(\tilde N)\to H_1\left(\tilde N,\bigcup_i B_\ve^{\tilde h}(p_i)\right).
$$
Therefore we obtain $[F\circ\gamma]=0$ from the fact that $\tilde N$ is contractible (and in particular $H_1(\tilde N)=0$), which completes the proof.
\end{proof}

Moreover, we have the following quantitative filling lemma, which is a relative version of \cite[Proposition 10]{CL20}.
\begin{lemma}\label{Lem: quantitative filling}
For any $k\in \mathbb N_+$ there is a function $F:(0,+\infty)\to (0,+\infty)$ depending only on the triple $(Y, N_\ve, g)$ such that if $C$ is a relative $k$-boundary in $\mathcal B_k(\ti Y, \ti X_{\ve})$ contained in some geodesic ball $B_{r}^{\ti{g}}(\tilde q)$ of $(\ti Y,\tilde g)$ with $\tilde q\in \tilde N_\ve$, then we can find a relative $(k+1)$-chain $\Gamma \in \mathcal Z_{k+1}(\ti Y,  \ti X_{\ve})$ contained in the geodesic ball $B_{F(r)}^{\ti{g}}(\tilde q)$ of $(\tilde Y,\tilde g)$ with $\partial\Gamma=C$ modulo $ \ti X_{\ve}$.
\end{lemma}
\begin{proof}
Fix a point $q_0\in\tilde N_\ve\subset \tilde Y$. For any $r>0$ {we can find a domain $U_r$ with smooth boundary such that $B_{r}^{\ti{g}}(q_0) \subset U_r \subset B_{r+1}^{\ti{g}}(q_0)$}. The kernel of the inclusion
$$
i_*:H_{k}({U_r}, \tilde X_{\ve})\to H_k(\tilde Y, \tilde X_{\ve})
$$
is finitely generated {because it is a subgroup of a finitely generated abelian group,} and so there exists some $R_0=R_0(r)<+\infty$ such that each relative $k$-boundary in $\mathcal B_k(\ti Y,  \ti X_{\ve})$ contained in ball $B_{r}^{\ti{g}}(q_0)$ can be filled in $B_{R_0}^{\ti{g}}(q_0)$.

For any point $q\in \tilde N_\ve\subset \tilde Y$ other than $q_0$, there is a deck transformation $\Psi$ so that $d(q_0,\Psi(q)) \leq D_0:=\diam (N_\ve,g)$. As such, we have $\Psi(B_{r}^{\ti{g}}(q)) \subset B_{r+D_0}^{\ti{g}}(q_0)$. Now let $C$ be a relative $k$-boundary in $\mathcal B_k(\tilde Y, \tilde X_{\ve})$ contained in $B_{r}^{\ti{g}}(q)$, then $\Psi_\#C$ is a relative $k$-boundary contained in $B_{r+D_0}^{\ti{g}}(q_0)$. From previous discussion we can find a relative $(k+1)$-chain $\Gamma \in\mathcal Z_{k+1}(\tilde Y,\tilde X_{\ve})$ contained in $B_{R_0(r+D_0)}^{\ti{g}}(q_0)$ such that $\partial \Gamma=\Psi_\# C$ modulo $\tilde X_{\ve}$. Then it follows that $(\Psi^{-1})_\#\Gamma$ is a relative $(k+1)$-chain with relative $k$-boundary $\Gamma$ which is contained in $B_{R_0(r+D_0)+D_0}^{\ti{g}}(q)$. The proof is completed by taking $F(r)=R_0(r+D_0)+D_0$.
\end{proof}

\subsection{$\mathbb T^l$-invariant variational problems}
The results here are similar to the ones in \cite[Sections 3-4]{CL20} about warped $\mu$-bubbles and free boundary warped $\mu$-bubbles, but we present them using a warped product formulation.
\begin{definition}\label{Defn: stabilized curvature}
Let $(M,g)$ be a Riemannian manifold possibly with non-empty boundary and $\Gamma$ be a smooth portion of $\partial M$. Let $R$ and $H$ be smooth functions on $M$ and $\Gamma$ respectively. We say that $(M,\Gamma)$ has $\mathbb T^l$-stabilized scalar-mean curvature $(R,H)$ if there are $l$ positive smooth functions $v_1,v_2,\ldots,v_l$ on $M$ such that $(M\times \mathbb T^l,g+\sum_i v_i^2\mathrm d\theta_i^2)$ has scalar curvature $\tilde R$ in $\hat M : = M\times \mathbb T^l$ and has mean curvature $\tilde H$ on $\Gamma\times \mathbb T^l$, where $\tilde R$ and $\tilde H$ are $\mathbb T^l$-invariant extension of $R$ and $H$ respectively (of course $\operatorname{dim} M+l\geq 2$ is required to compute the scalar curvature).
\end{definition}
\begin{lemma}\label{Lem: plateau}
Let $2\leq n\leq 7$. Let $(M^n,g)$ be a complete Riemannian manifold {without boundary} and $S$ be an embedded {closed} submanifold in $M$ with codimension two and $0 = [S]\in H_{n-2}(M)$. If $M$ has $\mathbb T^l$-stabilized scalar curvature $R$, then we can construct a properly embedded, complete hypersurface $\Sigma$ with $\partial\Sigma=S$ and $\mathring \Sigma$ has $\mathbb T^{l+1}$-stabilized scalar curvature $R'$ which is no less than $R|_\Sigma$.
\end{lemma}
\begin{proof}
Denote $\hat M=M\times \mathbb T^l$. Since $M$ has $\mathbb T^l$-stabilized scalar curvature $R$, there are finitely many smooth positive functions $v_1,v_2,\ldots, v_l$ on $M$ such that the warped product metric
\begin{equation}\label{Eq: hat g}
\hat g=g+\sum_{i=1}^lv_i^2\mathrm d\theta_i^2
\end{equation}
satisfies $R(\hat g)=\tilde R$, where $\tilde R$ is the $\mathbb T^l$-invariant extension of $R$ on $\hat M$. We would like to minimize the area functional among hypersurfaces in the form of $\Sigma\times \mathbb T^l$ with $\partial\Sigma=S$. This is equivalent to solving the Plateau problem for $S$ in the Riemannian manifold $(M,\tilde g)$ with
\begin{equation}\label{Eq: conformal tilde g}
\tilde g=\left(\prod_i v_i\right)^{\frac{2}{n-1}}g.
\end{equation}
It follows from a modification of the proof of \cite[Lemma 34.1]{Sim83} and the regularity theory for mass-minimizing current that when $2\leq n\leq 7$ we can find a properly embedded (possibly non-compact) hypersurface $\Sigma$ minimizing the area functional with $\partial\Sigma=S$.

Next we show that $\mathring\Sigma$ has $\mathbb T^{l+1}$-stabilized scalar curvature $R'$ which is no less than $R|_{\Sigma}$. Notice that $\Sigma \times \mathbb T^l$ is actually a constrained area minimizer in $(\hat M,\hat g)$ among $\mathbb T^l$-invariant hypersurfaces. From the $\mathbb T^l$-invariance of $(\hat M,\hat g)$ we see that the mean curvature of $\Sigma \times \mathbb T^l$ is $\mathbb T^l$-invariant. This yields that $\Sigma \times \mathbb T^l$ has to be a minimal hypersurface as a constrained area-minimizer. We claim that $\Sigma \times \mathbb T^l$ is also stable. To see this we consider a $\mathbb T^l$-invariant exhaustion $U_i\times \mathbb T^l$ of $\Sigma\times \mathbb T^l$. The key is to notice that the Jacobi operator
\begin{equation}\label{Eq: Jacobi}
\mathcal J=-\Delta-(\mathrm{Ric}(\nu)+|A|^2)
= -\Delta - \frac{1}{2}(R(\hat g)|_{\Sigma\times \mathbb T^l} - R(\hat{g}|_{\Sigma\times \mathbb T^l})+|A|^2)
\end{equation}
(corresponding to the second variation of the area functional $\mathcal H^{n+l}_{\hat g}(\cdot)$) on $\Sigma\times\mathbb T^{l}$ is also $\mathbb T^{l}$-invariant. {Here the second identity follows from Schoen--Yau's rearrangement trick.} The uniqueness of the first eigenfunction (up to scaling) implies that the first eigenfunction with Dirichlet boundary condition of the Jacobi operator on $U_i\times \mathbb T^l$ is again $\mathbb T^l$-invariant. The property of constrained area minimizers then yields the stability of $\Sigma \times \mathbb T^l$. Notice that the properness of $\Sigma$ actually implies the completeness of $\Sigma\times \mathbb T^l$.

Now the argument can be divided into two cases.

\medskip

{\it Case 1. $\Sigma$ is compact.} We take $\tilde v_{l+1}$ to be the first eigenfunction with Dirichlet boundary condition of the Jacobi operator. Then $\tilde v_{l+1}$ is $\mathbb T^l$-invariant from previous discussion and it induces a smooth function $v_{l+1}$ on $\Sigma$ with $v_{l+1}>0$ in $\mathring \Sigma$, and we have $\mathcal J\tilde v_{l+1}\geq 0$.

\medskip

{\it Case 2. $\Sigma$ is non-compact.} Repeating Fischer-Colbrie--Schoen's construction
{as in \cite[Proof of Theorem 1]{FcS80}} we can construct a $\mathbb T^l$-invariant function $\tilde v_{l+1}$ on $\Sigma \times \mathbb T^l$ such that $\mathcal J\tilde v_{l+1}=0$ and $\tilde v_{l+1}>0$ in $\mathring \Sigma \times \mathbb T^l$. Again it induces a smooth function $v_{l+1}$ on $\Sigma$ with $v_{l+1}>0$ in $\mathring \Sigma$.

\medskip
On $\Sigma \times \mathbb T^{l+1}$, denote
\begin{equation}\label{Eq: hat g'}
\hat g'=g|_\Sigma+\sum_{i=1}^l(v_i|_{\Sigma})^2\mathrm d\theta_i^2+v_{l+1}^2\mathrm d\theta_{l+1}^2.
\end{equation}
In both cases, by viewing $(\Sigma \times \mathbb T^{l+1}, g')$ as the warped product of $(\Sigma \times \mathbb T^{l}, \hat{g}|_{\Sigma\times \mathbb T^l})$ and $(\mathbb{S}^1, d\theta_{l+1}^2)$ with warping function $\tilde v_{l+1}$, we can compute
$$
R(\hat g')=R(\hat{g}|_{\Sigma\times \mathbb T^l})-\frac{2\Delta \tilde v_{l+1}}{\tilde v_{l+1}}\geq R(\hat g)|_{\Sigma\times \mathbb T^l} = \tilde R
\, \mbox{ in }\mathring \Sigma \times \mathbb T^l,
$$
{where the inequality follows from \eqref{Eq: Jacobi}, $\mathcal{J}\tilde v_{l+1} \ge 0$ and $v_{l+1}>0$ in $\mathring \Sigma$.}
That is, $\mathring\Sigma$ has $\mathbb T^{l+1}$-stabilized scalar curvature $R'$ which is no less than $R|_\Sigma$.
\end{proof}

\begin{lemma}\label{Lem: homology minimizing}
Let $2\leq n\leq 7$. Let $(M^n,g)$ be a closed Riemannian manifold and $\beta$ be an $(n-1)$-homology class with $0 \neq \beta\in H_{n-1}(M)$. If $M$ has $\mathbb T^l$-stabilized scalar curvature $R$, then we can construct an embedded hypersurface $\Sigma$ with integer multiplicity representing $\beta$ such that $\Sigma$ has $\mathbb T^{l+1}$-stabilized scalar curvature $R'$ which is no less than $R|_\Sigma$ and the metric completion of $M-\Sigma$ associated with its boundary has $\mathbb T^{l}$-stabilized scalar-mean curvature $(R,0)$.
\end{lemma}
\begin{proof}
Since $M$ has $\mathbb T^l$-stabilized scalar curvature $R$, there are finitely many smooth positive functions $v_1,v_2,\ldots, v_l$ on $M$ such that the warped product metric $\hat g$ defined by \eqref{Eq: hat g}
satisfies $R(\hat g)=\tilde R$, where $\tilde R$ is the $\mathbb T^l$-invariant extension of $R$ on $\hat M$.
It follows from geometric measure theory that we can find an embedded hypersurface $\Sigma$ with integer multiplicity representing the homology class $\beta$ which is homologically area-minimizing in $(M,\tilde g)$, where $\tilde g$ is the metric defined by \eqref{Eq: conformal tilde g}.
From a similar argument as in the proof of Lemma \ref{Lem: plateau} we conclude that $\Sigma$ has $\mathbb T^{l+1}$-stabilized scalar curvature $R'$ which is no less than $R|_\Sigma$. The last assertion comes immediately from the construction of $\Sigma$.
\end{proof}

\begin{definition}
We say that $(M,\partial_\pm,g)$ is a Riemannian band if $(M,g)$ is a {connected, compact, smooth Riemannian manifold with piecewise smooth boundary (e.g. a square in the plane),} and $\partial_+$ and $\partial_-$ are two disjoint non-empty smooth compact portions of $\partial M$ such that $\overline{\partial M-(\partial_+\cup \partial_-)}$ is also a compact smooth portion of $\partial M$ {(note that $\overline{\partial_{+}\cup\partial_{-}}$ and $\overline{M-\partial_{+}\cup\partial_{-}}$ can intersect transversally)}.
\end{definition}
\begin{lemma}\label{Lem: smooth mu bubble}
Let $2\leq n\leq 7$. Let $(M,\partial_\pm,g)$ be a Riemannian band with $\overline{\partial M-(\partial_+\cup\partial_-)}=\emptyset$ and $\dist(\partial_+,\partial_-)> 2d_0$. If M has $\mathbb T^l$-stabilized scalar curvature $R$, then we can find an embedded hypersurface $\Sigma$ {such that it bounds a region with $\partial_-$, and} has $\mathbb T^{l+1}$-stabilized scalar curvature $R'$ with
$$
R' > R|_\Sigma-\frac{\pi^2}{d_0^2}.
$$
\end{lemma}
\begin{proof}
Since $\dist(\partial_+,\partial_-)>2d_0$, we can construct a smooth short function $\rho:(M,\partial_\pm)\to ([-d_0,d_0],\pm d_0)$ such that $\rho^{-1}(d_0)=\partial_+$, $\rho^{-1}(-d_0)=\partial_-$ and $\Lip \rho<1$. Take $h: [-d_0,d_0] \to [-\infty, \infty]$ such that
$$
h=-\frac{n+l-1}{n+l}\cdot\frac{\pi}{d_0}\tan\left(\frac{\pi}{2d_0}t\right)
$$
and $h$ satisfies
$$
\frac{n+l}{n+l-1}h^2-2|h'|=-\frac{n+l-1}{n+l}\cdot\frac{\pi^2}{d_0^2}
> -\frac{\pi^2}{d_0^2}.
$$
Since $M$ has $\mathbb T^l$-stabilized scalar curvature $R$, there are finitely many smooth positive functions $v_1,v_2,\ldots, v_l$ on $M$ such that the warped product metric $\hat g$ defined by \eqref{Eq: hat g}
satisfies $R(\hat g)=\tilde R$, where $\tilde R$ is the $\mathbb T^l$-invariant extension of $R$ on $\hat M = M \times \mathbb T^l$. Let us consider the following minimization problem. Set
$$
\mathcal A^h(\Omega)=\mathcal H^{n-1}_{\tilde g}(\partial\Omega \cap \mathring M)-\int_{\mathring M}(\chi_\Omega-\chi_{\Omega_0})(h\circ \rho)\left(\prod_i v_i\right)^{-\frac{1}{n-1}}\,\mathrm d\mathcal H^n_{\tilde g},
$$
where $\tilde g$ is the metric defined by \eqref{Eq: conformal tilde g}, $\Omega_0=\rho^{-1}((-\infty,0])$ and $\Omega$ belongs to the class
\begin{equation}\label{defn: C}
\mathcal C=\{\text{Caccioppoli sets $\Omega$ in $M$ such that $\Omega\Delta \Omega_0\Subset \mathring M$}\}.
\end{equation}
Notice that the function
\begin{equation}\label{Eq: tilde h}
\tilde h:=(h\circ \rho)\left(\prod_i v_i\right)^{-\frac{1}{n-1}}
\end{equation}
takes value $+\infty$ on $\partial_-$ and takes value $-\infty$ on $\partial_+$. It follows from \cite[Proposition 2.1]{Zhu21} that we can find a smooth minimizer $\Omega_{\mathrm{min}}$ of $\mathcal A^h$ in $\mathcal C$. Let $\Sigma=\partial\Omega_{\mathrm{min}}\cap \mathring M$. Then $\Sigma$ and $\partial_-$ bound the region $\Omega_{\mathrm{min}}$.

{Notice that the above minimizing problem is equivalent to minimizing the functional
$$\hat{\mathcal A}^h(\Omega)=\mathcal H_{\hat g}^{n+l-1}((\partial\Omega\cap \mathring M)\times \mathbb T^l)-\int_{\mathring M\times\mathbb T^l}(\chi_{\Omega\times \mathbb T^l}-\chi_{\Omega_0\times \mathbb T^l})h\circ \rho\,\mathrm d\mathcal H^{n+l}_{\hat g}$$
among the class $\mathcal C$ on $(M\times \mathbb T^l,\hat g)$ with $\hat g=g+\sum_{i=1}^l v_i^2\mathrm d\theta_i^2$. It follows from \cite[(1.3)]{ZZ20} that the Jacobi operator $\mathcal J^h$ associated to the functional $\mathcal A^h$ is
$$\mathcal J^h=-\Delta_{\Sigma\times \mathbb T^l}-(\mathrm{Ric}_{\hat g}(\nu)+|A|^2+\partial_\nu(h\circ \rho\circ\pi)).$$
Using Schoen--Yau's rearrangement trick and the facts $|A|^2\geq \frac{(h\circ\rho)^2}{n+l-1}$ and $\partial_\nu(h\circ \rho)\geq -|h'|\circ\rho$, we know that the modified Jacobi operator
$$\mathcal J^h_*=-\Delta_{\Sigma\times \mathbb T^l}- \frac{1}{2}\left(R(\hat g)|_{\Sigma\times \mathbb T^l} - R(\hat{g}|_{\Sigma\times \mathbb T^l})+\left(\frac{n+l}{n+l-1}h^2-2|h'|\right)\circ \rho\right)$$
is a positive operator.}
 With a similar argument as in the proof of Lemma \ref{Lem: plateau} by taking $v_{l+1}$ to be the first eigenfunction of the modified Jacobi operator $\mathcal J^h_*$, we obtain on $(\Sigma \times \mathbb T^{l+1}, \hat g')$ that
 \begin{equation}\label{Eq: scalar change}
     R(\hat g')= R(\hat g)|_{\Sigma\times \mathbb T^l}+\left(\frac{n+l}{n+l-1}h^2-2|h'|\right)\circ \rho,
 \end{equation}
where $\hat g'$ is defined by \eqref{Eq: hat g'}. {Using the expression of $u$, we have
$$ R(\hat g')> R(\hat g)|_{\Sigma\times \mathbb T^l}-\frac{\pi^2}{d_0^2}.$$}
\end{proof}

\begin{lemma}\label{Lem: corner mu bubble}
Let $2\leq n\leq 7$. Let $(M,\partial_\pm,g)$ be a Riemannian band with $\Gamma:=\overline{\partial M-(\partial_+\cup\partial_-)}\neq\emptyset$ and $\dist(\partial_+,\partial_-)>2d_0$. If the dihedral angles between $\Gamma$ and $\partial_+\cup\partial_-$ are less than $\pi/2$ and $(M,\Gamma)$ has $\mathbb T^l$-stabilized scalar-mean curvature $(R,H)$, then we can find a hypersurface $\Sigma$ (possibly with boundary $\partial\Sigma$) intersecting $\Gamma$ orthogonally such that
\begin{itemize}\setlength{\itemsep}{1mm}
\item $\Sigma$ and $\partial_-$ bound a relative region $\Omega$ relative to $\Gamma$, namely we have $\partial\Omega-(\Sigma\cup \partial_-)\subset \Gamma$;
\item $(\Sigma,\partial\Sigma)$ has $\mathbb T^{l+1}$-stabilized scalar-mean curvature $(R',H')$ where {$R' > R-\pi^2/d_0^2$} and $H'=H$.
\end{itemize}
\end{lemma}
\begin{proof}
The proof is almost the same as that of Lemma \ref{Lem: smooth mu bubble}. {We still consider the minimizing problem for the functional
$$
\mathcal A^h(\Omega)=\mathcal H^{n-1}_{\tilde g}(\partial\Omega \cap \mathring M)-\int_{\mathring M}(\chi_\Omega-\chi_{\Omega_0})(h\circ \rho)\left(\prod_i v_i\right)^{-\frac{1}{n-1}}\,\mathrm d\mathcal H^n_{\tilde g},
$$
but among the class
$$
\mathcal C=\{\text{Caccioppoli sets $\Omega$ in $M$ such that $\Omega\Delta \Omega_0\Subset  M\setminus(\partial_+\cup\partial_-)$}\}.
$$
Notice that we are considering warped $h$-hypersurface with free boundary here.}  The key to finding a smooth minimizer of $\mathcal A^h$ is the existence of a minimizing sequence $\Omega_i$ whose reduced boundary $\partial\Omega_i\cap \mathring M$ stays away from a fixed neighborhood of $\partial_+\cup\partial_-$, {where the smoothness of the limit can be obtained from the regularity result \cite[Theorem 2.1]{SWZ24}.}

Let $\Omega_i$ be a minimizing sequence of the functional $\mathcal A^h$ in $\mathcal C$. Namely, we have
$$
\lim_{i\to\infty}\mathcal A^h(\Omega_i)=\inf_{\Omega\in\mathcal C}\mathcal A^h(\Omega).
$$
We show how to modify $\Omega_i$ to obtain a new minimizing sequence $\tilde \Omega_i$ such that the reduced boundary $\partial\tilde\Omega_i\cap\mathring M$ stays away from a fixed neighborhood of $\partial_+$. The modification can be done similarly around $\partial_-$ and this provides the desired minimizing sequence. From continuity we can find a family of tubular neighborhoods $\{\mathcal N_s\}_{s\in (0,s_0]}$ of $\partial_+$ such that
\begin{itemize}\setlength{\itemsep}{1mm}
\item $\cap_s\mathcal N_s=\partial_+$;
\item $\partial\mathcal N_s\cap\mathring M$ are equidistant hypersurfaces to $\partial_+$ which intersect $\Gamma$ in acute angles (the inner product of outward unit normals of $\partial\mathcal N_s\subset \mathcal N_s$ and $\Gamma\subset M$ is positive);
\item $\partial\mathcal N_s$ has mean curvature $H_s$ with respect to outward unit normal of $\partial\mathcal N_s\subset \mathcal N_s$, where $H_s<-\tilde h|_{\partial\mathcal N_s}$ and $\tilde h$ is the function defined by \eqref{Eq: tilde h}.
\end{itemize}
Define
$$
\tilde \Omega_i=\Omega_i-\mathcal N_{s_0}.
$$
Clearly $\tilde \Omega_i$ still belongs to the class $\mathcal C$ {defined as above}. It remains to show $\mathcal A^h(\tilde \Omega_i)\leq \mathcal A^h(\Omega_i)$. {From the theory of Caccioppoli sets, we have
\[
\begin{split}
    \mathcal H_g^{n-1}(\partial(\Omega_i\cap\mathcal N_{s_0}^c)\cap \mathring M)+&\mathcal H_g^{n-1}(\partial(\Omega_i\cup\mathcal N_{s_0}^c)\cap\mathring M)\\&\leq \mathcal H_g^{n-1}(\partial\Omega_i\cap\mathring M)+\mathcal H_g^{n-1}(\partial\mathcal N_{s_0}^c\cap \mathring M),
\end{split}
\]
where $\mathcal N_{s_0}^c$ denotes the complement of $\mathcal N_{s_0}$ in $M$.
By the definition of $\mathcal A^h$ we can compute
\[
\begin{split}
\mathcal A^h(\tilde \Omega_i)-&\mathcal A^h(\Omega_i)\leq\int_{\Omega_i\cap \mathcal N_{s_0}}\tilde h\,\mathrm d\mathcal H^n_{\tilde g}\\
&+\mathcal H^{n-1}_{\tilde g}(\partial\mathcal N_{s_0}^c\cap\mathring M)-\mathcal H^{n-1}_{\tilde g}(\partial (\Omega_i\cup \mathcal N_{s_0}^c)\cap \mathring M).
\end{split}
\]
Let $X$ be the unit outward normal vector field determined by $\partial \mathcal N_s\subset \mathcal N_s$. Then the difference of the two area terms in the second line is no greater than
$$
\int_{\partial\mathcal N_{s_0}^c\cap \mathring M}X\cdot X\,\mathrm d\mathcal H^{n-1}_{\tilde g}
+\int_{\partial (\Omega_i\cup \mathcal N_{s_0}^c)\cap \mathring M} X\cdot \nu\,\mathrm d\mathcal H^{n-1}_{\tilde g},
$$
where $\nu$ is the unit outward normal of $\partial (\Omega_i\cup \mathcal N_{s_0}^c)\subset \Omega_i\cup \mathcal N_{s_0}^c$.} It follows from the divergence theorem that the above terms equal
$$
\int_{\Omega_i\cap \mathcal N_{s_0}}\mathrm{div}\, X\,\mathrm d\mathcal H^n_{\tilde g}-\int_{\Omega_i\cap \mathcal N_{s_0}\cap \Gamma}X\cdot \mathbf n\,\mathrm d\mathcal H^{n-1}_{\tilde g},
$$
where $\mathbf n$ is the outward unit normal of $\Gamma\subset M$. Combining with the acute-angle condition we finally obtain
$$
\mathcal A^h(\tilde \Omega_i)-\mathcal A^h(\Omega_i)\leq \int_{\Omega_i\cap \mathcal N_{s_0}}(\tilde h+H_s)\,\mathrm d\mathcal H^n_{\tilde g}\leq 0.
$$
This completes the proof.

{Now we work with the smooth minimizer $\Omega_{\mathrm{min}}$. Denote $\Sigma=\overline{\partial\Omega_{\mathrm{min}}\setminus\Gamma}$. Then $\Sigma$ and $\partial_-$ bound a relative region $\Omega$ relative to $\Gamma$. With the warped-product interpretation as before, it follows from a similar computation as in \cite[(1.3)]{ZZ20} and \cite[Appendix]{RS97} (handling the interior and boundary variations respectively) that the Jacobi operator associated to $\mathcal A^h$ is
$$(\mathcal J^h,\mathcal B)=\left(-\Delta_{\Sigma\times \mathbb T^l}-(\mathrm{Ric}_{\hat g}(\nu)+|A|^2+\partial_\nu(h\circ \rho\circ\pi)),\frac{\partial}{\partial \vec n}-\mathrm{II}(\nu,\nu)\right),$$
where $\vec n$ denotes the outward pointing unit normal vector field on $\partial \Sigma\subset \Sigma$, $\nu$ denotes the outward pointing unit normal vector field on $\Sigma\subset \bar\Omega_{\mathrm{min}}$, $\mathrm{II}$ is the second fundamental form of $\partial M\times \mathbb T^l$ with respect to the outward pointing unit normal vector field. As in the proof of the previous lemma, we use $\mathcal J^h_*$ to denote the modified operator of $\mathcal J^h$ and take $v_{l+1}$ to be the first eigenfunction of the modified operator $\mathcal J^h_*$ with the Robin boundary condition $\mathcal Bv_{l+1}=0$. From a direct computation we have
\begin{equation}\label{Eq: scalar change 2}
     R(\hat g')= R(\hat g)|_{\Sigma\times \mathbb T^l}+\left(\frac{n+l}{n+l-1}h^2-2|h'|\right)\circ \rho >R(\hat g)|_{\Sigma\times \mathbb T^l}-\frac{\pi^2}{d_0^2},
 \end{equation} and
 \begin{equation}\label{Eq: mean curvature change}
     \begin{split}
     H_{\partial\Sigma\times \mathbb T^{l+1}}(\hat g')= {} & H_{\partial\Sigma\times \mathbb T^l}(\hat g)+v_{l+1}^{-1}\frac{\partial v_{l+1}}{\partial \vec n} \\
     = {} & H_{\partial\Sigma\times \mathbb T^l}(\hat g)+\mathrm{II}(\nu,\nu) \\[0.6mm]
     = {} & H_{\partial M\times \mathbb T^l}(\hat g), \\
     \end{split}
 \end{equation}
where $\hat g$ and $\hat g'$ denote the warped product metrics defined in \eqref{Eq: hat g} and \eqref{Eq: hat g'} respectively, and we use the relation
$$
H_{\partial\Sigma\times \mathbb T^l}(\hat g)=\mathrm{tr}^{\hat g}_{\partial\Sigma\times \mathbb T^l}\mathrm{II}
$$
coming from the fact that $\partial\Sigma\times \mathbb T^l$ has free boundary.
This completes the proof.
}
\end{proof}
\subsection{Inradius estimate and intrinsic diameter estimate}\label{Inradius estimate and intrinsic diameter estimate}
\begin{definition}\label{def: inradius}
Let $(M,\partial M,g)$ be a Riemannian manifold with non-empty boundary and $\Gamma$ be a smooth (possibly empty) portion of $\partial M$ such that $\partial M \neq \Gamma$. The inradius of $(M,\partial M,\Gamma)$ is defined to be
$$
r_{\mathrm{in}}(M,\partial M,\Gamma)=\sup_{p\in M}\dist_{g}(p,\partial M-\Gamma).
$$
If $\Gamma=\emptyset$, we write $r_{\mathrm{in}}(M,\partial M)$ instead of $r_{\mathrm{in}}(M,\partial M,\emptyset)$ for short.
\end{definition}

\begin{lemma}\label{Lem: inradius}
Let $(\Sigma,\partial\Sigma,g)$ be a surface or a curve with non-empty boundary and $\Gamma$ be a smooth (possibly empty) portion of $\partial\Sigma$ such that $\partial \Sigma \neq \Gamma$. If $(\Sigma,\Gamma)$ has $\mathbb T^l$-stabilized scalar-mean curvature $(R,H)$ with $R\geq \sigma_0>0$ and $H\geq 0$, then we have $r_{\mathrm{in}}(\Sigma,\partial\Sigma,\Gamma)\leq 2\pi/\sqrt{\sigma_0}$.
\end{lemma}
\begin{proof}
It suffices to prove the case when $\Sigma$ is a surface since we can consider $\Sigma\times \mathbb S^1$ instead of $\Sigma$ when $\Sigma$ is a curve.

We argue by contradiction. Suppose that there is a point $p$ such that $\dist(p,\partial\Sigma-\Gamma)>2\pi/\sqrt{\sigma_0}$. Without loss of generality we may assume $p$ to be an interior point. In particular, by smoothing the distance function we can take a connected region $\Omega$ containing $p$ such that
\begin{itemize}\setlength{\itemsep}{1mm}
\item $\partial\Omega$ is piecewisely smooth consisting of smooth curves $S':=\overline{\partial\Omega-\Gamma}$ and $\Gamma':=\Gamma\cap \partial\Omega$;
\item $S'$ intersects $\Gamma'$ transversely;
\item $\dist(p,S')>2\pi/\sqrt{\sigma_0}$.
\end{itemize}
Take a small geodesic ball $B_\delta$ centered at $p$ which is disjoint from $\partial\Omega$. Then $\Omega_\delta:=\Omega-B_\delta$ is a Riemannian band with $\partial_-=\partial B_\delta$ and $\partial_+=S'$. Notice that $(\Omega_\delta,\Gamma')$ has $\mathbb T^l$-stabilized scalar curvature $(R,H)$. If $S'$ intersects $\Gamma'$ in acute angles, then we can apply Lemma \ref{Lem: corner mu bubble} to find a curve $\gamma$ such that $(\gamma,\partial\gamma)$ has $\mathbb T^{l+1}$-stabilized scalar curvature $(R',H')$ where
$$
R'>R-\pi^2\left(\frac{\pi}{\sqrt{\sigma_0}}\right)^{-2}\geq 0\mbox{ and }H'=H\geq 0.
$$
This is impossible since by doubling trick {\cite[Theorem 5.7]{GL80a}} one can construct a smooth metric on $\mathbb T^{l+2}$ with positive scalar curvature and this leads to a contradiction.

Generally $S'$ does not intersect $\Gamma'$ in acute angles, so we have to find a deformation of $S'$ which produces acute dihedral angles but does not affect much on the distance between $S'$ and $p$. This can be done as follows. Given any tubular neighborhood $\mathcal N$ of $S'$ it is not difficult to construct a smooth vector field $X$ on $\Omega$ supported in $\mathcal N$ which satisfies
\begin{itemize}\setlength{\itemsep}{1mm}
\item $X$ is tangential to $\Gamma'$;
\item $X$ is transverse to $S'$ and inward-pointing.
\end{itemize}
From this vector field we can construct a diffeomorphism
$$
\Phi:S'\times [0,1] \to \mathcal N,\ \, (q,t)\mapsto \phi_t(q),
$$
where $\phi_t$ is the flow generated by $X$. {With the help of the map $\Phi$ we can consider graphs over $S'$. Namely, given any smooth function $f:S'\to [0,1]$ over $S'$ we can consider
$$
G_f:= \big\{\Phi(q,f(q)) \mid q\in S'\big\}.
$$
Let $f$ be a smooth function on $S'$ with $0\leq f\leq 1$ in $S'$ and
$$
f = c \cdot \dist_{S'}(\cdot,\partial S')  \mbox{ around }\partial S',
$$
where $c$ is a constant to be determined later. Let $\nu_{S'}$ denote the outward pointing unit normal vector field on $S'\subset \bar\Omega$, $\nu$ denote the unit normal vector field on $G_f$ pointing to $S'$, $\vec n$ denote the outward pointing unit normal vector field on $\Gamma'\subset \bar\Omega$, and $\beta$ denote the inward pointing conormal of $\partial S'\subset S'$. From a direct  computation we have
$$
\text{$\nu=\nu_{S'}\cos\theta+\beta\sin\theta$ where $\theta\in\left[0,\frac{\pi}{2}\right)$ with $\tan\theta=c$.}
$$
Note that we have $\beta\cdot \vec n<0$ since $S'$ intersects $\Gamma'$ transversally, $\beta$ is inward pointing and $\vec n$ is outward pointing. By taking the constant $c$ large enough and correspondingly $\theta$ sufficiently close to $\pi/2$, we can guarantee $\nu\cdot \vec n<0$ and this means that $G_f$ intersects $\Gamma'$ in acute angles. Clearly, we can require the $C^0$-norm of $f$ as small as possible and so the distance between $G_f$ and $p$ can be arbitrarily close to that between $S'$ and $p$.
In particular, we can still guarantee $\dist(p,G_f)>2\pi/\sqrt{\sigma_0}$. This completes the proof.}
\end{proof}

\begin{corollary}\label{Cor: diameter}
Let $(\Sigma,\partial\Sigma,g)$ be an orientable surface possibly with non-empty boundary. Assume that $(\Sigma,\partial\Sigma)$ has $\mathbb T^l$-stabilized scalar-mean curvature $(R,H)$ with $R\geq \sigma_0>0$ and $H\geq 0$.
Then $\Sigma$ is a topological sphere or disk and we have $\diam(\Sigma,g)\leq 2\pi/\sqrt{\sigma_0}$.
\end{corollary}
\begin{proof}
It follows from \cite[formula (2.6)]{GZ21} that
\[
\begin{split}
0<&\int_{\Sigma} R\,\mathrm d\sigma_g+2\int_{\partial\Sigma}H\,\mathrm ds_g \leq  \int_{\Sigma} R_g\,\mathrm d\sigma_g+2\int_{\partial\Sigma}\kappa_g\,\mathrm ds_g=4\pi\chi(\Sigma).
\end{split}
\]
This implies that $\Sigma$ is a topological sphere or disk.

Next we pick two points $p$ and $q$ with $\dist(p,q)=\diam(\Sigma,g)$. For any small constant $\ve>0$ we can find an interior point $p_\ve$ with $\dist(p_\ve,p)<\ve$. Take a small geodesic ball $B_\delta$ centered at $p_\ve$ which is disjoint from $\partial\Sigma$ and $q$. From Lemma \ref{Lem: inradius} we see
$$r_{\mathrm{in}}(\Sigma-B_\delta, \partial(\Sigma-B_\delta),\partial\Sigma)\leq 2\pi/\sqrt{\sigma_0}.$$ In particular, we have $\dist(q,\partial B_\delta)\leq 2\pi/\sqrt{\sigma_0}$ and so
$$
\dist(p,q)\leq \dist(p_\ve,q)+\ve\leq 2\pi/\sqrt{\sigma_0}+\delta+\ve.
$$
Notice that we can make $\ve$ and $\delta$ arbitrarily small and so we can obtain the desired estimate.
\end{proof}

\section{Proof of Theorem \ref{Thm: main}}\label{Proof of main theorem}
We adopt the same notation as in Section \ref{Preliminaries}. Let $3 \le n \le 7$.

\subsection{Dimension reduction}
Recall that $\tilde{\sigma}$ is a proper line contained in $\tilde{N}_\ve$ constructed in Subsection \ref{subsec: line}.
\begin{lemma}
For any $L>0$ we can find a compact two-sided hypersurface $\tilde M_{n-1}$ with non-empty boundary in $(\tilde Y,\tilde g)$ such that
\begin{itemize}\setlength{\itemsep}{1mm}
\item we have
$$
\dist_{\tilde g}(\partial \tilde M_{n-1},\tilde \sigma)\geq 3L;
$$
\item $\tilde M_{n-1}$ has non-zero algebraic intersection with the line $\tilde \sigma$.
\end{itemize}
\end{lemma}

\begin{proof}
We follow the argument of Chodosh--Li \cite[Section 2]{CL20}. Fix two smooth functions $\rho_{1},\rho_{2}:\ti{Y}\rightarrow\mathbb{R}$ satisfying
\[
|\rho_{1}(\cdot)-\dist_{\ti{g}}(\cdot,\ti{\sigma}([0,+\infty)))| \leq 1, \quad
|\rho_{2}(\cdot)-\dist_{\ti{g}}(\cdot,\ti{\sigma}(0))| \leq 1.
\]
{Here $|\cdot|$ denotes the $C^{0}$-norm.} Let $L_{1}\geq 3L+3$ be a large regular value of $\rho_{1}$ and define $U:=\rho_{1}^{-1}((-\infty,L_{1}])$. The construction can be divided into four steps.

\medskip
\noindent
{\bf Step 1.} For $L_{2}\gg1$, $\ti{\sigma}\cap\partial U\subset\rho_{2}^{-1}((-\infty,L_{2}])$.
\medskip

By the definition of $U$, it is clear that $\ti{\sigma}([0,+\infty))\subset U$. We claim that for $K\gg1$,
\begin{equation}\label{K U emptyset}
\ti{\sigma}((-\infty,-K)) \cap U = \emptyset.
\end{equation}
Given this claim, $\ti{\sigma}\cap\partial U$ is contained in the compact set $\ti{\sigma}([-K,0])$, which implies the existence of $L_{2}$. To prove \eqref{K U emptyset}, we argue by contradiction. Suppose that $\ti{\sigma}(t_{1})\in U$ for some $t_{1}<-K$. Then
\[
\dist_{\ti{g}}(\ti{\sigma}(t_{1}),\ti{\sigma}([0,+\infty))) \leq \rho_{1}(\tilde\sigma(t_{1}))+1 \leq L_{1}+1,
\]
and so there exists $t_{2}\geq0$ such that
\[
\dist_{\ti{g}}(\ti{\sigma}(t_{1}),\ti{\sigma}(t_{2})) \leq L_{1}+1.
\]
Since $|t_{1}-t_{2}|=t_{2}-t_{1} \geq K$. Using Lemma \ref{geodesic} (iii), when $K\gg1$, we have $\dist_{\ti{g}}(\ti{\sigma}(t_{1}),\ti{\sigma}(t_{2})) \geq L_{1}+2$, which is a contradiction.

\medskip
\noindent
{\bf Step 2.} Construct $\ti{M}_{n-1}$.
\medskip

Let $L_{2}\gg L_{1}$ be a regular value of $\rho_{2}|_{\partial U}$ and define
\[
\ti{M}_{n-1} := (\partial U) \cap \rho_{2}^{-1}((-\infty,L_{2})).
\]
Perturb $\ti{\sigma}$ slightly such that it intersects $\ti{M}_{n-1}$ transversely, $\ti{\sigma}\cap\partial U=\ti{\sigma}\cap\ti{M}_{n-1}$ and the following inequalities hold:
\[
|\rho_{1}(\cdot)-\dist_{\ti{g}}(\cdot,\ti{\sigma}([0,+\infty)))| \leq 2, \quad
|\rho_{2}(\cdot)-\dist_{\ti{g}}(\cdot,\ti{\sigma}(0))| \leq 2.
\]

\begin{figure}[htbp]
\centering
\includegraphics[width=\linewidth]{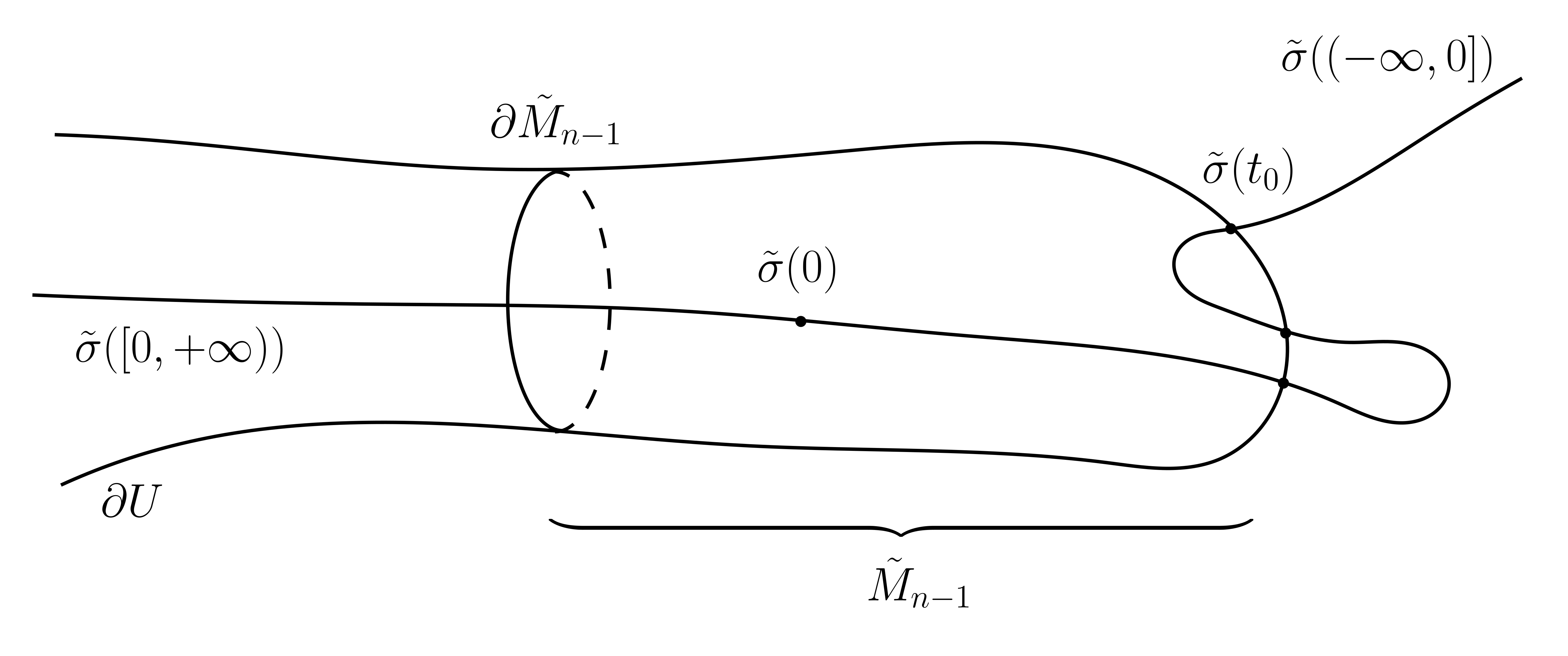}
\caption{The hypersurface $\ti{M}_{n-1}$}
\label{construction_of_M}
\end{figure}

\medskip
\noindent
{\bf Step 3.} For $L_{2}\gg1$, the curve $\ti{\sigma}$ has non-zero algebraic intersection with $\ti{M}_{n-1}$ and $\partial\ti{M}_{n-1}\neq\emptyset$.
\medskip

As shown in Figure \ref{construction_of_M}, let $t_{0}$ be the smallest intersection time of $\ti{\sigma}$ and $\partial U$, i.e.
\[
t_{0} := \min\{t\in\mathbb{R}\mid\ti{\sigma}(t)\in\partial U \}.
\]
Then $\ti{\sigma}|_{(-\infty,t_{0})}$ does not intersect $\partial U$, and $\ti{\sigma}|_{(t_{0},+\infty)}$ leaves and re-enters $U$ in oppositely oriented pairs. Together with $\ti{\sigma}\cap\partial U=\ti{\sigma}\cap\ti{M}_{n-1}$, the curve $\ti{\sigma}$ has non-zero algebraic intersection with $\ti{M}_{n-1}$. If $\partial\ti{M}_{n-1}=\emptyset$, the above shows
\[
0 \neq [\ti{M}_{n-1}] \in H_{n-1}(\ti{Y},\ti{X}_\ve),
\]
contradicting $H_{n-1}(\ti{Y},\ti{X}_\ve) = 0$ by Lemma \ref{Lem: excision}. Then we obtain $\partial\ti{M}_{n-1}\neq\emptyset$.

\medskip
\noindent
{\bf Step 4.} For $L_{2}\gg L_{1}$, $\dist_{\ti{g}}(\partial\ti{M}_{n-1},\ti{\sigma})\geq 3L$.
\medskip

Suppose that $\dist_{\ti{g}}(\partial\ti{M}_{n-1},\ti{\sigma})$ is achieved at $p\in\partial\ti{M}_{n-1}$ and $t_{2}\in\mathbb{R}$, i.e.
\[
\dist_{\ti{g}}(p,\ti{\sigma}(t_{2})) = \dist_{\ti{g}}(\partial\ti{M}_{n-1},\ti{\sigma}).
\]
Recalling $L_{1}\geq 3L+3$, it suffices to show that for $L_{2}\gg L_{1}$,
\[
\dist_{\ti{g}}(p,\ti{\sigma}(t_{2}))\geq L_{1}-3.
\]
We argue by contradiction. Suppose that $\dist_{\ti{g}}(p,\ti{\sigma}(t_{2}))<L_{1}-3$. If $t_{2}\geq0$, then
\[
L_{1}-3 > \dist_{\ti{g}}(p,\ti{\sigma}(t_{2}))
\geq \dist_{\ti{g}}(p,\ti{\sigma}([0,+\infty)))
\geq \rho_{1}(p)-2 = L_{1}-2,
\]
which is a contradiction. If $t_{2}<0$, by Lemma \ref{geodesic} (ii), we have
\begin{equation}\label{A t 2}
\begin{split}
A|t_{2}| \geq {} & \dist_{\ti{g}}(\ti{\sigma}(0),\ti{\sigma}(t_{2}))
\geq \dist_{\ti{g}}(p,\ti{\sigma}(0))-\dist_{\ti{g}}(p,\ti{\sigma}(t_{2})) \\
> {} & (\rho_{2}(p)-2)-(L_{1}-3) = L_{2}-L_{1}+1.
\end{split}
\end{equation}
Let $t_{1}\geq0$ be the number such that
\[
\dist_{\ti{g}}(p,\ti{\sigma}(t_{1})) = \dist_{\ti{g}}(p,\ti{\sigma}([0,+\infty))).
\]
Then
\begin{equation}\label{d p sigma t 2}
\begin{split}
\dist_{\ti{g}}(p,\ti{\sigma}(t_{2}))
\geq {} & \dist_{\ti{g}}(\ti{\sigma}(t_{1}),\ti{\sigma}(t_{2}))-\dist_{\ti{g}}(p,\ti{\sigma}(t_{1})) \\
= {} & \dist_{\ti{g}}(\ti{\sigma}(t_{1}),\ti{\sigma}(t_{2}))-\dist_{\ti{g}}(p,\ti{\sigma}([0,+\infty))) \\
\geq {} & \dist_{\ti{g}}(\ti{\sigma}(t_{1}),\ti{\sigma}(t_{2}))-\rho_{1}(p)-2 \\
= {} & \dist_{\ti{g}}(\ti{\sigma}(t_{1}),\ti{\sigma}(t_{2}))-L_{1}-2.
\end{split}
\end{equation}
Using \eqref{A t 2}, $t_{1}\geq0$ and $t_{2}<0$,
\[
|t_{1}-t_{2}| \geq |t_{2}| \geq A^{-1}(L_{2}-L_{1}+1).
\]
Combining this with Lemma \ref{geodesic} (iii), for $L_{2}\gg L_{1}$, we have $\dist_{\ti{g}}(\ti{\sigma}(t_{1}),\ti{\sigma}(t_{2}))\geq 2L_{1}$. Together with \eqref{d p sigma t 2}, we obtain $\dist_{\ti{g}}(p,\ti{\sigma}(t_{2}))\geq L_{1}-2$, which contradicts $\dist_{\ti{g}}(p,\ti{\sigma}(t_{2}))<L_{1}-3$.
\end{proof}

\begin{lemma}\label{Lem: Mn-1}
We can find a smoothly embedded, complete hypersurface $M_{n-1}$ in $(\tilde Y,\tilde g)$ with $\partial M_{n-1}=\partial\tilde M_{n-1}$ such that
\begin{itemize}\setlength{\itemsep}{1mm}
\item $M_{n-1}\cap \tilde N_\ve$ is compact;
\item $\mathring M_{n-1}$ with the induced metric has $\mathbb T^1$-stabilized scalar curvature $R_{n-1}$ which is no less than $R(\tilde g)|_{M_{n-1}}$;
\item Every component of $M_{n-1}$ has nonempty boundary.
\end{itemize}
\end{lemma}
\begin{proof}
From Lemma \ref{Lem: plateau} we can find the desired hypersurface $M_{n-1}$, where the only thing we need to verify is that $M_{n-1}\cap \tilde N_\ve$ is compact. Recall from the proof of Lemma \ref{Lem: plateau} {(when $l=0$)} that $M_{n-1}$ has finite area {in $(\ti{Y},\ti{g})$}. Since $(\tilde N_\ve,\ti{g})$ is the Riemannian covering of the compact Riemannian manifold $(N_\ve,g)$, it has bounded geometry.  So we can conclude the compactness of $M_{n-1}\cap\tilde N_\ve$ from the monotonicity formula.
\end{proof}

\begin{lemma}\label{Lem: Mn-2}
Given any positive constant $\mu_{\mathrm{loss}}$, there is a universal constant $L_0=L_0(\mu_{\mathrm{loss}})$ such that for any constant $L> L_0$, we can find a non-empty closed hypersurface $M_{n-2}\subset M_{n-1}$ such that
\begin{itemize}\setlength{\itemsep}{1mm}
\item $M_{n-2}$ and $\partial M_{n-1}$ enclose a bounded region $\Omega_{n-1}$ in $M_{n-1}$ which is disjoint from $\tilde\sigma$;
\item
$
\dist_{\tilde g}(M_{n-2},\tilde\sigma)\geq L
$;
\item $M_{n-2}$ has $\mathbb T^2$-stabilized scalar curvature $R_{n-2}$ with
$$R_{n-2}\geq R(\tilde g)|_{M_{n-2}}-\mu_{\mathrm{loss}}.$$
\end{itemize}
\end{lemma}

\begin{proof}
Set
$$
L_0=\frac{2\pi}{\sqrt{\mu_{\mathrm{loss}}}}+7.
$$
We claim that $M_{n-1}$ is not contained in the $L_0$-neighborhood of $\partial M_{n-1}$. Otherwise, using the fact
$
\dist_{\tilde g}(\partial M_{n-1},\tilde \sigma)\geq 3L$,
we see that $$\dist_{\tilde g}(M_{n-1},\tilde \sigma)>0$$ when $L>L_0$, which means that $M_{n-1}$ is a compact hypersurface bounding $\partial M_{n-1}$ and is disjoint from $\tilde \sigma$. However, we know from Lemma \ref{Lem: excision} that $M_{n-1}$ and $\tilde M_{n-1}$ should have the same algebraic intersection number with $\sigma$, which is nonzero. This is impossible.

So we just need to handle the case when $M_{n-1}$ does not lie entirely in the $L_0$-neighborhood of $\partial M_{n-1}$. Let $\rho:\tilde Y\to \mathbb R$ be a smooth function such that
$$
|\rho(\cdot)-\dist_{\tilde g}(\cdot,\partial M_{n-1})|\leq 1.
$$
Fix two regular values $c_1\in(1,2)$ and $c_2\in(L_0-2,L_0-1)$ of the function $\rho|_{M_{n-1}}$. For convenience, we use $\rho$ to denote $\rho|_{M_{n-1}}$ in the following argument. From our assumption the image of $\rho$ contains the interval $[1,L_0-1]$ and so
$$
M_{n-1}\cap  \rho^{-1}([c_1,c_2])
$$
is a smooth Riemannian {manifold with boundary. For each component $V$, we set $\partial_-=\partial V\cap\rho^{-1}(c_1)$ and $\partial_+=\partial V\cap\rho^{-1}(c_2)$. Define
\[
\mathcal{S}_{1} = \{V \mid \partial_-\neq\emptyset, \ \partial_+\neq\emptyset\}, \ \
\mathcal{S}_{2} = \{V \mid \partial_-\neq\emptyset, \ \partial_+=\emptyset \}.
\]
For $V\in\mathcal{S}_{1}$, $(V,\partial_{+},\partial_{-})$ is a Riemannian band.} We claim
$$
\dist_{\tilde g}(\partial_+,\partial_-)>\frac{2\pi}{\sqrt{\mu_{\mathrm{loss}}}}.
$$
Otherwise, there are two points $p\in \partial_+$ and $q\in \partial_-$ satisfying $\dist_{\tilde g}(p,q)\leq 2\pi/\sqrt{\mu_{\mathrm{loss}}}$. But on the other hand we have
$$\dist_{\tilde g}(p,\partial M_{n-1})>\frac{2\pi}{\sqrt{\mu_{\mathrm{loss}}}}+4 \, \mbox{ and }  \dist_{\tilde g}(q,\partial M_{n-1})<3,$$
which contradicts the triangle inequality. Since the intrinsic distance is greater than the extrinsic distance, we obtain $\dist_V(\partial_+,\partial_-)>2\pi/\sqrt{\mu_{\mathrm{loss}}}$. From Lemma \ref{Lem: smooth mu bubble} we can construct a hypersurface $M_{n-2}^{V}$ bounding a region $\Omega_{V}$ with $\partial_-$. Then $M_{n-2}^{V}$ has $\mathbb T^2$-stabilized scalar curvature $R_{n-2}$ no less than
$$
R(\tilde g)|_{M_{n-2}^{V}}-\mu_{\mathrm{loss}}.
$$
{For $V\in\mathcal{S}_{2}$, we use $\Omega_{V}'$ to denote the interior of $V$. Set}
\[
M_{n-2} = \bigcup_{V\in\mathcal{S}_{1}}M_{n-2}^{V}, \ \
\Omega = \left(\bigcup_{V\in\mathcal{S}_{1}}\Omega_{V}\right)\cup\left(\bigcup_{V\in\mathcal{S}_{2}}\Omega_{V}'\right).
\]
Clearly we see that $M_{n-2}$ and $\partial M_{n-1}$ bounds the region
$$
\Omega_{n-1}=\Omega\cup(M_{n-1}\cap {\{\rho\leq c_{1}\}}).
$$
Both $\Omega_{n-1}$ and $M_{n-2}$ are contained in  the $L_0$-neighborhood of $M_{n-1}$. The first two properties required in this proposition follow from the facts $L>L_0$ and $\dist_{\tilde g}(\partial M_{n-1},\tilde \sigma)\geq 3L$.
\end{proof}

\subsection{Avoidance of two-ends touching}\label{Avoidance of two-ends touching}
As explained in Section \ref{Introduction} (see explanation after Theorem \ref{Thm: mapping rigidity}), we hope that the stable minimal surfaces $\Sigma$ or stable $\mu$-bubbles with boundary $(\Sigma,\partial\Sigma)$ can only lie in at most one end. To show this, we introduce some specific hypersurfaces in each end. Since we hope that $\Sigma$ and $(\Sigma,\partial\Sigma)$ avoid almost all such hypersurfaces, we refer to them as avoidance hypersurfaces.

\subsubsection{Definition of avoidance hypersurfaces}\label{Definition of avoidance hypersurfaces}
Denote
$$\underline R_1=\min_{Y_{\Lambda_1}} R(g).$$
From Lemma \ref{Lem: weight function 1} (with the choice $\nu_0=\underline R_1$ and $d_0=d_{\Lambda_1}$) we can find a positive constant $\delta_1$, some constant $T_1>d_{\Lambda_1}$, and a smooth function $h_1:(-T_1,T_1)\to \mathbb R$ such that
\begin{itemize}\setlength{\itemsep}{1mm}
\item $\lim_{t\to\pm T_1}h_1(t)=\mp\infty$ and $h_1'<0$;
\item $h_1^2-2|h_1'|+\underline R_1\chi_{[-d_{\Lambda_1},d_{\Lambda_1}]}\geq \delta_1$.
\end{itemize}
By smoothing out a level set of the distance function $\dist_g(\cdot, S^h_{\ve}(p))$, we can take a smooth closed hypersurface $\Lambda_2 \subset X_{\Lambda_1}$ homologous to $\Lambda_1$ such that
\begin{equation}\label{construction Lambda 2}
X_{\Lambda_2}\subset \rho_{\Lambda_1}^{-1}((T_1,+\infty)),
\end{equation}
where $\rho_{\Lambda_1}$ is the function from Lemma \ref{Lem: rho lambda}.

Let $\rho_{\Lambda_2}$ be the function from Lemma \ref{Lem: rho lambda} and fix an arbitrary positive constant $d_{\Lambda_2}$. Similarly, we can take a smooth closed hypersurface $\Lambda_3 \subset X_{\Lambda_2}$ homologous to $\Lambda_2$ such that
$$
X_{\Lambda_3}\subset \rho_{\Lambda_2}^{-1}((d_{\Lambda_2},+\infty)).
$$
Denote
$$
\underline R_2=\min_{Y_{\Lambda_3}} R(g).
$$
It follows from Lemma \ref{Lem: weight function 2} (with the choice $\nu_0=\underline R_2$ and $d_0=d_{\Lambda_2}$) that we can find a positive constant $\delta_2$, some constant $T_2>d_{\Lambda_2}$, and a smooth function $h_2:[0,T_2)\to (-\infty,0]$ such that
\begin{itemize}\setlength{\itemsep}{1mm}
\item $h_2\equiv 0$ around $t=0$, $\lim_{t\to T_2}h_2(t)=-\infty$ and $h_2'\leq 0$;
\item $h_2^2-2|h_2'|+\underline R_2\chi_{[0,d_{\Lambda_2}]}\geq \delta_2$.
\end{itemize}
By smoothing out a level set of the distance function $\dist_g(\cdot, S^h_{\ve}(p))$, we can take a smooth closed hypersurface $\Lambda_4 \subset X_{\Lambda_3}$ homologous to $\Lambda_3$ such that
\begin{equation}\label{construction Lambda 4}
X_{\Lambda_4}\subset \rho_{\Lambda_2}^{-1}((T_2,+\infty)).
\end{equation}
See Figure \ref{avoidance_hypersurface} for our setting of all auxiliary hypersurfaces $\Lambda_i$.

In the following, we will let either the hypersurfaces $\tilde{\Lambda}_{2,i}$ or the hypersurfaces $\tilde{\Lambda}_{4,i}$ be the avoidance hypersurfaces.

\begin{figure}[htbp]
\centering
\includegraphics[width=\linewidth]{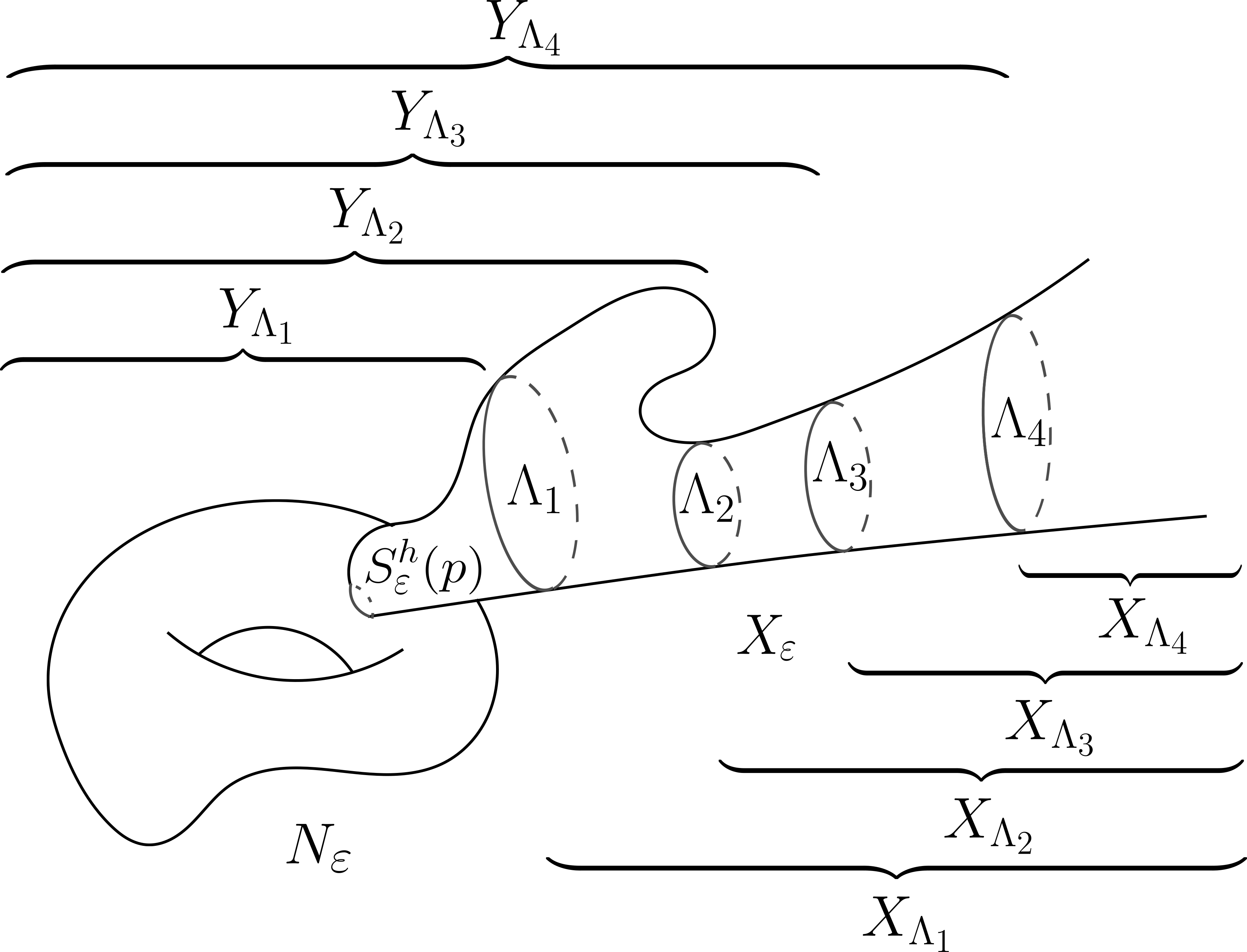}
\caption{Auxiliary hypersurfaces}
\label{avoidance_hypersurface}
\end{figure}

\subsubsection{{Intersection with} at most one avoidance hypersurface}
Setting $\tilde{\Lambda}_{2,i}$ as our avoidance hypersurfaces, we have
\begin{proposition}\label{Prop: closed}
Let $\Sigma$ be a smoothly embedded, connected closed surface or curve in $(\tilde Y,\tilde g)$ with the induced metric $g_\Sigma$. Suppose that $\Sigma$ has $\mathbb T^l$-stabilized scalar curvature $R$ which is greater than $R(\tilde g)|_{\Sigma}-\delta_1$, where $\delta_1$ is the constant in Subsection \ref{Definition of avoidance hypersurfaces}.
Then either we have
$$\Sigma\cap\tilde X_{\Lambda_{2}}=\emptyset$$
or there is a unique index $i_0$ such that
$$\Sigma\cap \tilde X_{\Lambda_{2}}\subset \tilde X_{\Lambda_2,i_0}.$$
\end{proposition}
\begin{proof}
It suffices to rule out the possibility that $\Sigma$ intersects with more than one components of $\tilde X_{\Lambda_2}$. Suppose by contradiction that $\Sigma$ intersects with $\tilde X_{\Lambda_2,i_1}$ and $\tilde X_{\Lambda_2,i_2}$ where $i_1\neq i_2$. Let $\tilde \rho_{i_1}$ be the short function from Lemma \ref{Lem: rho_i}. From \eqref{construction Lambda 2} we see $\tilde \rho_{i_1}<-T_1$ in $\tilde X_{\Lambda_2,i_1}$ and $\tilde \rho_{i_1}>T_1$ in $\tilde X_{\Lambda_2,i_2}$. Combining this with our assumption we see that $\tilde \rho_{i_1}(\Sigma)$ contains a neighborhood of the interval $[-T_1,T_1]$. This allows us to pick a value $c>T_1$ such that
$$
\Sigma\cap \tilde \rho_{i_1}^{-1}([-c,c])
$$
is a smooth Riemannian manifold with boundary. Since $\Sigma$ is connected, there exists a component $V$ such that $\partial_{\pm}:=\partial V\cap\tilde \rho_{i_1}^{-1}(\pm c)$ are both non-empty. Then $(V,\partial_{-},\partial_{+})$ is a Riemannian band. We take
$$
h:\mathring V\to \mathbb R,\quad h(x)=h_1\left(c^{-1}T_1 \cdot \tilde \rho_{i_1}(x)\right),
$$
where $h_{1}$ is the function in Subsection \ref{Definition of avoidance hypersurfaces}. We divide the discussion into the following two cases.

\medskip
{\it Case 1. $\Sigma$ is a surface.} {It follows from the proof of Lemma \ref{Lem: smooth mu bubble} and \eqref{Eq: scalar change} that}
 we can find a closed curve $\gamma$ in $V$ such that $\gamma$ has $\mathbb T^{l+1}$-stabilized scalar curvature $R'$ satisfying
$$
R'\geq R+(h_1^2-2|h_1'|)\circ {(c^{-1}T_1 \cdot \tilde \rho_{i_1})}
> R(\tilde g)|_\gamma+(h_1^2-2|h_1'|)\circ {(c^{-1}T_1 \cdot \tilde \rho_{i_1})}-\delta_1.
$$
Recall that we have $\tilde \rho_{i_1}^{-1}([-d_{\Lambda_1},d_{\Lambda_1}])\subset \tilde Y_{\Lambda_1}$. This implies
$$
R(\tilde g)|_\gamma+(h_1^2-2|h_1'|)\circ {(c^{-1}T_1 \cdot \tilde \rho_{i_1})}
\geq \left(\underline R_1\chi_{[-d_{\Lambda_1},d_{\Lambda_1}]}+h_1^2-2|h_1'|\right)\circ {(c^{-1}T_1 \cdot \tilde \rho_{i_1})} \geq \delta_{1}.
$$
So we obtain $R'>0$ but this contradicts the fact that $\mathbb T^n$ cannot admit any smooth metric with positive scalar curvature.

\medskip
{\it Case 2. $\Sigma$ is a curve.} In this case, we must have $l\geq 1$. Then we can interpret the situation as the case that $V\times \mathbb S^1$ has $\mathbb T^{l-1}$-stabilized scalar curvature $R$ which is greater than ${R(\tilde g)|_\Sigma}-\delta_1$ and then deduce a contradiction as in the surface case.
\end{proof}

An analogous result holds for surfaces with boundary, where we use  $\tilde{\Lambda}_{4,i}$ as the avoidance hypersurfaces:
\begin{proposition}\label{Prop: compact}
Let $(\Sigma,\partial\Sigma)$ be a smoothly embedded, connected surface with boundary in $(\tilde Y,\tilde g)$ with the induced metric $g_\Sigma$. Suppose that $(\Sigma,\partial\Sigma)$ has $\mathbb T^l$-stabilized scalar-mean curvature $(R,H)$ such that $R\geq R(\tilde g)|_{\Sigma}-\delta_2$ and $H\geq 0$, where $\delta_{2}$ is the constant in Subsection \ref{Definition of avoidance hypersurfaces}. Then either we have
$$\Sigma\cap \tilde X_{\Lambda_{4}}=\emptyset$$
or there is a unique index $i_0$ such that
$$\Sigma\cap\tilde X_{\Lambda_{4}}\subset \tilde X_{\Lambda_4,i_0}$$
and that each component of $\partial\Sigma$ has non-empty intersection with $\tilde X_{\Lambda_2,i_0}$.
\end{proposition}
\begin{proof}
The proof for $\Sigma$ not touching two components of $\tilde X_{\Lambda_4}$ follows the same argument as in the proof of Proposition \ref{Prop: closed}. It remains to show that if $\Sigma$ intersects some component $\tilde X_{\Lambda_4,i_0}$ then all the boundary components of $\Sigma$ have to intersect $\tilde X_{\Lambda_2,i_0}$. Suppose by contradiction that there is a boundary component disjoint from $\tilde X_{\Lambda_2,i_0}$, which we denote by $\partial_-$. We take
\[
\tilde \rho =
\begin{cases}
\ \rho_{\Lambda_{2}}\circ\pi & \mbox{ in } \ti X_{\Lambda_2,i_{0}};\\[1mm]
\ \quad \ 0 & \mbox{ in } \ti{Y} - \ti X_{\Lambda_2,i_{0}},
\end{cases}
\]
where $\rho_{\Lambda_{2}}$ is the function from Lemma \ref{Lem: rho lambda}.

Since $\Sigma$ touches $\tilde X_{\Lambda_4,i_0}$, \eqref{construction Lambda 4} shows that the image of $\tilde\rho|_\Sigma$ must contain $[0,c]$ for some $c>T_2$. Without loss of generality we may assume $c$ to be a regular value of both $\tilde \rho|_\Sigma$ and $\tilde\rho|_{\partial\Sigma}$, and so $\tilde\rho^{-1}(c)\cap \Sigma$ is a smooth curve in $\Sigma$ which intersects $\partial\Sigma$ transversely. Take
$$
\mathcal N=\tilde\rho^{-1}\left(\left(\frac{T_2+c}{2},c\right]\right)\cap \Sigma.
$$
From the proof of Lemma \ref{Lem: inradius} we can deform $\tilde\rho^{-1}(c)\cap \Sigma$ a little bit to find a curve $\partial_{+}$ contained in $\mathcal N$ which intersects $\partial\Sigma$ in acute angles and $\ti{\rho}|_{\partial_{+}}>T_{2}$. We use $V$ to denote {the component of the} relative region on $\Sigma$ bounded by $\partial_+$ and $\partial_-$ relative to $\partial\Sigma$ (i.e. $V\subset\Sigma$ and $\partial V-(\partial_{+}\cap\partial_{-})\subset\partial\Sigma$, see Figure \ref{region_V}) {such that $\partial_{-}\subset\partial V$}. Notice that $\partial_+$ is contained in $\mathcal N$ where $\tilde \rho$-values are greater than $T_2$, so we can modify the function $\tilde \rho|_V$ to a new smooth function $\rho:V\to[0,T_2]$ (the explicit construction of $\rho$ will be given in Appendix \ref{construction of rho}) such that
\begin{itemize}\setlength{\itemsep}{1mm}
\item $\rho=0$ on $\partial_-$ and $\rho^{-1}(T_2)=\partial_+$;
\item $\rho^{-1}([0,d_{\Lambda_2}])=(\tilde\rho|_{V})^{-1}([0,d_{\Lambda_2}])$;
\item $\Lip\rho<1$.
\end{itemize}
\begin{figure}[htbp]
\centering
\includegraphics[width=\linewidth]{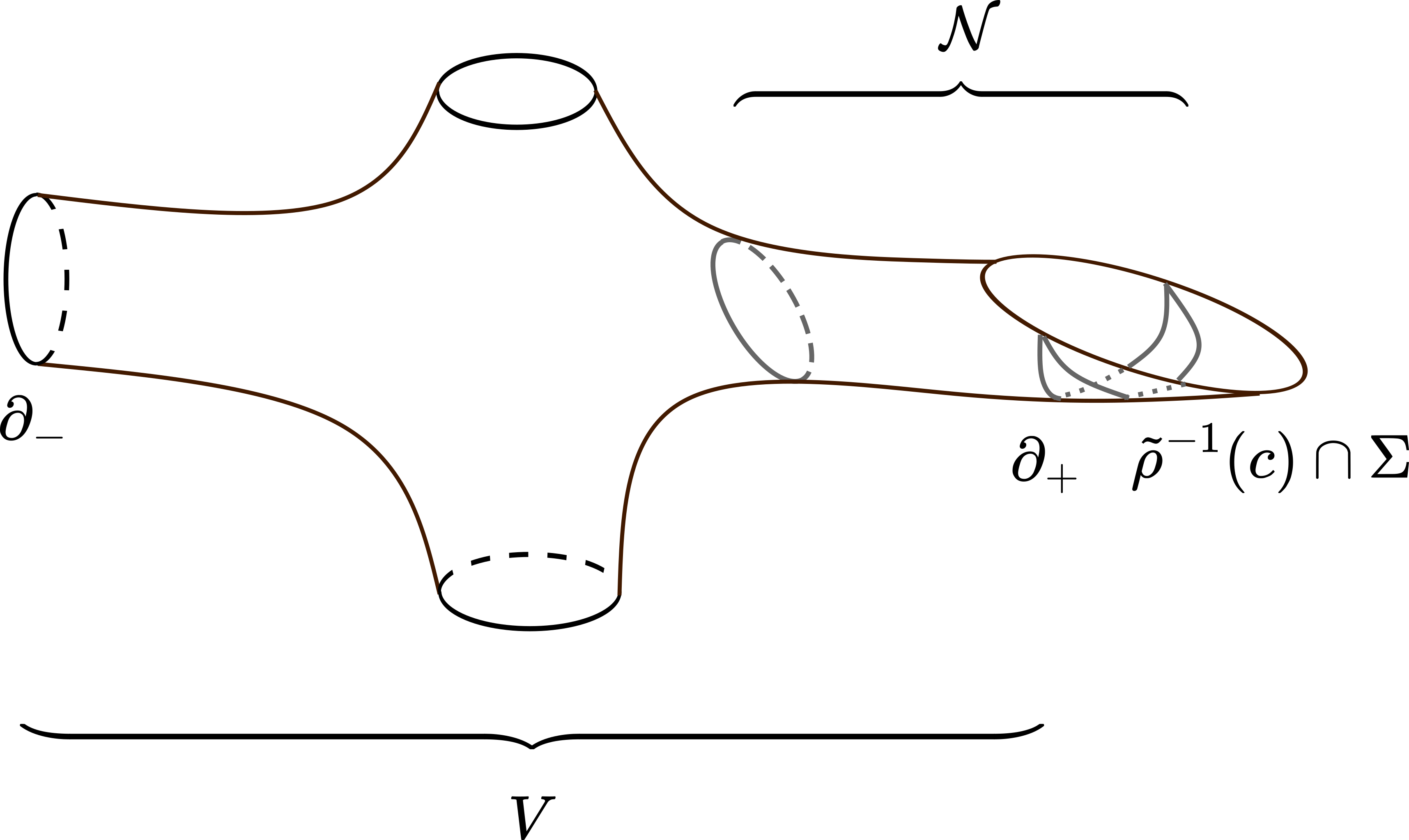}
\caption{The region $V$}
\label{region_V}
\end{figure}
Take
$
h=h_2\circ \rho
$, where $h_{2}$ is the function in Subsection \ref{Definition of avoidance hypersurfaces}.
{It follows from the proof of Lemma \ref{Lem: corner mu bubble} and \eqref{Eq: scalar change 2}-\eqref{Eq: mean curvature change} that} we can find a curve $\gamma$ in $V$ such that $(\gamma,\partial\gamma)$ has $\mathbb T^{l+1}$-stabilized scalar-mean curvature $(R',H')$ satisfying $H'=H\geq 0$ and
$$
R'\geq R+(h_2^2-2|h_2'|)\circ \rho>R(\tilde g)|_\gamma+(h_2^2-2|h_2'|)\circ \rho-\delta_2.
$$
Recall that we have $\rho^{-1}([0,d_{\Lambda_2}])=(\tilde\rho|_{V})^{-1}([0,d_{\Lambda_2}])\subset \tilde Y_{\Lambda_3}$. This implies
$$
R(\tilde g)|_\gamma+(h_2^2-2|h_2'|)\circ \rho\geq \left(\underline R_2\chi_{[0,d_{\Lambda_2}]}+h_2^2-2|h_2'|\right)\circ \rho\geq \delta_{2}
$$
and so we obtain $R'>0$. However, it is impossible to have $R'>0$ and $H\ge 0$ because otherwise by doubling trick {\cite[Theorem 5.7]{GL80a}} one can construct a smooth metric
on $\mathbb T^{l+2}$ with positive scalar curvature and this leads to a contradiction.
\end{proof}

\subsection{Extrinsic diameter estimate}\label{Extrinsic diameter estimate}
In the following argument, we will use the notion of extrinsic diameter frequently.
For convenience, we introduce the notation: for Riemannian manifold $(M,g_{M})$ and its subset $W$, denote the extrinsic diameter of $W$ in $(M,g_{M})$ by
\[
\diam(W\subset(M,g_{M})).
\]
Let $M_{n-2}$ be the $(n-2)$-dimensional submanifold from Lemma \ref{Lem: Mn-2}. In the following, we take
$$
\underline R=\min_{Y_{\Lambda_4}} R(g)
$$
and
$$
\mu_{\mathrm{loss}}=\frac{3}{4}\min\{\underline R,\delta_1,\delta_2\},
$$
where {$\mu_{\mathrm{loss}}$ will be set as the constant there in Lemma \ref{Lem: Mn-2}} and $\delta_1,\delta_2$ are the constants from Subsection \ref{Definition of avoidance hypersurfaces}. Note that all these constants are positive.

\subsubsection{Diameter estimate in dimensions three and four}
When $n=3$ or $4$, we can bound the diameter of each component of $M_{n-2}$ that intersects $\tilde Y_{\Lambda_2}$.

\begin{lemma}\label{Lem: 3D}
Let $n=3$. For each component $C$ of $M_{n-2}$ with
$$C_{\Lambda_2}:=C\cap\tilde Y_{\Lambda_2}\neq \emptyset,$$
we have $C\cap \tilde X_{\Lambda_2}\neq \emptyset$ and $$
\diam(C_{\Lambda_2}\subset (\tilde Y,\tilde g))\leq \frac{8\pi}{\sqrt{\underline R}}+\diam(\Lambda_2\subset(X_\ve,g)).
$$
\end{lemma}

\begin{proof}
Recall from Lemma \ref{Lem: Mn-2} that $M_{n-2}$ has $\mathbb T^2$-stabilized scalar curvature $R_{n-2}$ satisfying
$$
R_{n-2}\geq R(\tilde g)|_{M_{n-2}}-\mu_{\mathrm{loss}}>R(\tilde g)|_{M_{n-2}}-\delta_1.
$$
For each component $C$ of $M_{n-2}$ it follows from the above and Proposition \ref{Prop: closed} that either $C\cap \tilde X_{\Lambda_2}=\emptyset$ or $C\cap \tilde X_{\Lambda_2}\subset \tilde X_{\Lambda_2,i_0}$ for some uniquely determined $i_0=i_0(C)$.

In the former case, we see
$$
R_{n-2}|_C\geq R(\tilde g)|_C-\mu_{\mathrm{loss}}\geq \frac{1}{4}\underline R.
$$
Since $C$ is one-dimensional, it is a closed curve. This leads to a contradiction since $\mathbb T^3$ cannot admit any smooth metric with positive scalar curvature.

In the latter case, we decompose $C_{\Lambda_2}:= C\cap \tilde Y_{\Lambda_2}$ into the union of its connected components
$$
C_{\Lambda_2}=\bigcup_i U_i.
$$
For each component $U_i$ we conclude from Lemma \ref{Lem: inradius} that $r_{\mathrm{in}}(U_i,\partial U_i)\leq 4\pi/\sqrt{\underline R}$. Recall that we have $\partial U_i\subset \tilde\Lambda_{2,i_0}$ from the fact $C\cap \tilde X_{\Lambda_2}\subset \tilde X_{\Lambda_2,i_0}$.
To estimate $\diam(C_{\Lambda_2}\subset (\tilde Y,\tilde g))$, we take $p,q\in C_{\Lambda_2}$ arbitrarily, and assume that $p\in U_{i}$ and $q\in U_{j}$. From the triangle inequality we have
\[
\begin{split}
\dist_{\ti{g}}(p,q) \leq {} & \dist_{\ti{g}}(p,\partial U_{i})+\dist_{\ti{g}}(q,\partial U_{j})+\diam(\tilde \Lambda_{2,i_0}\subset(\ti{Y},\ti{g})) \\[1mm]
\leq {} & r_{\mathrm{in}}(U_{i},\partial U_{i})+r_{\mathrm{in}}(U_{j},\partial U_{j})+\diam(\tilde \Lambda_{2,i_0}\subset(\ti{Y},\ti{g})) \\
\leq {} & \frac{8\pi}{\sqrt{\underline R}}+\diam(\tilde \Lambda_{2,i_0}\subset (\tilde X_{\ve,i_0},\tilde g)).
\end{split}
\]
It then follows that
$$
\diam(C_{\Lambda_2}\subset (\tilde Y,\tilde g))\leq \frac{8\pi}{\sqrt{\underline R}}+\diam(\tilde \Lambda_{2,i_0}\subset (\tilde X_{\ve,i_0},\tilde g)).
$$
Notice that $\pi:(\tilde X_{\ve,i_0},\tilde g)\to(X_\ve,g)$ is an isometry which maps $\tilde \Lambda_{2,i_0}$ to $\Lambda_{2}$. We obtain the desired estimate
\begin{equation}\label{C Lambda 2 diameter estimate}
\diam(C_{\Lambda_2}\subset (\tilde Y,\tilde g))\leq \frac{8\pi}{\sqrt{\underline R}}+\diam(\Lambda_{2}\subset (X_{\ve},g)),
\end{equation}
and this completes the proof.
\end{proof}

\begin{lemma}\label{Lem: 4D}
Let $n=4$. For each component $C$ of $M_{n-2}$ with
$$C_{\Lambda_2}:=C\cap\tilde Y_{\Lambda_2}\neq \emptyset,$$
we have
$$
\diam(C_{\Lambda_2}\subset (\tilde Y,\tilde g))\leq \frac{8\pi}{\sqrt{\underline R}}+\diam(\Lambda_2\subset(X_\ve,g)).
$$
\end{lemma}
\begin{proof}
The proof is the same as that of Lemma \ref{Lem: 3D}, where the only difference is that $C\cap \tilde X_{\Lambda_2}$ can be possibly empty now. In this case we have from Corollary \ref{Cor: diameter} that
$$
\diam(C\subset (\tilde Y,\tilde g))\leq \diam(C,\tilde{g}|_C)\leq \frac{4\pi}{\sqrt{\underline R}}.
$$
Therefore the desired diameter estimate still holds.
\end{proof}

\subsubsection{Slice-and-dice in dimension five}
For $n=5$, it might hold that the diameter of each component
of $M_{n-2}$ is {not uniformly bounded}. Thus we need to employ the slice-and-dice procedure of Chodosh--Li \cite[Section 6]{CL20}. We first slice $M_{n-2}$ into a manifold with boundary whose second homology group comes from its boundary, and we then dice it so that the blocks after slice-and-dice have extrinsic diameter control.
\begin{proposition}\label{Prop: slice and dice}
Let $n=5$. If $M_{n-2}$ has non-empty intersection with $\tilde Y_{\Lambda_2}$, then for each component $C$ of $M_{n-2}$ we can construct a slice-and-dice of $C$ consisting of finitely many slicing surfaces $\{S_i\}$ and finitely many dicing surfaces $\{D_j\}$ such that
\begin{itemize}\setlength{\itemsep}{1mm}
\item $\{S_i\}$ are pairwise disjoint closed {connected} surfaces satisfying
$$
\diam(S_{i,\Lambda_2}\subset (\tilde Y,\tilde g))\leq \frac{8\pi}{\sqrt{\underline R}}+\diam(\Lambda_2\subset(X_\ve,g)),
$$
where $S_{i,\Lambda_2}:=S_i\cap \tilde Y_{\Lambda_2}$. {For each $i$, either $S_i\cap \tilde X_{\Lambda_2}=\emptyset$ or $S_i\cap \tilde X_{\Lambda_2}\subset \tilde X_{\Lambda_2,i_0}$ for a uniquely determined index $i_0=i_0(S_i)$};
\item $\{D_j\}$ are pairwise disjoint {connected} surfaces satisfying
$$
\diam(D_{j,\Lambda_4}\subset (\tilde Y,\tilde g))\leq \frac{16\pi}{\sqrt{\underline R}}+\diam(\Lambda_4\subset(X_\ve,g)),
$$
where $D_{j,\Lambda_4}:=D_j\cap \tilde Y_{\Lambda_4}$. {For each $j$, either $D_j\cap \tilde X_{\Lambda_4}=\emptyset$ or $D_j\cap \tilde X_{\Lambda_4}\subset \tilde X_{\Lambda_4,i_0'}$ for a uniquely determined index $i_0'=i_0'(D_j)$.}
Furthermore, if $D_j \cap \tilde{X}_{\Lambda_4} = \emptyset$, then $D_j$ is a topological sphere or disk.
\item Decompose the complement of $\{S_i\}$ and $\{D_j\}$ as the union of its connected components
$$
C-\left(\bigcup_i S_i\right)\cup\left(\bigcup_j D_j\right)=\bigcup_k U_k.
$$
Then $U_{k,\Lambda_2}:=U_k\cap \tilde Y_{\Lambda_2}$ satisfies
\[
\begin{split}
\diam(U_{k,\Lambda_2}\subset(\tilde Y,\tilde g))\leq \frac{48\pi}{\sqrt{\underline R}} & +5\diam(\Lambda_2\subset(X_\ve,g))\\
&+2\diam(\Lambda_4\subset(X_\ve,g))+\frac{8\pi}{\sqrt{\hat \mu_{\mathrm{loss}}}}
\end{split}
\]
with
$$
\hat \mu_{\mathrm{loss}}=\frac{3}{16}\min\{\underline R,\delta_1,\delta_2\}.
$$
\end{itemize}
\end{proposition}
\begin{proof}
The proof follows from the argument of Chodosh--Li \cite[Section 6]{CL20} (see also \cite[Section 2]{Zhu23}) with an essential improvement on the extrinsic diameter control of slicing surfaces, dicing surfaces and blocks after slice-and-dice. The proof will be divided into two steps.

\medskip
\noindent
{\bf Step 1.} Slicing.
\medskip

We are going to show that there are pairwise disjoint embedded closed {connected} surfaces $S_1,S_2,\ldots,S_{i_1}$ satisfying
\begin{itemize}\setlength{\itemsep}{1mm}
\item the inclusion $i:H_2(\partial\hat C)\to H_2(\hat{C})$ is surjective, where $\hat C$ is the metric completion of $C-\cup_i S_i$;
\item and we have
$$
\diam(S_{i,\Lambda_2}\subset (\tilde Y,\tilde g))\leq \frac{8\pi}{\sqrt{\underline R}}+\diam(\Lambda_2\subset(X_\ve,g)),
$$
where $S_{i,\Lambda_2}:=S_i\cap \tilde Y_{\Lambda_2}$.
\end{itemize}

{\it Inductive construction of $\{S_i\}$.} Set $\hat{C}_{0}= C$. If $\hat C_0$ satisfies $H_2(\hat C_0)=0$, then we are done (i.e. there is no $S_{i}$).
Otherwise $\hat C_0$ satisfies $H_2(\hat C_0)\neq 0$ and we can fix a non-zero class $\beta_1\in H_2(\hat C_0)$. Recall that $M_{n-2}$ has $\mathbb T^2$-stabilized scalar curvature $R_{n-2}$ satisfying
\begin{equation}\label{Eq: SD1}
R_{n-2}\geq R(\tilde g)|_{M_{n-2}}-\mu_{\mathrm{loss}}.
\end{equation}
The same thing holds for $\hat C_0$. It follows from Lemma \ref{Lem: homology minimizing} that there is an embedded surface $S$ with integer multiplicity representing $\beta_{1}$ such that $S$ has $\mathbb T^3$-stabilized scalar curvature $R_{n-3}$ which is no less than $R_{n-2}|_S$, and the metric completion of $\hat{C}_{0}-S$ associated with its boundary has $\mathbb{T}^{2}$-stabilized scalar-mean curvature $(R_{n-2},0)$.
Just take one component of $S$ {representing a non-zero homology class} and we denote it by $S_1$.

Next we consider the metric completion $\hat C_1$ of $\hat C_0-S_1$. In a natural way, $(\hat C_1,\partial\hat C_1)$ has $\mathbb T^2$-stabilized scalar-mean curvature $(R_{n-2},0)$ where $R_{n-2}$ is no less than $R(\tilde g)|_{M_{n-2}}-\mu_{\mathrm{loss}}$. Similarly, if the map $H_2(\partial \hat C_1)\to H_2(\hat C_1)$ is surjective, then we are done. Otherwise, we can find an embedded {closed connected} surface $S_2$ representing a non-trivial class
$$
\beta_2\in H_2(\hat C_1)-\mathrm{Im}\left(H_2(\partial \hat C_1)\to H_2(\hat C_1)\right),
$$
which has $\mathbb T^3$-stabilized scalar curvature $R_{n-3}$ that is no less than $R_{n-2}|_{S_2}$  and the metric completion $(\hat C_2,\partial\hat C_2)$ of $\hat{C}_{1}-S_{2}$ has $\mathbb{T}^{2}$-stabilized scalar-mean curvature $(R_{n-2},0)$.

By induction we can construct a sequence of surfaces $\{S_i\}_{i\in I}$ satisfying
$$[S_i]\in H_2(\hat C_{i-1})-\mathrm{Im}\left(H_2(\partial \hat C_{i-1})\to H_2(\hat C_{i-1})\right)$$
and $S_i$ has $\mathbb T^3$-stabilized scalar curvature $R_{n-3}$ which is no less than $R_{n-2}|_{S_i}$. Also, each $(\hat C_i,\partial\hat C_i)$ has $\mathbb{T}^{2}$-stabilized scalar-mean curvature $(R_{n-2},0)$.

\medskip

{\it Termination of the construction.} We have to show that the construction terminates in finitely many steps. This comes from a topological argument by Bamler, Li and Mantoulidis (see the proof of \cite[Lemma 2.5]{BLM22}) and here we write out the details for completeness. Denote $i_k:H_2(\partial\hat C_k)\to H_2(\hat C_k)$ and the goal is to show that the quotient $Q_k=H_{2}(\hat C_k)/\mathrm{Im}\,i_k$ is torsion-free and has strictly decreasing rank when $k$ increases. To see that $Q_k$ is torsion-free we consider the exact sequence
$$
H_2(\partial \hat C_k)\to H_2(\hat C_k)\to H_2(\hat C_k,\partial\hat C_k)\to H_1(\partial\hat C_k).
$$
Recall that $\mathrm{dim}\,\hat C_k=3$ and so $H_2(\hat C_k,\partial\hat C_k)\cong H^1(\hat C_k)$ is torsion-free from the universal coefficient theorem. From the above exact sequence $Q_k$ is isometric to the image of $H_2(\hat C_k)\to H_2(\hat C_k,\partial\hat C_k)$, which is a subgroup of $H_2(\hat C_k,\partial\hat C_k)$ and so is torsion-free. Next we show that the rank of $Q_k$ decreases strictly. Let $\mathcal N_{S_{k+1}}$ denote the tubular neighborhood of $S_{k+1}$ in $\hat C_k$. Consider the exact sequence
$$
H_2(\partial\mathcal N_{S_{k+1}})\to H_2(\hat C_{k+1})\oplus H_2(\mathcal N_{S_{k+1}})\to H_2(\hat C_{k}).
$$
It is clear that $\mathrm{Ker}\big(H_2(\hat C_{k+1})\to H_2(\hat C_k)\big)$ is contained in the image of $H_2(\partial\mathcal N_{S_{k+1}})\to H_2(\hat C_{k+1})$ and so contained in the image of $i_{k+1}$. This induces an injective map
$$
Q_{k+1}\to H_2(\hat C_k)/\mathrm{ Im }\,i_{k+1}\cong Q_k/[S_{k+1}]
$$
and so $\mathrm{rank}\, Q_{k+1}<\mathrm{rank}\, Q_k$.

\medskip
{\it Diameter bound for slicing surfaces.} Recall that $\mu_{\mathrm{loss}}$ is no greater than $3\delta_1/{4}$. Then we can apply Proposition \ref{Prop: closed} to each $S_i$ so that we obtain either $S_i\cap \tilde X_{\Lambda_2}=\emptyset$ or $S_i\cap \tilde X_{\Lambda_2}\subset \tilde X_{\Lambda_2,i_0}$ for a uniquely determined index $i_0=i_0(S_i)$. In the former case, it follows from
$$R_{n-3}\geq R_{n-2}|_{S_i}\geq R(\tilde g)|_{S_i}-\mu_{\mathrm{loss}}\geq \frac{1}{4}\underline R$$
and Corollary \ref{Cor: diameter} that
$$
\diam(S_{i,\Lambda_2}\subset (\tilde Y,\tilde g))\leq \diam(S_i,\ti{g}|_{S_i})\leq \frac{4\pi}{\sqrt{\underline R}}.
$$
In the latter case, we write $S_{i,\Lambda_2}=\cup_l S_{il}$ as the decomposition into connected components. Clearly we have $\partial S_{il}\subset \tilde \Lambda_{2,i_0}$ from the fact $S_i\cap \tilde X_{\Lambda_2}\subset \tilde X_{\Lambda_2,i_0}$. From Lemma \ref{Lem: inradius} we see $r_{\mathrm{in}}(S_{il},\partial S_{il})\leq 4\pi/\sqrt{\underline R}$ and the triangle inequality (cf. the argument of \eqref{C Lambda 2 diameter estimate}) shows the desired estimate:
$$
\diam(S_{i,\Lambda_2}\subset (\tilde Y,\tilde g))\leq \frac{8\pi}{\sqrt{\underline R}}+\diam(\Lambda_2\subset(X_\ve,g)).
$$

\medskip
\noindent
{\bf Step 2}. Dicing.
\medskip

Recall that $\hat C$ denotes the metric completion of $C-\cup_i S_i$. From the fact $C\cap \tilde Y_{\Lambda_2}\neq \emptyset$ we have $\hat C\cap \tilde Y_{\Lambda_2}\neq \emptyset$ as well. Fix a point $p\in \hat C\cap \tilde Y_{\Lambda_2}$ and take a small geodesic ball $B_\delta$ centered at $p$. Since $\tilde Y_{\Lambda_2}$ is open, by taking  $\delta$ small enough we can guarantee $\partial B_\delta\cap\tilde X_{\Lambda_{2}}=\emptyset$  and
$$
\diam(\partial B_{\delta}\subset (\tilde Y,\tilde g))\leq \frac{16\pi}{\sqrt{\underline R}}+\diam(\Lambda_4\subset(X_\ve,g)).
$$
Set
$$
\hat \mu_{\mathrm{loss}}=\frac{3}{16}\min\{\underline R,\delta_1,\delta_2\}.
$$
In the following, we divide the argument into two cases.

\medskip
{\it Case 1. The whole $\hat C$ is contained in the $(4\pi/\sqrt{\hat\mu_{\mathrm{loss}}})$-neighborhood of $\partial B_\delta$.} By assumption for each point $q\in \hat C\cap \tilde Y_{\Lambda_2}$ we can find a curve $\gamma$ connecting $q$ and $\partial B_\delta$ with length no greater than $4\pi/\sqrt{\hat\mu_{\mathrm{loss}}}$.
It follows from the diameter bound of $\partial B_\delta$ and  the triangle inequality (cf. the argument of \eqref{C Lambda 2 diameter estimate}) that

$$\diam(\hat C\subset (\tilde Y,\tilde g))\leq \frac{8\pi}{\sqrt {\hat \mu_{\mathrm{loss}}}}+\frac{16\pi}{\sqrt{\underline R}}+\diam(\Lambda_4\subset(X_\ve,g)).$$

\medskip
{\it Case 2. The complement $\hat C-B( B_\delta;4\pi/\sqrt{\hat \mu_{\mathrm{loss}}})$ is non-empty.} In this case, we have to construct suitable surfaces for dicing. For our purpose let us {take a smoothing $(A,\partial_{+}^{A},\partial_{-}^{A})$ of}
$$\left(\bar B(B_\delta;4\pi/\sqrt{\hat \mu_{\mathrm{loss}}})-B(B_\delta;\pi/\sqrt{\hat \mu_{\mathrm{loss}}}),
 \partial B(B_\delta;4\pi/\sqrt{\hat \mu_{\mathrm{loss}}}),
\partial B(B_\delta;\pi/\sqrt{\hat \mu_{\mathrm{loss}}})\right)
$$
such that $\dist(\partial_+^{A},\partial_-^{A})>2\pi/\sqrt{\hat \mu_{\mathrm{loss}}}.$ {We focus on the component $V$ such that $\partial_-:=\partial V\cap\partial_+^{A}$ and $\partial_+:=\partial V\cap\partial_+^{A}$ are both non-empty. Then $(V,\partial_{+},\partial_{-})$ is a Riemannian band.} With the same deformation argument as the one at the end of the proof of Lemma \ref{Lem: inradius}, we can further assume that $\partial_+\cup\partial_-$ intersects with $\overline{\partial V-(\partial_+\cup\partial_-)}$ in acute angles. By construction, $(\hat C,\partial\hat C)$ has $\mathbb T^2$-stabilized scalar-mean curvature $(R_{n-2},0)$ and the same thing holds for $(V,\overline{\partial V-(\partial_+\cup\partial_-)})$. Now it follows from Lemma \ref{Lem: corner mu bubble} and \eqref{Eq: SD1} that we can find a surface $\Sigma$ bounding a region with $\partial_-$ such that $(\Sigma,\partial\Sigma)$ has $\mathbb T^3$-stabilized scalar-mean curvature $(R_{n-3}',H_{n-3}')$ satisfying
$$
R_{n-3}'\geq R(\tilde g)|_{\Sigma}-\mu_{\mathrm{loss}}-\hat \mu_{\mathrm{loss}}\mbox{ and }H_{n-3}'=0.
$$
Clearly $\Sigma$ bounds a relative region $\Omega_1$ containing $B_\delta$ relative to $\partial \hat C$ (i.e. $B_{\delta}\subset\Omega_{1}$ and $\partial\Omega_{1}-\Sigma\subset\partial\hat{C}$). Since $\Omega_1$ may have multiple components, we take the component of $\Omega_1$ containing $B_\delta$, still denoted by $\Omega_1$.

With $B_\delta$ replaced by $\Omega_1$ we can repeat the above construction and inductively we end up with an exhaustion
$$
B_\delta=\Omega_0\subset\Omega_1\subset \Omega_2\subset \cdots \subset\Omega_m\subset\hat C,
$$
where each $\Omega_{l}$ is connected and the finiteness of this exhaustion comes from the facts
$$B(B_\delta,\pi {l}/\sqrt{\hat \mu_{\mathrm{loss}}})\subset \Omega_{l}$$ and that $\hat C$ has bounded diameter from its compactness. For each $\Omega_{l}$ its reduced boundary $\partial\Omega_{l}\cap \mathring {\hat C}$ has $\mathbb T^3$-stabilized scalar-mean curvature $(R_{n-3}',H_{n-3}')$ satisfying
$$
R_{n-3}'\geq R(\tilde g)|_{\Sigma}-\mu_{\mathrm{loss}}-\hat \mu_{\mathrm{loss}}\mbox{ and }H_{n-3}'=0.
$$
The dicing surfaces $D_j$ are taken to be components of the induced boundaries $\partial\Omega_{l} \cap \mathring {\hat C}$.

\medskip

{\it Diameter bound for dicing surfaces.} Recall that we have
$$
\mu_{\mathrm{loss}}+\hat \mu_{\mathrm{loss}}\leq \frac{15}{16}\delta_2.
$$
So we can apply Proposition \ref{Prop: compact} to each dicing surface $D_j$  and conclude that either $D_j\cap \tilde X_{\Lambda_4}=\emptyset$ or there is a uniquely determined index $i_0'$ such that $D_j\cap \tilde X_{\Lambda_4}\subset \tilde X_{\Lambda_4,i_0'}$.
In the former case, we conclude from Corollary \ref{Cor: diameter}, $R_{n-3}'\geq \underline R/16$ and $H_{n-3}'=0$ that  $D_j$ is a topological sphere or disk and
$$
\diam(D_{j,\Lambda_4}\subset (\tilde Y,\tilde g))\leq \diam(D_j)\leq \frac{8\pi}{\sqrt{\underline R}},\mbox{ where }D_{j,\Lambda_4}:=D_j\cap \tilde Y_{\Lambda_4}.
$$
In the latter case, we write $D_{j,\Lambda_4}=\cup_l D_{jl}$ where $D_{jl}$ are components of $D_{j,\Lambda_4}$ with $\partial D_{jl}-\partial D_j\subset \tilde \Lambda_{4,i_0'}$. It follows from Lemma \ref{Lem: inradius} and the triangle inequality (cf. the argument of \eqref{C Lambda 2 diameter estimate}) that
$$
\diam(D_{j,\Lambda_4}\subset (\tilde Y,\tilde g))\leq \frac{16\pi}{\sqrt{\underline R}}+\diam(\Lambda_4\subset (X_\ve,g)).
$$

\medskip

{\it Diameter bound for blocks.} Decompose the complement of $\{S_i\}$ and $\{D_j\}$ into its connected components as
$$
C-\left(\bigcup_i S_i\right)\cup\left(\bigcup_j D_j\right)=\bigcup_k U_k.
$$
We are going to estimate the extrinsic diameter bound for $U_{k,\Lambda_2}:=U_k\cap \tilde Y_{\Lambda_2}$.

First we need to prove the following key lemma.
\begin{lemma}\label{Lem: combination diameter}
Let $\Gamma$ be a connected component of the slice-and-dice trace
$$
\mathcal T_{SD}:=\left(\bigcup_i S_i\right)\cup\left(\bigcup_j D_j\right).
$$
Then we have
$$
\diam(\Gamma_{\Lambda_2}\subset (\tilde Y,\tilde g))\leq \frac{48\pi}{\sqrt{\underline R}}+3\diam(\Lambda_2\subset(X_\ve,g))+2\diam(\Lambda_4\subset(X_\ve,g)),
$$
where $\Gamma_{\Lambda_2}:=\Gamma\cap \tilde Y_{\Lambda_2}$.
\end{lemma}
\begin{proof}
Notice that $\Gamma$ has the structure of a graph where each vertex corresponds to a slicing surface or a dicing surface, and there is an edge between two vertices if the two surfaces representing the vertices have non-empty intersection. After marking each vertex by $\mathcal S$ or $\mathcal D$ to record its source from slicing surfaces or dicing surfaces, we know from the construction ($\{S_{i}\}$ are pairwise disjoint and so are $\{D_{j}\}$) that each edge must be of $\mathcal S$-$\mathcal D$ type (See Figure \ref{graph structure}).
To obtain a diameter bound of $\Gamma_{\Lambda_2}$ we just need to estimate a uniform diameter bound of $\Gamma_{0,\Lambda_2}:=\Gamma_0\cap \tilde Y_{\Lambda_2}$ for any path-structured ``subcomplex'' $\Gamma_0$ of $\Gamma$ (See Figure \ref{subcomplex}).

\begin{figure}[htbp]
\centering
\includegraphics[width=12cm]{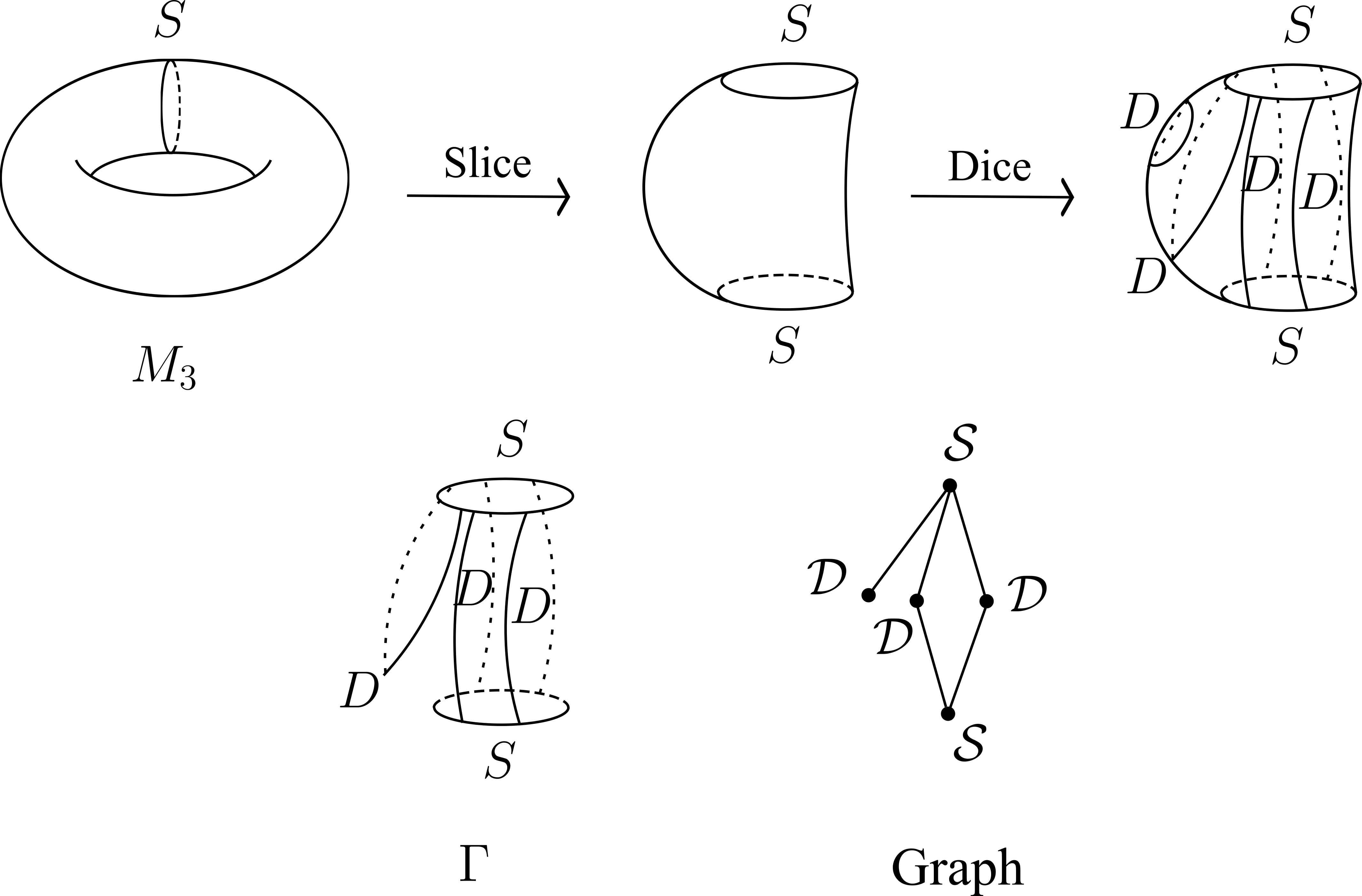}
\caption{The graph structure of $\Gamma$}
\label{graph structure}
\end{figure}

\begin{figure}[htbp]
\centering
\includegraphics[width=12cm]{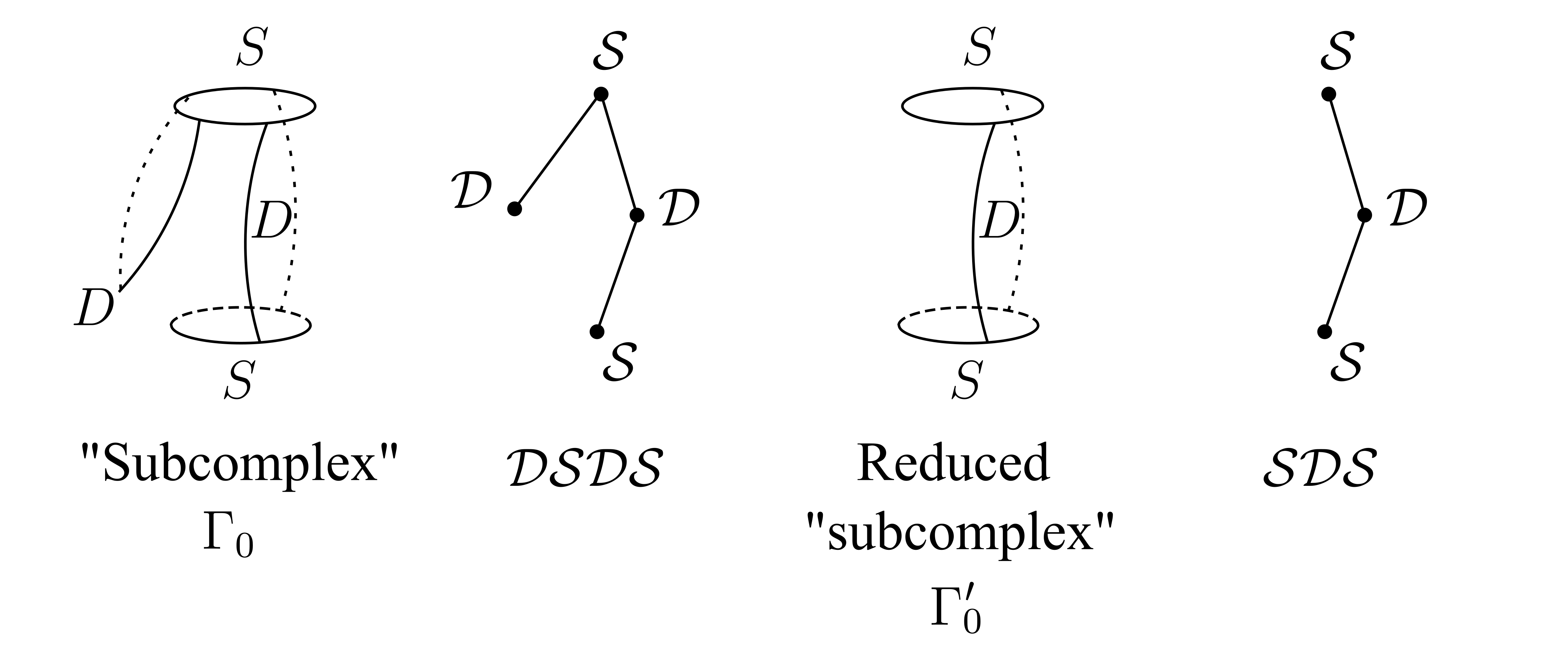}
\caption{``subcomplex'' $\Gamma_0$ and reduced ``subcomplex'' $\Gamma_0'$ }
\label{subcomplex}
\end{figure}

We begin with the following observation: there are at most two dicing surfaces in $\Gamma_0$ disjoint from $\tilde X_{\Lambda_4}$. In order to see this we recall that if a dicing surface $D_j$ is disjoint from $\tilde X_{\Lambda_4}$, then it has to be a topological sphere or disk. In the former case, $\Gamma_0$ is a single dicing sphere and we obtain the required diameter estimate
$$
\diam(\Gamma_{0,\Lambda_2})\leq  \frac{8\pi}{\sqrt{\underline R}}.
$$
In the latter case, since a dicing disk has only one boundary component and thus can only intersect exactly one slicing surface, it must lie in the end-point position of the path structure corresponding to $\Gamma_0$. Thus the number of dicing disks in $\Gamma_0$ cannot exceed two, establishing our observation. After removing the (at most two) dicing surfaces in the end-point position of $\Gamma_0$, we obtain a reduced ``subcomplex'' $\Gamma_0'$ which has the path structure of the type $\mathcal S\mathcal D\mathcal S\mathcal D\cdots\mathcal S$, where each dicing surface intersects with some $\tilde X_{\Lambda_4,i}$ (See Figure \ref{subcomplex}). When there is no $\mathcal D$-component (i.e. $\Gamma_0'$ consists of a single slicing surface), we obtain the required diameter estimate
$$
\diam(\Gamma_{0,\Lambda_2}'\subset (\tilde Y,\tilde g))\leq \frac{8\pi}{\sqrt{\underline R}}+\diam(\Lambda_2\subset (X_\ve,g)).
$$
Otherwise, we claim that all dicing surfaces in $\Gamma_0'$ must intersect the same $\tilde X_{\Lambda_{4},i_0}$. Clearly we are done if there is only one $\mathcal D$-component, so we just need to deal with the case when there are at least two $\mathcal D$-components. For our purpose, we take a closer look at the $\mathcal D\mathcal S\mathcal D$-structure with two dicing surfaces involved. By Proposition \ref{Prop: compact}, we see that if a dicing surface $D_j$ intersects some $\tilde X_{\Lambda_4,i}$ then every component of $\partial D_j$ must intersect $\tilde X_{\Lambda_2,i}$ with the same index $i$. Combining the fact that a slicing surface can only intersect one $\tilde X_{\Lambda_2,i}$, we conclude that adjacent dicing surfaces in $\Gamma_0'$ must intersect the same $\tilde X_{\Lambda_{4},i_0}$ for some index $i_0$ and consequently all dicing surfaces in $\Gamma_0'$ intersect the same $\tilde X_{\Lambda_{4},i_0}$. As a consequence, every slicing surface in $\Gamma_0'$ must intersect $\tilde X_{\Lambda_2,i_0}$ since it contains some boundary component of the dicing surfaces.

Take any pair of points $p,q\in \Gamma_{0,\Lambda_2}'$. Then we can find two surfaces $\Sigma_1$ and $\Sigma_2$ in $\Gamma_0'$ such that $p\in \Sigma_{1,\Lambda_2}:=\Sigma_{1}\cap\ti{Y}_{\Lambda_{2}}$ and $q\in \Sigma_{2,\Lambda_2}:=\Sigma_{2}\cap\ti{Y}_{\Lambda_{2}}$. From our previous discussion we can find a point $p^*$ in $\Sigma_1\cap \tilde \Lambda_{2,i_0}$ and a point $q^*$ in $\Sigma_2\cap \tilde \Lambda_{2,i_0}$. Then we have
\[
\begin{split}
\dist(p,q)&\leq \dist(p,p^*)+\dist(p^*,q^*)+\dist(q^*,q)\\
&\leq \frac{32\pi}{\sqrt{\underline R}}+3\diam(\Lambda_2\subset(X_\ve,g))+2\diam(\Lambda_4\subset(X_\ve,g)).
\end{split}
\]
Since $p$ and $q$ are chosen arbitrarily, we have the same bound for the extrinsic diameter $\diam(\Gamma_{0,\Lambda_2}'\subset(\tilde Y,\tilde g))$.
Recall that $\Gamma_0$ and $\Gamma_0'$ differ by at most two slicing disks with diameter bounded by $8\pi/\sqrt{\underline R}$, we conclude
$$
\diam(\Gamma_{0,\Lambda_2}\subset(\tilde Y,\tilde g))\leq \frac{48\pi}{\sqrt{\underline R}}+3\diam(\Lambda_2\subset(X_\ve,g))+2\diam(\Lambda_4\subset(X_\ve,g)).
$$
The proof is completed by noticing that any pair of points $p,q\in \Gamma_{\Lambda_2}$ can be realized as a pair of points in some $\Gamma_{0,\Lambda_2}$ for some ``subcomplex'' $\Gamma_0$ with path structure.
\end{proof}

Now we are ready to give diameter estimates for blocks $U_k$. Recall that $U_k$ denotes the block obtained from the $l$-th dicing, namely a component of $\Omega_l-\Omega_{l-1}$. From our construction of the dicing surfaces in Step 2, every point in $U_k$ has distance no greater than $4\pi/\sqrt{\hat \mu_{\mathrm{loss}}}$ from $\partial\Omega_{l-1} \cap \partial U_k$. Thus, to obtain a diameter estimate of $U_k$, it suffices to bound $\partial\Omega_{l-1}\cap \partial U_k$. For this end, we denote the components of $\partial\Omega_{l-1}\cap \partial U_k$ by $D_{k1}, D_{k2}, \ldots, D_{km}$. {Note that each $D_{ki}$ is a dicing surface obtained from the $(l-1)$-th dicing. To complete the proof of Proposition \ref{Prop: slice and dice}}, we divide the discussion into the following two cases:

\medskip

{\it Case 1. $m=1$.} If we have $D_{k1}\cap \tilde X_{\Lambda_4}=\emptyset$, then the triangle inequality (cf. the argument of \eqref{C Lambda 2 diameter estimate}) shows
$$
\diam(U_{k}\subset (\tilde Y,\tilde g))\leq \diam(D_{k1})+\frac{8\pi}{\sqrt{\hat \mu_{\mathrm{loss}}}}\leq \frac{8\pi}{\sqrt{\underline R}}+\frac{8\pi}{\sqrt{\hat \mu_{\mathrm{loss}}}}.
$$
Otherwise, we consider the set
$$
\mathcal K=D_{k1,\Lambda_4}\cup\left(\bigcup_{i:D_{k1}\cap \tilde X_{\Lambda_4,i}\neq \emptyset}\tilde \Lambda_{4,i}\right),
$$
where $D_{k1,\Lambda_4}:=D_{k1}\cap \tilde Y_{\Lambda_4}$.
Clearly we have that
\begin{itemize}\setlength{\itemsep}{1mm}
\item when $D_{k1}$ is completely contained in some $\tilde X_{\Lambda_4,i}$, we have $\mathcal K=\tilde \Lambda_{4,i}$ and so $\diam(\mathcal K\subset (\tilde Y,\tilde g))\leq \diam(\Lambda_4\subset(X_\ve,g))$;
\item otherwise $D_{k1,\Lambda_4}\neq \emptyset$ and so $D_{k1}\cap \tilde \Lambda_{4,i}\neq \emptyset$ for {a unique} index $i$ satisfying $D_{k1}\cap \tilde X_{\Lambda_4,i}\neq \emptyset$. From the triangle inequality  (cf. the argument of \eqref{C Lambda 2 diameter estimate}) we conclude
\[
\begin{split}
\diam(\mathcal K\subset(\tilde Y,\tilde g)) & \leq \diam(\Lambda_4\subset(X_\ve,g))+\diam(D_{k1,\Lambda_4}\subset (\tilde Y,\tilde g))\\
& \leq \frac{16\pi}{\sqrt{\underline R}}+{2}\diam(\Lambda_4\subset(X_\ve,g)).
\end{split}
\]
\end{itemize}
Given any point $p$ in $U_{k,\Lambda_4}:=U_{k}\cap\tilde{Y}_{\Lambda_{4}}$ we can connect $p$ and $D_{k1}$ by a curve $\gamma$ with length no greater than $4\pi/\sqrt{\hat \mu_{\mathrm{loss}}}$ by our construction. Notice that $\gamma$ must intersect $\mathcal K$.
From the triangle inequality (cf. the argument of \eqref{C Lambda 2 diameter estimate}) we obtain
$$
\diam(U_{k,\Lambda_4}\subset (\tilde Y,\tilde g))\leq \frac{16\pi}{\sqrt{\underline R}}
+\frac{8\pi}{\sqrt{\hat \mu_{\mathrm{loss}}}}+{2}\diam(\Lambda_4\subset(X_\ve,g)).
$$

\medskip
{\it Case 2. $m\geq 2$.} We claim that each $D_{ki}$ must intersect at least one slicing surface. {If this is not the case for some $D_{ki}$, then $D_{ki}$ is a closed surface hence a cycle in $\hat C$. We can fix a point $p$ in $D_{ki}$ and a point $q$ in $D_{ki'}$ for $i '\neq i$, and we can find a curve connecting $p$ and $q$ in $U_k$ and $\Omega_{l-1}$, respectively, thanks to the path-connectivity of $U_k$ and $\Omega_{l-1}$. This gives} a simple closed curve $\gamma$ in $\hat C$ which intersects only once with $D_{ki}$. This implies that $D_{ki}$ represents a non-trivial homology class in $H_2(\hat C)$, which does not lie in the image of $H_2(\partial \hat C)\to H_2(\hat C)$. This contradicts the surjectivity of $H_2(\partial \hat C)\to H_2(\hat C)$ from the construction in the slicing procedure in Step 1.

Now we collect all the slicing surfaces intersecting ${\bigcup_i} D_{ki}$ and denote them by $S_{k1}, S_{k2},\ldots, {S_{kr}}$. We claim that the union
$$
\mathcal T_k:=\left(\bigcup_{i=1}^mD_{ki}\right)\cup\left(\bigcup_{j=1}^{r} S_{kj}\right)
$$
must lie in some component of the slice-and-dice trace
$$
\mathcal T_{SD}:=\left(\bigcup_i S_i\right)\cup\left(\bigcup_j D_j\right).
$$
Otherwise, $D_{ki}$ and $S_{kj}$ can be divided into $s\geq 2$ collections
$$
\mathcal B_\tau=\{D_{k1}^\tau,\ldots, D_{km_\tau}^\tau,S_{k1}^\tau,\ldots, S_{k{r}_\tau}^\tau\},\ \tau=1,2,\ldots,s,
$$
such that the unions of $D_{ki}$ and $S_{kj}$ from different collections lie in different components of $\mathcal T_{SD}$. For the first collection we notice that the combination of chains
$$
D_{k1}^1+\cdots+D_{km_1}^1-\partial\Omega_{l-1}\cap \left(\bigcup_{j=1}^{{r}_1}S_{kj}^1\right)
$$
turns out to be a cycle in $\hat C$. Since there are other dicing surfaces from the second collection, as before, {thanks to the path-connectivity of $U_k$ and $\Omega_{l-1}$,} we can construct a simple closed curve in the interior of $\hat C$ which only intersects once with the above cycle, which yields a contradiction to the surjectivity of $H_2(\partial \hat C)\to H_2(\hat C)$ again.

From previous discussion we can conclude from Lemma \ref{Lem: combination diameter} that
$$
\diam(\mathcal T_{k,\Lambda_2}\subset(\tilde Y,\tilde g))\leq \frac{48\pi}{\sqrt{\underline R}}+3\diam(\Lambda_2\subset(X_\ve,g))+2\diam(\Lambda_4\subset(X_\ve,g)).
$$
From a similar discussion as in Case 1 depending on $\mathcal T_k\cap\tilde X_{\Lambda_2}$ empty or not we can obtain
\[
\begin{split}
\diam(U_{k,\Lambda_2}\subset(\tilde Y,\tilde g)) \leq \frac{48\pi}{\sqrt{\underline R}} & +5\diam(\Lambda_2\subset(X_\ve,g))\\
& +2\diam(\Lambda_4\subset(X_\ve,g))+\frac{8\pi}{\sqrt{\hat \mu_{\mathrm{loss}}}}.
\end{split}
\]
This completes the proof.
\end{proof}

\subsection{Completion of the proof of Theorem \ref{Thm: main}}
First we present a quantitative filling lemma for $M_{n-2}$.
\begin{proposition}\label{Prop: filling Mn-2}
Let $n \in \{3,4,5\}$. Assume that $M_{n-2}\cap \tilde N_\ve\neq \emptyset$. There is a positive constant $r_0$ independent of $L$ such that $M_{n-2}$ can be realized as the relative boundary of a chain $\Gamma$ relative to $ \ti X_{\ve}$, whose support is contained in $B^{\tilde g}_{r_0}(M_{n-2})$.
\end{proposition}
\begin{proof}
It suffices to deal with each component $C$ of $M_{n-2}$.

\medskip

{\it Dimension three and four.} It follows from Lemma \ref{Lem: excision} that we have $[C]=0\in H_{n-2}(\tilde Y,\tilde X_\ve)$ and so the same thing holds for $C_\ve:=C\cap \tilde N_\ve$. From Lemma \ref{Lem: 3D} and Lemma \ref{Lem: 4D} we see
$$
\diam(C_\ve\subset (\tilde Y,\tilde g))\leq \diam(C_{\Lambda_2}\subset (\tilde Y,\tilde g))\leq r_1^*,
$$
where
$$
r_1^*=\frac{8\pi}{\sqrt{\underline R}}+\diam(\Lambda_2\subset (X_\ve,g)).
$$
Fix a point $p\in C_{\ve}$ (if $C_\ve=\emptyset$ then nothing needs to be done). From Lemma \ref{Lem: quantitative filling} we can find a chain $\Gamma$ supported in $B_{r_2^*}(p)$ such that $C_\ve$ is a relative boundary of $\Gamma$ relative to $\tilde X_\ve$, where $r_2^*=R(r_1^*)$ and $R(\cdot)$ is the function from Lemma \ref{Lem: quantitative filling}. We just take $r_0=r_2^*$. Then $\Gamma$ is supported in $B_{r_0}(C)$.

\medskip

{\it Dimension five.} We just need to deal with the case when $C_\ve\neq \emptyset$. To fill $C_\ve$ we use the slice-and-dice from Proposition \ref{Prop: slice and dice} with the same notation. Recall that each block $U_k$ satisfies
$$
\diam(U_{k,\ve})\leq r_1^* \mbox{ with } U_{k,\ve}:=U_k\cap \tilde N_\ve,
$$
where
$$
r_1^*= \frac{48\pi}{\sqrt{\underline R}}+5\diam(\Lambda_2\subset(X_\ve,g)) +2\diam(\Lambda_4\subset(X_\ve,g))+\frac{8\pi}{\sqrt{\hat \mu_{\mathrm{loss}}}}.
$$
Let $V_{k1}, V_{k2},\ldots,V_{kl_k}$ denote the boundary components of $U_k$. If $V_{ki,\ve}:=V_{ki}\cap\tilde N_\ve=\emptyset$, we fix a point $p_{ki}$ in $V_{ki}$ and take $\hat U_{ki}=\emptyset$. Otherwise, we fix a point $p_{ki}$ in $V_{ki,\ve}$ and it follows from Lemma \ref{Lem: excision} and Lemma \ref{Lem: quantitative filling} that we can find a chain $\hat U_{ki}$ supported in $B_{r_2^*}(p_{ki})$ with $r_2^*=R(r_1^*)$ such that $V_{ki}$ is a relative boundary of $\hat U_{ki}$ relative to $\tilde X_\ve$. Denote $Q_k=U_{k,\ve}-\sum_i \hat U_{ki}$. Then $Q_k$ is a relative cycle in $\tilde Y$ relative to $\tilde X_\ve$. Notice that $Q_k=\emptyset$ when $U_{k,\ve}=\emptyset$. Otherwise, we fix a point $q_k\in U_{k,\ve}$. Combining the diameter bound for $U_{k,\ve}$ and the control for support of $\hat U_{ki}$ from previous discussion, we see that $Q_k$ is supported in $B_{r_3^*}(q_k)$ with
$
r_3^*=r_2^*+r_1^*.
$
Using Lemma \ref{Lem: quantitative filling} again we can find a chain ${\Xi_k}$ supported in $B_{r_4^*}(q_k)$ with $r_4^*=R(r_3^*)$ such that $Q_k$ is a relative boundary of ${\Xi_k}$ relative to $\tilde X_\ve$. In the chain level, we have obtain
$$
\partial\left(\sum_k{\Xi_k}\right)
= \sum_{k}Q_{k}
= \sum_{k}U_{k,\ve}-\sum_{k,i}\hat U_{ki}
= C-\sum_{k,i}\hat U_{ki}\mbox{ relative to }\tilde X_\ve.
$$

Next we let $W_1,W_2,\ldots, W_s$ denote the connected components of
$$
\mathcal T_{SD}=\left(\bigcup_i S_i\right)\cup\left(\bigcup_j D_j\right).
$$
Let
$$
 \Theta_j=\bigcup_{V_{ki}\subset W_j}\hat U_{ki}.
$$
Notice that $\Theta_j$ is a relative cycle in $\tilde Y$ relative to $\tilde X_\ve$. If $\Theta_{j,\ve}:=\Theta_{j}\cap\tilde N_\ve=\emptyset$, we take $\hat \Gamma_j=\emptyset$. Otherwise, we pick up a point $x_j\in \Theta_{j,\ve}$ and it follows from Lemma \ref{Lem: excision} and Lemma \ref{Lem: quantitative filling} as well as Lemma \ref{Lem: combination diameter} that there is a chain $\hat \Gamma_j$ supported in $B_{r_6^*}(x_j)$ {such that $\partial\hat{\Gamma}_{j}=\Theta_{j}$ relative to $\ti{X}_{\ve}$}, where $r_6^*=R(r_5^*)$ with
$$r_5^*=\frac{48\pi}{\sqrt{\underline R}}+3\diam(\Lambda_2\subset(X_\ve,g))+2\diam(\Lambda_4\subset(X_\ve,g)).$$

Now we arrive at
$$
C=\partial\left(\sum_k{\Xi_k}+\sum_j\hat\Gamma_j\right)\mbox{ relative to }\tilde X_\ve.
$$
Notice that we have ${\Xi_k}\subset B_{r_4^*}(q_k)\subset B_{r_4^*}(C)$ and
$$\hat\Gamma_j\subset B_{r_6^*}(x_j)\subset B_{r_6^*}(\hat U_{ki})\subset B_{r_6^*+r_2^*}(p_{ki})\subset B_{r_6^*+r_2^*}(C).$$
By taking
$$\Gamma=\sum_k{\Xi_k}+\sum_j\hat\Gamma_j\mbox{ and }r_0=r_2^*+r_4^*+r_6^*$$
we complete the proof.
\end{proof}

\begin{proof}[Proof of Theorem \ref{Thm: main}]
We divide the argument into two cases:

\medskip

{\it Case 1. $M_{n-2}\cap \tilde N_\ve=\emptyset$.} The chain $\tilde\Gamma_{n-1}:=\tilde M_{n-1}+\Omega_{n-1}$ with $\tilde M_{n-1}$ and $\Omega_{n-1}$ coming from Lemma \ref{Lem: Mn-1} and Lemma \ref{Lem: Mn-2} respectively is a relative cycle in $\tilde Y$ relative to $\tilde X_\ve$. From our construction the line $\tilde \sigma$ has non-zero algebraic intersection with $\tilde\Gamma_{n-1}$, which yields ${[\tilde\Gamma_{n-1}]}\neq 0\in H_{n-1}(\tilde Y,\tilde X_\ve)$ contradicting Lemma \ref{Lem: excision}.

\medskip
{\it Case 2. $M_{n-2}\cap \tilde N_\ve\neq\emptyset$}. We consider the chain $\tilde \Gamma_{n-1}:=\tilde M_{n-1}+\Omega_{n-1}+\Gamma$, where the extra $\Gamma$ is the chain from Proposition \ref{Prop: filling Mn-2}. Fix $L>r_0$ with $r_0$ from Proposition \ref{Prop: filling Mn-2}. Combining the facts
$$
\dist_{\tilde g}(M_{n-2},\tilde \sigma(\mathbb R))\geq L>r_0\mbox{ and }\mathrm{supp}\,\Gamma\subset B^{\tilde g}_{r_0}(M_{n-2})
$$
we see that $\tilde \sigma(\mathbb R)$ cannot have any intersection with $\Gamma$. In particular, the chain $\tilde \Gamma_{n-1}$ is still a relative cycle in $\tilde Y$ relative to $\tilde X_\ve$ having non-zero algebraic intersection number with the line $\tilde \sigma$, which leads to the same contradiction as above.

\medskip

So far we have obtained a desired contradiction assuming that there is a complete metric on $Y$ with positive scalar curvature, which proves Theorem \ref{Thm: main}.
\end{proof}

\section{Proof of Theorem \ref{Thm: partial classification} and \ref{Thm: mapping rigidity}}\label{Sec: compact-to-complete}
First let us recall
{
\begin{definition}[i.e. Definition \ref{Defn: quasi-proper}]
Let $M^{n}$ and $N^{n}$ be orientable $n$-manifolds. A continuous map $f:M \to N$ is said to be quasi-proper if for all proper maps $\phi:\mathbb{R}_{+}\to M$, the composed map $f\circ\phi:\mathbb{R}_{+}\to N$ is either proper or converges to a point in $N$ for $t\to\infty$.
\end{definition}
}
\begin{definition}[i.e. Definition \ref{Defn: degree}]
Let $M^{n}$ and $N^{n}$ be orientable $n$-manifolds (possibly non-compact). A quasi-proper map $f:M\to N$ is said to have non-zero degree if
	\begin{itemize}\setlength{\itemsep}{1mm}
	\item $S_\infty$ consists of discrete points, where
 $$S_\infty=\bigcap_{K\subset M \text{ compact}} \overline{f(M-K)};$$
	\item the composed map
	$$H_n^{\mathrm{lf}}(M)\xrightarrow{i_*} H_n^{\mathrm{lf}}(f^{-1}(N-S_\infty))\xrightarrow{f_*} H_n^{\mathrm{lf}}(N-S_\infty)$$
	is non-zero, where $H_{\ast}^{\mathrm{lf}}(\cdot)$ are the locally-finite singular homology groups {with $\mathbb{Z}$ coefficients} and $i_{*}$ is the restriction map.
\end{itemize}
\end{definition}

To compute the degree of the map $f$ we need to fix a choice of orientations on $M$ and $N$. Correspondingly we denote the fundamental classes by $[M]\in H_n^{\mathrm{lf}}(M)$ and by $[N]\in H_n^{\mathrm{lf}}(N)$ respectively. Since the set $S_\infty$ consists of discrete points, we have the natural isomorphism
$$
H_n^{\mathrm{lf}}(N-S_\infty)\cong H_n^{\mathrm{lf}}(N)\cong \mathbb Z[N].
$$
Then there is a uniquely determined integer, denoted by $\deg f$ and called the \emph{degree} of $f$, such that we have
$$f_*(i_*([M]))=\deg f\cdot [N]\in H_n^{\mathrm{lf}}( N-S_\infty).$$
We point out the fact that the degree of $f$ is uniquely determined up to sign concerning all possible choices of those orientations on $M$ and $N$. In particular, it makes sense to talk about the absolute value $|\deg f|$ without determining a choice of the orientations.

The key tool to prove Theorem \ref{Thm: partial classification} appears to be the quantitative filling proposition below, which can be derived from our previous arguments.

\begin{proposition}\label{Prop: quantitative filling}
    For $n\in\{3,4,5\}$, let $(M^n,g_M)$ be a complete orientable $n$-manifold with positive scalar curvature and $(N,g_N)$ be a closed $n$-manifold whose universal cover is $(n-2)$-connected. Suppose that
 \[
f:(M,g_{M})\to (N,g_{N})
\]
is a smooth quasi-proper map with non-zero degree. Then there exists a universal constant $L$ such that for any embedded closed $(n-2)$-submanifold $\Sigma_{n-2}$ in the universal cover $\tilde N$, the cycle $|\deg(f)|\cdot\Sigma_{n-2}$ can be realized as the boundary of an $(n-1)$-chain supported in $B_{L}(\Sigma_{n-2})$.
\end{proposition}
\begin{proof}
Given \cite[Lemma 19]{CLL23} we only need to deal with the case when $M$ is non-compact. As in \cite[Lemma 18]{CLL23} we lift $N$ to its universal cover $\tilde N$ and lift the map $f:M\to N$ to a map $\hat f:\widehat M\to \tilde N$ such that $\widehat M$ is a connected cover of $M$ satisfying
\begin{equation}\label{Eq: lifting}
\ker\big(\pi_1(M,x_0)\to \pi_1(N,y_0)\big)=(p_M)_*(\pi_1(\widehat M,\hat x_0)),
\end{equation}
where marked points are taken such that we have the commutative diagram
\begin{equation}\label{Eq: lifting diagram}
\begin{split}
\xymatrix{(\widehat M,\hat x_0)\ar[r]^{\hat f}\ar[d]^{p_M}&(\tilde N,\tilde y_0)\ar[d]^{p_N}\\
(M,x_0)\ar[r]^{f}&(N,y_0).}
\end{split}
\end{equation}

We have to verify the following
\begin{lemma}\label{Lem: lifting}
The map	$\hat f:\widehat M\to \tilde N$ is a smooth quasi-proper map with non-zero degree and we have $|\deg \hat f|\,\big|\,|\deg f|$.
\end{lemma}
\begin{proof}
We first observe that for any $\hat{z}\in\widehat{M}$, the restricted map $\hat f:\mathcal O_{\hat z}\to \mathcal O_{\tilde w}$ is injective, where
\[
\tilde{w} = \hat{f}(\hat{z}), \
z = p_{M}(\hat{z}), \
w = p_{N}(\tilde{w}), \
\mathcal O_{\hat z} = p_M^{-1}(z), \
\mathcal O_{\tilde w} = p_N^{-1}(w).
\]
Otherwise, we can find a curve $\hat\gamma:[0,1] \to \widehat M$ such that $\hat\gamma(0)$ and $\hat\gamma(1)$ are two different points in $\mathcal O_{\hat z}$. Then $p_M\circ \hat\gamma$ induces an element in $\ker\big(\pi_1(M,x_0)\to \pi_1(N,y_0)\big)$ but not in $(p_M)_*(\pi_1(\widehat M,\hat x_0))$, which contradicts \eqref{Eq: lifting}.

Now we show that the set
$$
\tilde S_\infty=\bigcap_{\widehat K\subset \widehat M \text{ compact}} \overline{\hat f(\widehat M- \widehat K)} \subset \tilde N
$$
consists of discrete points. Since $p_N$ is a covering map, $\tilde S_\infty$ consists of discrete points if and only if $p_N(\tilde S_\infty)$ consists of discrete points. It then suffices to prove $p_N(\tilde S_\infty)\subset S_\infty$. For any $\tilde{w}\in\tilde S_\infty$, there exists a sequence $\{\hat{z}_{i}\}\subset\widehat{M}$ such that
\[
\tilde{w} = \lim_{i\to\infty}\hat{f}(\hat{z}_{i}), \ \
\lim_{i\to\infty}d_{g_{\widehat{M}}}(\hat{x}_{0},\hat{z}_{i}) = +\infty,
\]
where $g_{\widehat{M}}$ and $g_{\tilde{N}}$ denote the lifted metrics of $g_{M}$ and $g_{N}$ on $\widehat{M}$ and $\tilde{N}$ respectively. Set $z_{i}=p_{M}(\hat{z}_{i})$. Using the commutative diagram \eqref{Eq: lifting diagram} it is clear that
\[
\lim_{i\to\infty}f(z_{i})
= \lim_{i\to\infty}f(p_{M}(\hat{z}_{i}))
= \lim_{i\to\infty}p_{N}(\hat{f}(\hat{z}_{i}))
= p_{N}(\tilde{w}).
\]
We claim $\lim_{i\to\infty}d_{g_{M}}(x_{0},z_{i})=+\infty$. It then follows that $p_{N}(\tilde{w})\in S_\infty$ as required. To prove the claim, we argue by contradiction. Suppose that the claim is not true. Then after passing to a subsequence, we have $d_{g_{M}}(x_{0},z_{i})\leq L$ for some constant $L$ independent of $i$. Then for each $z_{i}$, there exists $\hat{x}_{i}\in\mathcal{O}_{\hat{x}_{0}}$ such that $d_{g_{\widehat{M}}}(x_{0},z_{i})\leq L$ and so
\[
d_{g_{\tilde{N}}}(\hat{f}(\hat{x}_{i}),\hat{f}(\hat{z}_{i}))
\leq \Lip f|_{B_{L}^{g_{M}}(x_{0})}\cdot d_{g_{\widehat{M}}}(x_{0},z_{i})
\leq \Lip f|_{B_{L}^{g_{M}}(x_{0})}\cdot L.
\]
Set $\tilde{L}=\Lip f|_{B_{L}^{g_{M}}(x_{0})}\cdot L$. Then when $i$ is sufficiently large, we have
\[
d_{g_{\tilde{N}}}(\tilde{w},\hat{f}(\hat{x}_{i}))
\leq d_{g_{\tilde{N}}}(\hat{f}(\hat{x}_{i}),\hat{f}(\hat{z}_{i}))+d_{g_{\tilde{N}}}(\tilde{w},\hat{f}(\hat{z}_{i}))
< \tilde{L}+1.
\]
Since the restricted map $\hat f:\mathcal O_{\hat z}\to \mathcal O_{\tilde w}$ is injective, the set $\mathcal{O}_{\tilde{w}}\cap B_{\tilde{L}+1}^{g_{\tilde{N}}}(\tilde{w})$ is infinite, which is impossible.

Next we want to compute the degree of $\hat f$.
Fix an orientation on $M$ and on $N$ respectively, and then we denote the corresponding fundamental classes by $[M]$ and $[N]$. For any point $x\in M$ we shall let $[M]_x$ denote the image of $[M]$ under the map $H_n^{\mathrm{lf}}(M)\to H_n(M,M-\{x\})$. As a convention, similar notations will be used later without further explanation when there is no confusion caused. As a natural choice, we take the orientation on $\widehat M$ and $\tilde N$ such that the corresponding fundamental classes $[\widehat M]$ and $[\tilde N]$ satisfy
$$
(p_M)_{*}([\widehat M]_{\hat x})=[M]_{p_M(\hat x)} \mbox{ for all }\hat x\in \widehat M
$$
and
$$
(p_N)_{*}([\tilde N]_{\tilde y})=[N]_{p_N(\tilde y)} \mbox{ for all }\tilde y\in \tilde N.
$$

Let us consider the restricted map of $f$ on $f^{-1}(N-S_\infty)$ denoted by
$$
f_{\mathrm{res}}:f^{-1}(N-S_\infty)\to N-S_\infty.
$$
Recall that the map $f:M\to N$ has non-zero degree and so the map $f_{\mathrm{res}}$
is surjective. Based on a change of marked points we may assume that $y_0 $ is a point contained in $ N-S_\infty$ and that it is also a regular value of the map $f$.

Let us divide our discussion into two cases:

\medskip
{\it Case 1. the map $f_*:\pi_1(M,x_0)\to \pi_1(N,y_0)$ is surjective.} Notice that the map $f_{\mathrm{res}}$ is proper, so the preimage $f^{-1}(y_0)$ consists of finitely many (regular) points and we can do labeling
$$
f^{-1}(y_0)=\{x_0,x_1,\ldots,x_l\}.
$$
Then we have
$$
\sum_{i=0}^lf_*([M]_{x_i})=\deg f\cdot [N]_{y_0}.
$$
For each $i$ let us fix a point $\hat x_i\in\widehat M$ satisfying $p_M(\hat x_i)=x_i$.

We have already known that the restricted map $\hat f:\mathcal O_{\hat x_i}\to \mathcal O_{\tilde y_0}$ is injective. We show it is also surjective. To see this we denote $\tilde y_i=\hat f(\hat x_i)$ and for any point $\tilde y_0'\in\mathcal O_{\tilde{y}_0}$ we take a curve $\tilde \gamma:[0,1]\to \tilde N$ connecting $\tilde y_i$ and $\tilde y'_0$. Correspondingly we obtain a closed curve $\gamma:(\mathbb S^1,1)\to (N,y_0)$. Recall that the map $f_*: \pi_1(M,x_0)\to \pi_1(N,y_0)$ is surjective by assumption and the same thing holds with $\pi_1(M,x_0)$ replaced by $\pi_1(M,x_i)$. As a result, we can find a closed curve $\gamma_i:(\mathbb S^1,1)\to (M,x_i)$ such that $f_*([\gamma_i])=[\gamma]$. By lifting we can obtain a curve $\hat\gamma_i:[0,1]\to \widehat M$ with $\hat\gamma_i(0)=\hat x_i$ and $p_M\circ \hat\gamma_i=\gamma_i$. Then it is easy to verify that $\hat\gamma_i(1)$ is a point in $\mathcal O_{\hat x_i}$ with $\hat f$-image $\tilde y_0'$.

Now we can compute
\[
\begin{split}
\deg \hat f\cdot [\tilde N]_{\tilde y_0}&=\sum_{\hat x\in \hat f^{-1}(\tilde y_0)}\hat f_*([\widehat M]_{\hat x})\\
&=\sum_{i=0}^l\sum_{\hat x\in \hat f^{-1}(\tilde y_0)\cap \mathcal O_{\hat x_i}}\hat f_*([\widehat M]_{\hat x}).
\end{split}
\]
The above discussion shows that $\hat{f}^{-1}(\tilde{y}_{0})\cap\mathcal{O}_{\hat{x}_{i}}$ consists of one point. Applying $(p_{N})_{*}$ on both sides, we obtain
$$
\deg \hat f\cdot [N]_{y_0}=\sum_{i=0}^l f_*([M]_{x_i})=\deg f\cdot [N]_{y_0}.
$$
Therefore we obtain $\deg \hat f=\deg f$ and in particular $|\deg\hat f|\,\big |\,|\deg f|$.

\medskip
{\it Case 2. the map $f_*:\pi_1(M,x_0)\to \pi_1(N,y_0)$ is not surjective.} In this case, we consider the covering $\check p_N:(\check N, \check y_0) \to (N,y_0)$ such that
$$
(\check p_N)_*(\pi_1(\check N,\check y_0))=f_*(\pi_1(M,x_0))\subset \pi_1(N,y_0).
$$
In particular, we can lift the map $f$ to a map $\check f:(M,x_0)\to (\check N,\check y_0)$ such that $\check{f}_{*}:\pi_{1}(M,x_{0})\to\pi_{1}(\check{N},\check{y}_{0})$ is surjective. From the discussion in Case 1 we conclude $\deg \check f=\deg \hat f$. Let us label
$$
\check p_N^{-1}(y_0)=\{\check y_j\}_{j\in J}.
$$
Notice that we have the relation
\[
\begin{split}
\deg f\cdot [N]_{y_0}&=\sum_{i=0}^l f_*([M]_{x_i})\\
&=\sum_{j\in J}\sum_{x_i\in \check f^{-1}(\check y_j)}{(\check{p}_N)_*}\check f_*([M]_{x_i})\\
&=|J|\cdot \deg \check f\cdot [N]_{y_0},
\end{split}
\]
which implies $|J|$ is finite and so $|\deg\hat f|\,\big |\,|\deg f|$.
\end{proof}

Now we are ready to prove Proposition \ref{Prop: quantitative filling}. Since we have $|\deg\hat f|\,\big|\,|\deg f|$, it suffices to show that $\deg\hat f\cdot \Sigma_{n-2}$ can be filled in its $L$-neighborhood.

Given any embedded closed $(n-2)$-submanifold $\Sigma_{n-2}$, we can perturb it a little bit in its $1$-neighborhood by an isotopy of $\tilde N$ such that $\Sigma_{n-2}$ avoids all points in $p_N^{-1}(S_\infty)$ and that $\Sigma_{n-2}$ is transversal to the smooth map $\hat f$. Note that $\tilde N$ satisfies $\pi_i(\tilde N)=0$ for $1\leq i\leq n-2$ and so the Hurewicz theorem yields $H_i(\tilde N)=0$ for these $i$. In particular, we can find an $(n-1)$-chain whose boundary is $\Sigma_{n-2}$. As before, through a perturbation we may assume that the support of this $(n-1)$-chain is disjoint from $p_N^{-1}(S_\infty)$. In other words, we can find a smooth bounded region $V$ containing $\Sigma_{n-2}$ satisfying $\overline V\cap p_N^{-1}(S_\infty)=\emptyset$ such that $[\Sigma_{n-2}]=0\in H_{n-2}(V)$. Once again we assume that $\partial V$ is transversal to $\hat f$.
The pullback set $U:=\hat f^{-1}(V)$ is a smooth bounded region in $\widehat M$ due to the properness of the restricted map  $\hat f_{\mathrm{res}}:\hat f^{-1}(N-p_N^{-1}(S_\infty))\to \tilde N-p_N^{-1}(S_\infty)$, and $\mathcal S_{n-2}:=\hat f^{-1}(\Sigma_{n-2})$ is an $(n-2)$-submanifold in $U$ satisfying
$$
\hat f_*([\mathcal S_{n-2}])=\deg \hat f\cdot [\Sigma_{n-2}].
$$
It is also well-known that
$$
[\mathcal S_{n-2}]=[\overline U,\partial\overline U]\frown(\hat f^*[\Sigma_{n-2}]^*)\in H_{n-2}(\overline U),
$$
where $[\Sigma_{n-2}]^*$ is the Poincar\'e dual of $[\Sigma_{n-2}]$ in $H^2(\overline V,\partial \overline V)$. Then we see $[\mathcal S_{n-2}]=0$ in $H_{n-2}(\overline U)$ from the fact $[\Sigma_{n-2}]=0$ in $H_{n-2}(\overline V)$.

The rest of the proof is almost a repetition of those arguments from the previous section except that the slice-and-dice argument will be done on the source manifold $\widehat M$ while the filling argument is done on the target manifold $\tilde N$. So we just sketch the idea.
Since the manifold $\widehat M$ is no longer obtained from non-compact connected sum, let us fix the notations first. Since $S_\infty$ consists of discrete points, we can find small open balls centered at these points disjoint from each other whose radii are all less than one. Their union will be denoted by $\mathcal N_\infty$ and we take $E_\infty=f^{-1}(\mathcal N_\infty)$.
By removing bounded components we may assume that each component of $E_\infty$ is unbounded. {By the quasi-properness of $f$, $E_\infty$ only has finitely many components.} Label the components as $E^1_{\infty},\ldots, E_\infty^s$. Since each end $E_\infty^I$ is mapped to an open ball in $N$, {by the lifting criterion, this map can be lifted to $\tilde N$. By our construction of the cover $\widehat M$, the preimage of $E_\infty^I$ in $\widehat M$ thus decomposes into disjoint components, each of which is mapped diffeomorphically onto $E_\infty^I$.} So we can fix an end $\widehat E_\infty^I$ of $\widehat M$ diffeomorphic to $E_\infty^I$ for each $I=1,\ldots,s$.
For each $\mathfrak{q} \in \pi_1(\widehat M,\hat x_0)/\pi_1(M,x_0)$, we define
$$
\widehat E_{\infty,\mathfrak{q}}=\bigcup_{I=1}^s\widehat E_{\infty,\mathfrak{q}}^I:=\bigcup_{I=1}^s\hat\Psi(\mathfrak q,\widehat E_{\infty}^I),
$$
where $\hat\Psi(\cdot,\cdot)$ denotes the deck transformation corresponding to the covering $p_M:(\widehat M,\hat x_0)\to (M,x_0)$. With the above labeling the same avoidance criteria as Proposition \ref{Prop: closed} and \ref{Prop: compact} can be established based on the same argument as in previous section after setting appropriate avoidance hypersurfaces, and we won't repeat this argument in the following discussion.

Now let us set hypersurfaces $\Lambda_i$, $i=1,2,3,4$, by smoothing the distance function to the compact subset $K:=M-E_{\infty}$ following the same way as in Subsection \ref{Definition of avoidance hypersurfaces} (where we also introduce two positive constants $\delta_1$ and $\delta_2$). Let $M_{\Lambda_4}$ denote the bounded region of $M$ enclosed by $\Lambda_4$, then we can define
$$
\mu_{\mathrm{loss}}=\frac{3}{4}\min\{\underline R,\delta_1,\delta_2\}, \mbox{ where }\underline R=\min_{M_{\Lambda_4}}R(g).
$$
Recall that we have $[\mathcal S_{n-2}]=0\in H_{n-2}(\widehat M)$. From the proof of Lemma \ref{Lem: Mn-1} and \ref{Lem: Mn-2} we conclude that there is a universal constant $L_0=L_0(\mu_{\mathrm{loss}})$ such that
\begin{itemize}\setlength{\itemsep}{1mm}
\item either there is a hypersurface $M_{n-1}$ with $\partial M_{n-1}=\mathcal S_{n-2}$ lying in the $L_0$-neighborhood of $\mathcal S_{n-2}$;
\item or there exists a closed $(n-2)$-submanifold $M_{n-2}$ with stabilized scalar curvature $R_{n-2}\geq R(\hat g)|_{M_{n-2}}-\mu_{\mathrm{loss}}$ such that there is a hypersurface $\Omega_{n-1}$ in the $L_0$-neighborhood of $\mathcal S_{n-2}$ with $\partial\Omega_{n-1}=M_{n-2}\cup\mathcal S_{n-2}$.
\end{itemize}

In the former case, $\hat f_\#M_{n-1}$ is an $(n-1)$-chain contained in the $(\Lip f|_{K}\cdot L_0)$-neighborhood of $\Sigma_{n-2}$ with $\partial(\hat f_\#M_{n-1})=\deg \hat f\cdot \Sigma_{n-2}$, and we are done.

In the latter case, we use the notation $(\cdot)_K$ to represent the intersection of the set $(\cdot)$ and the bounded region $K$. It follows from the same argument as in Subsection \ref{Extrinsic diameter estimate} that we can prove
$$
\diam(M_{n-2,K}\subset (\widehat M,\hat g))\leq L_1\text{ when $n=3$ or $4$},
$$
or based on the avoidance criteria we can break $M_{n-2}$ into small blocks $U_k$ when $n=5$, where the slicing surfaces $\{S_i\}$, the dicing surfaces $\{D_j\}$, the blocks $\{U_k\}$ and the components $\{\mathcal T_l\}$ of the slice-and-dice trace satisfy those properties listed in Proposition \ref{Prop: slice and dice} and Lemma \ref{Lem: combination diameter}. In particular, we have
\begin{equation}\label{Eq: extrinsic diameter again}
\begin{split}
\diam(S_{i,K}&\subset (\widehat M,\hat g))+\diam(D_{j,K}\subset (\widehat M,\hat g))\\
&+\diam(U_{k,K}\subset (\widehat M,\hat g))+
\diam(\mathcal T_{l,K}\subset (\widehat M,\hat g))\leq L_1,
\end{split}
\end{equation}
where $L_1$ denotes a universal constant depending only on $\underline R$, $\delta_1$, $\delta_2$, $\{\Lambda_i\}_{i=1}^4$ and $(M,g)$.

Now we have to complete the filling argument in $(\tilde N,\tilde g_N)$ since we don't have any homology-vanishing information on $\widehat M$. Recall that the manifold $N$ is closed and so $(\tilde N,\tilde g_N)$ satisfies the following quantitative filling property \cite[Proposition 6]{CLL23}: there is a function $F:(0,+\infty)\to (0,+\infty)$ such that if $C$ is a $k$-boundary in $\tilde N$ contained in some geodesic ball $B_{r}^{\ti{g}_N}(\tilde q)$ of $(\ti N,\tilde g_N)$ with $\tilde q$ a point in $\tilde N$, then we can find a $(k+1)$-chain $\Gamma_C$ contained in the geodesic ball $B_{F(r)}^{\ti{g}_N}(\tilde q)$ of $(\tilde N,\tilde g_N)$ with $\partial\Gamma_C=C$.

When $n=3$ or $4$, the chain $\hat f_\#\Omega_{n-1}$ is supported in the $(\Lip f|_{K}\cdot L_0)$-neighborhood of $\Sigma_{n-2}$ and the chain
$$\hat f_\#(\partial\Omega_{n-1})-\deg\hat f\cdot \Sigma_{n-2}$$ is an $(n-2)$-cycle whose support has its extrinsic diameter no greater than $\Lip f|_{K}\cdot L_1+4$. Recall that we have $H_{n-2}(\tilde N)=0$ and so the cycle is actually a boundary. From the quantitative filling property we conclude that the chain $\deg \hat f\cdot \Sigma_{n-2}$ is the boundary of an $(n-1)$-chain lying in the $L$-neighborhood of $\Sigma_{n-2}$ with
$$L:=\Lip f|_{K}\cdot L_0+F(\Lip f|_{K}\cdot L_1+4).$$

When $n=5$ we need to investigate the filling problem for the block image $\hat f_\#U_k$ since we have
$$
\hat f_\#(\partial\Omega_{n-1})-\deg\hat f\cdot \Sigma_{n-2}=\sum_k\hat f_\#U_k.
$$
From \eqref{Eq: extrinsic diameter again} we can derive
\begin{equation*}
\begin{split}
&\diam(\mathrm{supp}(\hat f_\#S_{i})\subset (\tilde N,\tilde g_N))+\diam(\mathrm{supp}(\hat f_\#D_{j})\subset (\tilde N,\tilde g_N))\\
&\qquad+\diam(\mathrm{supp}(\hat f_\#U_{k})\subset (\tilde N,\tilde g_N))+
\diam(\mathrm{supp}(\hat f_\#\mathcal T_{l})\subset (\tilde N,\tilde g_N))\\
&\qquad\qquad\leq \Lip f|_{K}\cdot L_1+4.
\end{split}
\end{equation*}
Through a similar argument as in the proof of Proposition \ref{Prop: filling Mn-2} we can find a universal constant $L_2$ depending only on $\underline R$, $\delta_1$, $\delta_2$, $\{\Lambda_i\}_{i=1}^4$ and $(M,g)$ such that the chain
$$
\sum_k\hat f_\#U_k
$$
can be filled in the $L_2$-neighborhood of its support, and so $\deg\hat f\cdot \Sigma_{n-2}$ can be realized as the boundary of an $(n-1)$-chain supported in its $L$-neighborhood with $L:=\Lip f|_{K}\cdot L_0+L_2$.
\end{proof}

Based on Proposition \ref{Prop: quantitative filling} it is then straightforward to prove Theorem \ref{Thm: partial classification} following the work in \cite{CLL23}.

\begin{proof}[Proof of Theorem \ref{Thm: partial classification}]
Fix an arbitrary smooth metric $g_N$ on $N$ and perturb the map $f$ to be smooth.
{
From Proposition \ref{Prop: quantitative filling}, for any embedded closed $(n-2)$-submanifold $\Sigma_{n-2}$ in the universal cover $\tilde N$, the cycle $|\deg(f)|\cdot\Sigma_{n-2}$ can filled in $B_{L}(\Sigma_{n-2})$.
Although this condition is slightly weaker than the assumption stated in \cite[Proposition 8]{CLL23} (where they assume the filling radius estimate for every $(n-2)$-chain rather than every $|\deg f|$ multiple of an $(n-2)$-chain), as pointed out as the end of \cite{CLL23}, \cite[Proposition 8]{CLL23} actually holds under our weaker assumption. Thus it follows from \cite[Proposition 8]{CLL23} that the universal cover $(\tilde N,\tilde g_N)$ satisfies the following property: for any point $p\in \tilde N$ each connected component of a level set of $\dist(p,\cdot)$ has diameter no greater than $20L$, where $L$ is the universal constant from Proposition \ref{Prop: quantitative filling}.}

Using \cite[Corollary 14]{CLL23} we conclude that the fundamental group $\pi_1(N)$ is virtually free, i.e. there is a free subgroup $G$ of $\pi_1(N)$ with finite index. As a consequence, we can find a finite cover $\check N$ of $N$ with $\pi_1(\check N)= G$. Notice that $\pi_i(\check N)=\pi_i(\tilde N)=0$ for all $2\leq i\leq n-2$. It follows from \cite[Section 2 and 3]{GS09} that $\check N$ is homotopy equivalent to $\mathbb S^n$ or connected sums of $\mathbb S^{n-1}\times \mathbb S^1$.
\end{proof}

In the following, we prove a rigidity result for closed Ricci-flat manifolds, which would yield
Theorem \ref{Thm: mapping rigidity} as a corollary.

\begin{proposition}\label{Prop: mapping rigidity general n}
Let $n \ge 3$. Assume that $(M,g)$ is a closed smooth $n$-manifold with Ricci-flat metric and there exists a non-zero degree map $f: M \to N$ to a closed $n$-manifold $N$ such that no finite cover of $N$ is homotopic to $\mathbb{S}^n$ or $\mathbb{S}^{n-1} \times \mathbb{S}^1$, and that $\pi_i(N) = 0$ holds for all $2 \le i \le n-2$. Then $M$ and $N$ are both aspherical, and the metric on $M$ is flat.
\end{proposition}

\begin{proof}
Because $M$ is Ricci-flat, it follows from {\cite[Theorem 4.5]{FW75}} that there is a finite Riemannian covering $\mathbb{T}^{l} \times K^{n-l} \to M$ with $l \le n$, where $K$ is a simply-connected Ricci-flat closed manifold. Then the composition map $h: \mathbb{T}^l \times K \to M \to N$ also has non-zero degree, and this implies $h_* \pi_1(\mathbb{T}^l \times K)$ has finite index in $\pi_1(N)$. Since $\pi_1(\mathbb{T}^l \times K) = \mathbb{Z}^l$ is free abelian, the image $h_* \pi_1({\mathbb{T}^{l} \times K})$ is also abelian and in particular we can write $h_* \pi_1({\mathbb{T}^{l} \times K}) \cong \mathbb{Z}^{k} \times G$, where $0 \le k \le l$ and $G$ is a finite abelian group. Thus the group $\mathbb{Z}^{k}$ is also a subgroup of $\pi_1(N)$ with finite index, and correspondingly there exists a finite covering $N' \to N$ by a closed $n$-manifold $N'$ such that $\pi_1(N') = \mathbb{Z}^{k}$. Clearly we also have $\pi_i(N') =  \pi_i(N) = 0$ for $2 \le i \le n-2$.
In the following, let us make a detailed discussion depending on the value of $k$.

When $k=0$ the manifold $N'$ is $(n-2)$-connected. From the Hurewicz theorem we can obtain $H_1(N') = \pi_1(N') = 0$, which yields $H^1(N') = 0$ as well. By the Poincar\'e duality we have $H_{n-1}(N') =H^1(N') = 0$. From the Hurewicz theorem we conclude $\pi_{n-1}(N')=H_{n-1}(N')=0$, $\pi_n(N')=H_n(N')=\mathbb Z$ and
\[
H_i(N')=\left\{\begin{array}{cc}
0,& i\neq 0,n;\\
\mathbb Z,&i=0,n.
\end{array}\right.
\]
Notice that an element $\alpha:\mathbb S^n\to N'$ in $\pi_n(N')$ induces isomorphisms between all homology groups of $\mathbb S^n$ and $N'$. It follows from \cite[Corollary 4.33]{Hat02} that $N'$ is homotopic to $\mathbb{S}^n$, which is impossible from our assumption.

When $k=1$ it follows from \cite[Theorem 1.3]{GS09} that $N'$ is homotopic to $\mathbb{S}^{n-1} \times \mathbb{S}^1$, which is again impossible from our assumption.

It remains to deal with the case when $k \ge 2$, where our goal is to show that
\begin{equation}\label{Eq: homotopy vanishing}
    \pi_{n-1}(N')=0.
\end{equation}
Let $\tilde{N}$ be the universal cover of $N'$. Since we have $|\pi_1(N')| =\infty$, the manifold $\tilde{N}$ is non-compact. Notice that we have
\begin{equation}\label{Eq: homotopy homology}
\pi_{n-1}(N') = \pi_{n-1}(\tilde{N}) = H_{n-1}(\tilde{N})
\end{equation}
 by the Hurewicz theorem.
Using homology groups with local coefficients, we have
$$H_{n-1}(\tilde{N}) = H_{n-1}(N'; \mathbb{Z}[\mathbb{Z}^k])$$ by \cite[Example 3H.2]{Hat02}. By the Poincar\'e duality, we further have $$H_{n-1}(N'; \mathbb{Z}[\mathbb{Z}^k]) = H^1(N'; \mathbb{Z}[\mathbb{Z}^k]).$$ Since $\mathbb{T}^{k}$ is an Eilenberg--MacLane space $K(\mathbb{Z}^k,1)$, we have $H^1(N'; \mathbb{Z}[\mathbb{Z}^k]) = H^1(\mathbb{T}^k; \mathbb{Z}[\mathbb{Z}^k])$ by \cite[Lemma 2.2]{GS09}. (Notice that although \cite[Lemma 2.2]{GS09} is only stated for the case when the fundamental group is free, the proof actually holds for any fundamental group.)
By \cite[Proposition 3H.5]{Hat02}, we further have $H^1(\mathbb{T}^k; \mathbb{Z}[\mathbb{Z}^k]) = H^1_c (\mathbb{R}^k; \mathbb{Z})$, where $H_c^1$ denotes the first cohomology group with compact supports. Notice that $H^1_c (\mathbb{R}^k; \mathbb{Z}) = 0$ when $k \ge 2$. Combining these equalities, we obtain our desired result
$$\pi_{n-1}(N') = H^1_c (\mathbb{R}^k; \mathbb{Z}) = 0.$$

Since $\tilde{N}$ is non-compact, we have $H_i(\tilde{N}) = 0$ for all $i \ge n$. Then we can conclude $\pi_i(\tilde{N}) = 0$ for all $i \ge n$ from the Hurewicz theorem. Since we also have $\pi_i(\tilde{N}) = 0$ for $2 \le i \le n-2$,  it follows that $N'$ is aspherical and the same thing holds for $N$.

By the homotopy classification of aspherical spaces \cite[Theorem 1.1]{Luc08}, the manifold $N'$ is homotopy equivalent to $\mathbb{T}^{k}$. By considering the homology group in dimension $n$ we must have $n = k$. Since $k \le l \le n$, this also forces $l = k$, so $K$ is a point and $M$ is then isometrically covered by flat $\mathbb{T}^n$. This shows the flatness of $M$, which further implies that $M$ is aspherical by {\cite[Theorem 4.6]{FW75}}.
\end{proof}

{
\begin{remark}
    The referee gave an alternative proof of \eqref{Eq: homotopy vanishing} as follows. By \eqref{Eq: homotopy homology}, we aim to show $H_{n-1}(\tilde N)=0$ for the universal cover $\tilde N$ of $N'$.
    By the Poincar\'e duality, we have $H_{n-1}(\tilde{N}) = H^1_c (\tilde{N}; \mathbb{Z})$, where $H_c^1$ denotes the first cohomology group with compact supports. By definition, we have $$H^1_c (\tilde{N}) = \underset{K \subset \tilde{N} \text{ compact}}{\varinjlim} H^1(\tilde{N}, \tilde{N} \setminus K),$$ where we take the direct limit over all compact subsets $K$ of $\tilde{N}$. For each compact $K$, we have the long exact sequence of reduced cohomology of pairs
$$\tilde{H}^0(\tilde{N}) \to \tilde{H}^0(\tilde{N}\setminus K) \to H^1(\tilde N,\tilde{N}\setminus K) \to H^1 (\tilde{N}).$$
Using $\tilde{H}^0(\tilde{N}) = 0$ and $H^1 (\tilde{N}) = 0$, we have
$$\tilde{H}^0(\tilde{N}\setminus K) \cong H^1(\tilde N,\tilde{N}\setminus K).$$
Taking direct limit over $K$, we obtain
$$  H^1_c (\tilde{N};\mathbb Z)=\mathbb Z^{e(\tilde{N}) - 1},$$
where $e(\tilde{N})$ denotes the number of ends of $\tilde{N}$.
By a classical result of Hopf \cite{Hop43}, since $N'$ and $\mathbb T^k$ has the same fundamental group $\mathbb Z^k$, $k \ge 2$, their universal covers have the same number of ends. Thus $e(\tilde{N}) = 1$, and we have $H_{n-1}(\tilde{N}) = H^1_c (\tilde{N}; \mathbb{Z})=0$.

\end{remark}
}

Now we are ready to prove Theorem \ref{Thm: mapping rigidity}.
\begin{proof}[Proof of Theorem \ref{Thm: mapping rigidity}]
Suppose that $M$ admits a complete metric $g$ with nonnegative scalar curvature. By a result of Kazdan \cite{Kaz82}, if $(M,g)$ is not Ricci-flat, then there exists a complete metric on $M$ with positive scalar curvature. Our Theorem \ref{Thm: partial classification} then yields that a finite cover of $N$ is homotopy equivalent to $\mathbb{S}^n$ or connected sums of $\mathbb{S}^{n-1} \times \mathbb{S}^1$. Next we assume that no finite cover of $N$ is homotopy equivalent to $\mathbb{S}^n$ or connected sums of $\mathbb{S}^{n-1} \times \mathbb{S}^1$.  It follows from above discussion that $(M,g)$ has to be Ricci-flat. In order to apply Proposition \ref{Prop: mapping rigidity general n} we just need to prove the compactness of $M$.

As a preparation we show that $\pi_1(N)$ cannot be a finite group. Otherwise, there is a finite cover $\check N$ of $N$ which is $(n-2)$-connected. As in the proof of Proposition \ref{Prop: mapping rigidity general n}, $\check N$ is homotopy equivalent to $\mathbb S^n$, which is impossible from our assumption.

Now we are ready to derive a contradiction supposing that $M$ is non-compact.
Recall that there is a quasi-proper map $f:M\to N$ is said to have non-zero degree if
\begin{itemize}\setlength{\itemsep}{1mm}
	\item $S_\infty$ consists of discrete points, where
 $$S_\infty=\bigcap_{K\subset M \text{ compact}} \overline{f(M-K)};$$
	\item the composed map
	$$H_n^{\mathrm{lf}}(M)\xrightarrow{i_*} H_n^{\mathrm{lf}}(f^{-1}(N-S_\infty))\xrightarrow{f_*} H_n^{\mathrm{lf}}(N-S_\infty)$$
	is non-zero, where $H_{\ast}^{\mathrm{lf}}(\cdot)$ are the locally-finite singular homology groups and $i_{*}$ is the restriction map.
\end{itemize}
Since $S_\infty$ consists of discrete points, we can find small open balls centered at these points disjoint from each other whose radii are all less than one. Their union will be denoted by $\mathcal N_\infty$ and we take $E_\infty=f^{-1}(\mathcal N_\infty)$.
Take $E$ to be an unbounded component of $E_\infty$. As in the proof of Lemma \ref{Lem: lifting} (Case 2 therein) we can assume that the induced map $f_*:\pi_1(M,x_0)\to \pi_1(N,y_0)$ is surjective after passing to a finite cover of $N$ (which still has infinite fundamental group). Take the same lifting as in \eqref{Eq: lifting diagram} and use the same notations there. Since we have the isomorphism
$$
\pi_1(\widehat M,\hat x_0)/\pi_1(M,x_0)\cong \pi_1(N,y_0),
$$
the deck transformation group of the covering $(\widehat M,\hat x_0)\to (M,x_0)$ contains infinitely many elements.
Since the end $E$ is mapped to an open ball in $N$, {it is lifted to diffeomorphic copies of itself in the cover $\widehat M$. It follows that $\widehat M$ has infinitely many ends,} which is impossible due to the Cheeger--Gromoll splitting theorem \cite{CG71}.
\end{proof}

\appendix

\section{Construction of $\rho$}\label{construction of rho}

In this section, we give the explicit construction of $\rho$ in the proof of Proposition \ref{Prop: compact}. For positive constants $\mu$ and $T$ with $\mu<T$, let $\eta_{\mu,T}:[0,+\infty)\to[0,T]$ be a one-variable function such that
\begin{itemize}\setlength{\itemsep}{1mm}
\item $\eta_{\mu,T}(t)=t$ in $[0,T-\mu]$;
\item $\eta_{\mu,T}(t)=T$ in $[T+\mu,+\infty)$;
\item $\Lip\eta_{\mu,T}\leq1$.
\end{itemize}
\begin{figure}[htbp]
\centering
\includegraphics[width=\linewidth]{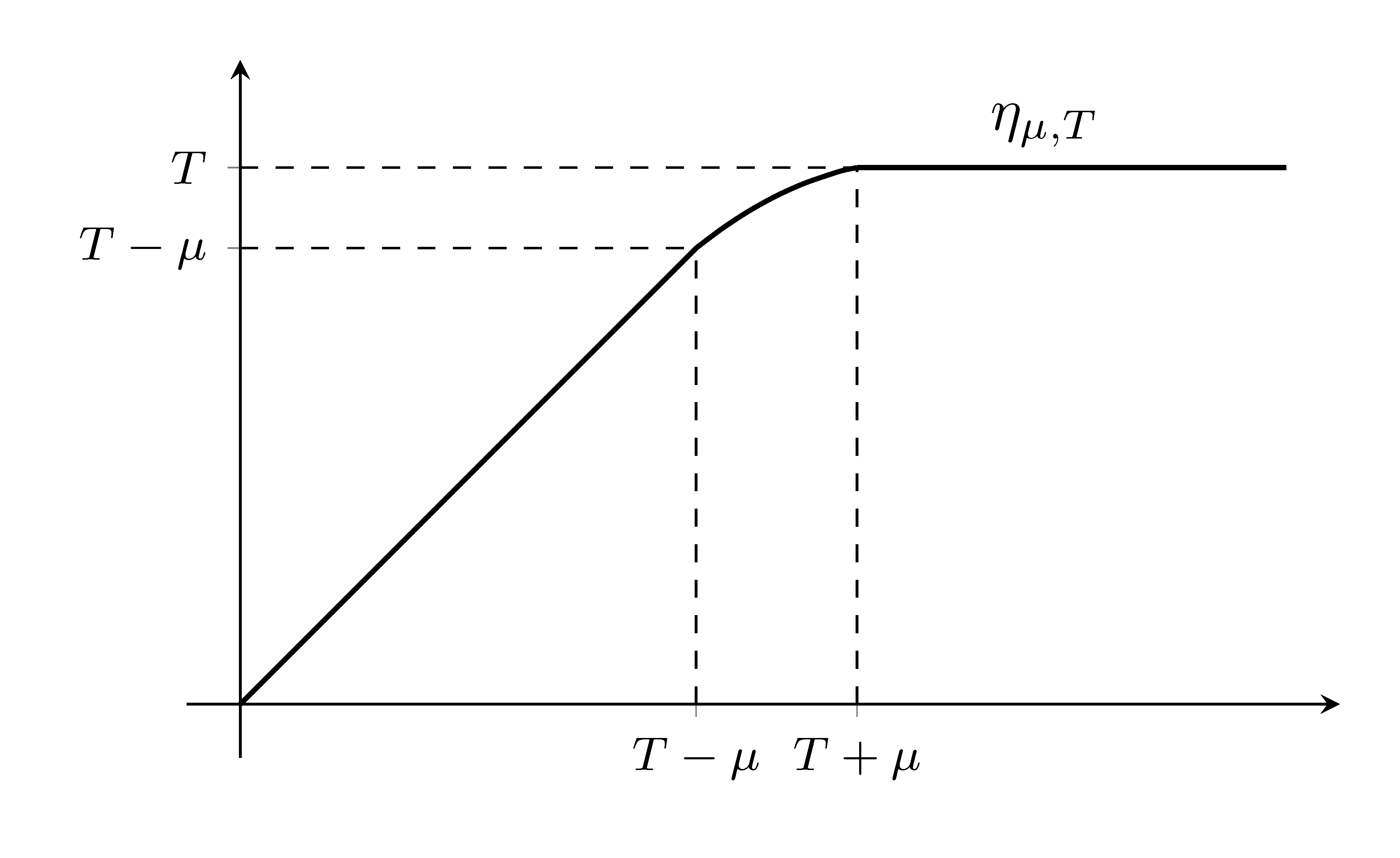}
\caption{The function $\eta_{\mu,T}$}
\label{function_eta}
\end{figure}
The graph of $\eta$ is illustrated in Figure \ref{function_eta}.
For positive constants $\mu_{1}$ and $\mu_{2}$, we define the function $\rho$ on $V$ by
\[
\rho = \eta_{\mu_{2},T_{2}-\mu_{2}}\circ\ti{\rho}|_{V}+\frac{\mu_{2}}{\mu_{1}}\big(\mu_{1}-\eta_{\mu_{2},\mu_{1}}\circ\dist_{V}(\cdot,\partial_{+})\big).
\]
We will choose $\mu_{1}$ and $\mu_{2}$ suitably such that $\rho$ satisfies the desired properties. Observe that $\ti{\rho}|_{\partial_{-}}=0$ and $\ti{\rho}|_{\partial_{+}}>T_{2}$. We obtain $\partial_{-}\cap\partial_{+}=\emptyset$ and so $\dist_{V}(\partial_{-},\partial_{+})>0$. We choose $\mu_{1}<\frac{1}{2}\dist_{V}(\partial_{-},\partial_{+})$ sufficiently small such that $\dist_{V}(\cdot,\partial_{+})$ is smooth and $\ti{\rho}|_{V}>T_{2}$ in the $(2\mu_{1})$-neighborhood of $\partial_{+}$. Recalling $\Lip\ti{\rho}<1$ and $d_{\Lambda_{2}}<T_{2}$, we can choose $\mu_{2}\ll\mu_{1}$ sufficiently small such that
\[
\Lip\rho \leq \Lip\ti{\rho}+\frac{\mu_{2}}{\mu_{1}} < 1, \ \
d_{\Lambda_{2}} < T_{2}-2\mu_{2}.
\]
It then suffices to verify that the above $\rho$ satisfies the first two properties in the proof of Proposition \ref{Prop: compact}.

\medskip
\noindent
{\bf Step 1.} $\rho=0$ on $\partial_-$ and $\rho^{-1}(T_2)=\partial_+$.
\medskip

On $\partial_{-}$, we have $\ti{\rho}=0$ and $\dist_{V}(\cdot,\partial_{+})>2\mu_{1}>\mu_{1}+\mu_{2}$, which implies $\rho=0$. If $p\in\rho^{-1}(T_2)$, then
\[
T_{2} = \rho(p) \leq (T_{2}-\mu_{2})+\frac{\mu_{2}}{\mu_{1}}\big(\mu_{1}-0\big) \leq T_{2}.
\]
This implies the equality holds and then $\dist_{V}(p,\partial_{+})=0$, which is equivalent to $p\in\partial_{+}$. If $p\in\partial_{+}$, then $\ti{\rho}(p)>T_{2}>T_{2}-\mu_{2}+\mu_{2}$ and $\dist_{V}(p,\partial_{+})=0$, which show
\[
\rho(p) = (T_{2}-\mu_{2})+\frac{\mu_{2}}{\mu_{1}}\big(\mu_{1}-0\big) = T_{2}.
\]

\medskip
\noindent
{\bf Step 2.} $\rho^{-1}([0,d_{\Lambda_2}])=(\tilde\rho|_{V})^{-1}([0,d_{\Lambda_2}])$.
\medskip

If $p\in\rho^{-1}([0,d_{\Lambda_2}])$, then
\[
\eta_{\mu_{2},T_{2}-\mu_{2}}\circ\ti{\rho}(p)\leq \rho(p) \leq d_{\Lambda_{2}}.
\]
Using $d_{\Lambda_{2}}<T_{2}-2\mu_{2}$, we obtain $\ti{\rho}(p)=\eta_{\mu_{2},T_{2}-\mu_{2}}\circ\ti{\rho}(p)$, which implies $\ti{\rho}(p)\leq d_{\Lambda_2}$.
If $p\in(\tilde\rho|_{V})^{-1}([0,d_{\Lambda_2}])$, then $\ti{\rho}(p)\leq d_{\Lambda_{2}}<T_{2}$ and so $p$ does not belong to $(2\mu_{1})$-neighborhood of $\partial_{+}$. It follows that
\[
\rho(p)=\eta_{\mu_{2},T_{2}-\mu_{2}}\circ\ti{\rho}(p) = \ti{\rho}(p) \leq d_{\Lambda_{2}}.
\]


\begin{thebibliography}{10}


\bibitem{BLM22} Bamler, R. H., Li, C., Mantoulidis, C. {\em Decomposing $4$-manifolds with positive scalar curvature}, {Adv. Math. {\bf 430} (2023), Paper No. 109231, 17 pp.}

\bibitem{CG71} Cheeger, J., Gromoll, D. {\em The splitting theorem for manifolds of nonnegative Ricci curvature}, J. Differential Geom. {\bf 6} (1971), 119--128.

\bibitem{Chen22} Chen, S. {\em A Generalization of the Geroch Conjecture with Arbitrary Ends}, {Math. Ann. {\bf 389} (2024), no. 1, 489--513.}

\bibitem{CL20} Chodosh, O., Li, C. {\em Generalized soap bubbles and the topology of manifolds with positive scalar curvature}, {Ann. of Math. (2) {\bf 199} (2024), no. 2, 707--740.}

\bibitem{CLL23} Chodosh, O., Li, C., Liokumovich, Y. {\em Classifying sufficiently connected PSC manifolds in $4$ and $5$ dimensions}, Geom. Topol. {\bf 27} (2023), no. 4, 1635--1655.

\bibitem{FcS80} Fischer-Colbrie, D., Schoen, R. {\em The structure of complete stable minimal surfaces in $3$-manifolds of nonnegative scalar curvature}, Comm. Pure Appl. Math. {\bf 33} (1980), no. 2, 199--211.

\bibitem{FW75} Fischer, A., Wolf, J. {\em The structure of compact Ricci-flat Riemannian manifolds}, J. Differential Geom. {\bf 10} (1975): 277--288.

\bibitem{GS09} Gadgil, S., Seshadri, H. {\em On the topology of manifolds with positive
isotropic curvature}, Proc. Amer. Math. Soc. {\bf 137} (2009), no. 5, 1807--1811.

\bibitem{GW79} Greene, R. E., Wu, H. {{\em $C^\infty$ approximation of convex, subharmonic, and plurisubharmonic functions}, Ann. Sci. \'{E}cole Norm. Sup. (4) {\bf 12} (1979), no. 1, 47--84.}

\bibitem{Gro86} Gromov, M. {\em Large Riemannian manifolds}, Curvature and topology of Riemannian manifolds (Katata, 1985), 108--121. Lecture Notes in Math., 1201 Springer-Verlag, Berlin, 1986.

\bibitem{Gro96} Gromov, M. {\em Positive curvature, macroscopic dimension, spectral gaps and higher signatures}, Functional analysis on the eve of the 21st century, Vol. II (New Brunswick, NJ, 1993), 1--213, Progr. Math., 132, Birkhäuser Boston, Boston, MA, 1996.

\bibitem{Gro18} Gromov, M. {\em Metric inequalities with scalar curvature}, Geom. Funct. Anal. {\bf 28} (2018), no. 3, 645--726.

\bibitem{Gro19} Gromov, M. {\em Four Lectures on Scalar Curvature}, preprint, arXiv: 1908.10612.

\bibitem{Gro20} Gromov, M. {\em No metrics with Positive Scalar Curvatures on Aspherical 5-Manifolds}, preprint, arXiv: 2009.05332.

\bibitem{GL80a} Gromov, M., Lawson, H. B. Jr. {\em Spin and scalar curvature in the presence of a fundamental group. I}, Ann. of Math. (2) {\bf 111} (1980), no. 2, 209--230.

\bibitem{GL80b} Gromov, M., Lawson, H. B. Jr. {\em The classification of simply connected manifolds of positive scalar curvature}, Ann. of Math. (2) {\bf 111} (1980), no. 3, 423--434.

\bibitem{GL83} Gromov, M., Lawson, H. B. Jr. {\em Positive scalar curvature and the Dirac operator on complete Riemannian manifolds}, Inst. Hautes \'{E}tudes Sci. Publ. Math. (1983), no. 58, 83--196.

\bibitem{GZ21} Gromov, M., Zhu, J. {\em Area and Gauss-Bonnet inequalities with scalar curvature}, {Comment. Math. Helv. {\bf 99} (2024), no. 2, 355--395.}

\bibitem{Hat02} Hatcher, A. {\em Algebraic Topology}, Cambridge University Press, Cambridge, 2002.

\bibitem{Hop43}  Hopf, H. {{\em Enden offener R\"aume und unendliche diskontinuierliche Grouppen}, Comment. Math. Helv. \textbf{16} (1943), 81--100.}

\bibitem{Kaz82} Kazdan, J. {\em Deformation to positive scalar curvature on complete manifolds}, Math. Ann. \textbf{261} (1982), no. 2, 227--234.

\bibitem{LM89} Lawson, H. B. Jr., Michelsohn, M.-L. {\em Spin geometry}, Princeton Mathematical Series, 38. Princeton University Press, Princeton, NJ, 1989. xii+427 pp.

\bibitem{LUY20} Lesourd, M., Unger, R., Yau, S.-T. {\em Positive Scalar Curvature on Noncompact Manifolds and the Liouville Theorem}, {Comm. Anal. Geom. {\bf 32} (2024), no. 5, 1311--1337.}

\bibitem{Lic63} Lichnerowicz, A. {\em Spineurs harmoniques}, C. R. Acad. Sci. Paris {\bf 257} (1963), 7--9.

\bibitem{Luc08} L\"uck, W. {\em Survey on aspherical manifolds}, European Congress of Mathematics, 53--82, Eur. Math. Soc., Z\"{u}rich, 2010.


\bibitem{Mor18} Morgan, J. {\em Homotopy Theory Lecture Notes} (https://scgp.stonybrook.edu/archives/27538).

\bibitem{RT05} Ratcliffe, J. G., Tschantz, S. T. {\em Some examples of aspherical 4-manifolds that are homology 4-spheres}, Topology {\bf 44} (2005), no. 2, 341--350.

\bibitem{RS97}
Ros, A., Souam, R. {\em On stability of capillary surface in a ball}, Pacific J. Math. {\bf 178} (1997), no. 2, 345--361.

\bibitem{Sch84} Schoen, R. {\em Conformal deformation of a Riemannian metric to constant scalar curvature}, J. Differential Geom. {\bf 20} (1984), no. 2, 479--495.

\bibitem{Sch89} Schoen, R. {\em Variational theory for the total scalar curvature functional for Riemannian metrics and related topics}, Topics in calculus of variations (Montecatini Terme, 1987), 120--154. Lecture Notes in Math., 1365 Springer-Verlag, Berlin, 1989.

\bibitem{SY79c} Schoen, R., Yau, S.-T. {\em Existence of incompressible minimal surfaces
and the topology of three-dimensional manifolds with nonnegative scalar curvature},
Ann. of Math. (2) {\bf 110} (1979), no. 1, 127--142.

\bibitem{SY79a} Schoen, R., Yau, S.-T. {\em On the structure of manifolds with positive scalar curvature}, Manuscripta Math. {\bf 28} (1979), no. 1-3, 159--183.

\bibitem{SY79b} Schoen, R., Yau, S.-T. {\em On the proof of the positive mass conjecture in general relativity}, Comm. Math. Phys. {\bf 65} (1979), no. 1, 45--76.

\bibitem{SY87} Schoen, R., Yau, S.-T. {\em The structure of manifolds with positive scalar curvature}, Directions in partial differential equations (Madison, WI, 1985), 235--242. Publ. Math. Res. Center Univ. Wisconsin, 54 Academic Press, Inc., Boston, MA, 1987.

\bibitem{SY22} Schoen, R., Yau, S.-T. {\em Positive scalar curvature and minimal hypersurface singularities}, Surveys in differential geometry 2019. Differential geometry, Calabi-Yau theory, and general relativity. Part 2, 441--480. Surv. Differ. Geom., 24 International Press, Boston, MA, [2022], \copyright 2022.

\bibitem{Schi98} Schick, T. {\em A counterexample to the (unstable) Gromov-Lawson-Rosenberg conjecture}, Topology {\bf 37} (1998), no. 6, 1165--1168.

\bibitem{Sim83} Simon, L. {\em Lectures on geometric measure theory}, Proc. Centre Math. Anal. Austral. Nat. Univ., 3 Australian National University, Centre for Mathematical Analysis, Canberra, 1983, vii+272 pp.

\bibitem{Ste22} Stern, D. L. {\em Scalar curvature and harmonic maps to $S^{1}$}, J. Differential Geom. {\bf 122} (2022), no. 2, 259--269.

\bibitem{SWZ24} Sun, A., Wang, Z., Zhou, X. {\em Multiplicity one for min-max theory in compact manifolds with boundary and its applications}, Calc. Var. Partial Differential Equaions {\bf 63} (2024), no. 3, Paper No. 70, 52 pp.

\bibitem{Wang19} Wang, J. {\em Contractible 3-manifolds and positive scalar curvature}, Ph.D. thesis, Universit\'e Grenoble Alpes, 2019.

\bibitem{WZ22} Wang, X., Zhang, W. {\em On the generalized Geroch conjecture for complete spin manifolds}, Chinese Ann. Math. Ser. B {\bf 43} (2022), no. 6, 1143--1146.

\bibitem{ZZ20}
Zhou, X., Zhu, J. {\em Existence of hypersurfaces with prescribed mean curvature {I}---generic min-max}, Camb. J. Math. {\bf 8} (2020), no. 2, 311--362.

\bibitem{Zhu21} Zhu, J. {\em Width estimate and doubly warped product}, Trans. Amer. Math. Soc. {\bf 374} (2021), no. 2, 1497--1511.

\bibitem{Zhu23} Zhu, J. {\em The Gauss-Bonnet inequality beyond aspherical conjecture}, Math. Ann. {\bf 386} (2023), no. 3-4, 2321--2347.

\end{thebibliography}
\end{document}